\documentclass[12pt,twoside]{amsart}

\usepackage{amssymb}
\usepackage{graphicx}
\usepackage[pdfauthor={David Wayne Cook II},
            pdftitle={Enumerations deciding the weak Lefschetz property},
            pdfproducer={LaTeX with hyperref},
            pdfcreator={latex->dvips->ps2pdf},
            pdfborder={0 0 0},
            colorlinks=false]{hyperref}  
\usepackage[margin=1in]{geometry}
\newcommand{\HF}{\mathcal{H}}
\newcommand{\SO}{\mathcal{O}}
\newcommand{\NN}{\mathbb{N}}
\newcommand{\PP}{\mathbb{P}}
\newcommand{\PS}{\mathfrak{S}}
\newcommand{\ZZ}{\mathbb{Z}}
\DeclareMathOperator{\CEKZ}{CEKZ}
\DeclareMathOperator{\charf}{char}
\DeclareMathOperator{\Mac}{Mac}
\DeclareMathOperator{\per}{perm}
\DeclareMathOperator{\reg}{reg}
\DeclareMathOperator{\sgn}{sgn}
\DeclareMathOperator{\syz}{syz}
\newcommand{\st}{\; | \;}
\newcommand{\DeltaPlus}{\Delta^{\hspace{-0.333em} +}\hspace{-0.167em}}
\def\urltilda{\kern -.15em\lower .7ex\hbox{\~{}}\kern .04em}
\newcommand{\flfr}[2]{\left\lfloor\frac{#1}{#2}\right\rfloor}
\newcommand{\clfr}[2]{\left\lceil\frac{#1}{#2}\right\rceil}

\numberwithin{figure}{section}
\numberwithin{equation}{section}

\newtheorem{theorem}{Theorem}[section]
\newtheorem{lemma}[theorem]{Lemma}
\newtheorem{proposition}[theorem]{Proposition}
\newtheorem{corollary}[theorem]{Corollary}
\newtheorem{conjecture}[theorem]{Conjecture}

\theoremstyle{definition}
\newtheorem{definition}[theorem]{Definition}
\newtheorem{remark}[theorem]{Remark}
\newtheorem{example}[theorem]{Example}
\newtheorem{question}[theorem]{Question}
\newtheorem*{acknowledgement}{Acknowledgement}

\newcount\HOUR
\newcount\MINUTE
\newcount\HOURSINMINUTES
\newcount\INTVAL
\newcommand{\twodigit}[1]{\INTVAL=#1\relax\ifnum\INTVAL<10 0\fi\the\INTVAL}
\HOUR=\time\divide\HOUR by 60\relax
\HOURSINMINUTES=\HOUR\multiply\HOURSINMINUTES by 60\relax
\MINUTE=\time\advance\MINUTE by -\HOURSINMINUTES\relax
\newcommand\rightnow{
           \twodigit{\the\HOUR}:\twodigit{\the\MINUTE},
           \twodigit{\number\day}.\space
           \ifcase\month\or January\or February\or March\or April\or May\or June\or July\or August\or September\or October\or November\or December\fi
           \space\number\year}

\begin{document}
\noindent {\footnotesize {\em Draft:}\rightnow}

% -- Title Block and Abstract
\title[Enumerations deciding the WLP]{Enumerations deciding the weak Lefschetz property}
\author[D.\ Cook II, U.\ Nagel]{David Cook II${}^{\star}$, Uwe Nagel}
\address{Department of Mathematics, University of Kentucky, 715 Patterson Office Tower, Lexington, KY 40506-0027, USA}
\email{\href{mailto:dcook@ms.uky.edu}{dcook@ms.uky.edu}, \href{mailto:uwe.nagel@uky.edu}{uwe.nagel@uky.edu}}
\thanks{Part of the work for this paper was done while the authors were partially supported by the National Security Agency under Grant Number H98230-09-1-0032.\\
        \indent ${}^{\star}$ Corresponding author.}

\begin{abstract}
    We introduce a natural correspondence between artinian monomial almost complete intersections in three variables and punctured hexagonal
    regions.  We use this correspondence to investigate the algebras for the presence of the weak Lefschetz property.  In particular, we
    relate the field characteristics in which such an algebra fails to have the weak Lefschetz property to the prime divisors of 
    the enumeration of signed lozenge tilings of the associated punctured hexagonal region.  On the one side this allows us to establish
    the weak Lefschetz property in many new cases.  On the other side we can determine some of the prime divisors of the enumerations by
    means of an algebraic argument.

    For numerous classes of punctured hexagonal regions we find closed formulae for the enumerations of signed lozenge tilings, and thus
    the field characteristics in which the associated algebras fail to the have the weak Lefschetz property.  Further, we offer a 
    conjecture for a closed formula for the enumerations of signed lozenge tilings of symmetric punctured hexagonal regions.  These formulae 
    are exploited to lend further evidence to a conjecture by Migliore, Mir\'o-Roig, and the second author that classifies the {\em level}
    artinian monomial almost complete intersections in three variables that have the weak Lefschetz property in characteristic zero.  Moreover,
    the formulae are used to generate families of algebras which never, or always, have the weak Lefschetz property, regardless of field
    characteristic.  Finally, we determine (in one case, depending on the presence of the weak Lefschetz property) the splitting type of the 
    syzygy bundle of an artinian monomial almost complete intersection in three variables, when the characteristic of the base field is zero.  

    Our results convey an intriguing interplay between problems in algebra, combinatorics, and algebraic geometry, which raises new questions 
    and deserves further investigation.
\end{abstract}

\maketitle

\setcounter{tocdepth}{1}
\tableofcontents

% -- Section
\section{Introduction} \label{sec:intro}

The starting point of this paper has been an intriguing conjecture in~\cite{MMN} on the weak Lefschetz property of certain algebras.
Though the presence of this property implies considerable restrictions on invariants of the algebra, many algebras are expected to have
the weak Lefschetz property.  However, establishing this property is often rather difficult.  In this paper we make progress on 
the above conjecture and illustrate the depth of the problem by considering a larger class of algebras and relating the problem
to {\em a priori} seemingly unrelated questions in combinatorics and algebraic geometry.  This builds on the work of many authors (e.g., 
\cite{BK}, \cite{CGJL}, \cite{CN}, \cite{LZ}, and~\cite{MMN}).

Throughout this work we consider in particular the question of how the weak Lefschetz property of a certain $K$-algebra $A$ depends on 
the characteristic of the field $K$.  We begin by relating the algebra $A$ to two square integer matrices, $N$ and $Z$, where the entries
of $N$ are binomial coefficients and $Z$ is a zero-one matrix.  We show that $A$ has the weak Lefschetz property if and only if the determinant
of either of these matrices does not vanish modulo the characteristic of $K$.  Next, we establish that the determinant of $N$ enumerates signed 
lozenge tilings of a punctured hexagonal region and that the determinant of $Z$ enumerates signed perfect matchings of a bipartite graph 
associated to the same punctured hexagonal region.  The relation to the weak Lefschetz property implies that both determinants have the same 
prime divisors; in fact, we show that their absolute values are the same by using combinatorial arguments.  Finally, we show that in certain
cases deciding the presence of the weak Lefschetz property is equivalent to determining the splitting type of some semistable rank three vector
bundles on the projective plane.

We now describe the contents of this paper in more detail.
Let $R = K[x_1, \ldots, x_n]$ be the standard graded $n$-variate polynomial ring over the infinite field $K$, and let $A$ be a standard
graded $K$-algebra over $R$.  We say $A$ is {\em artinian} if $A$ is finite dimensional as a vector space over $K$.  Further, an artinian
$K$-algebra $A$ is said to have the {\em weak Lefschetz property} if there exists a linear form $\ell \in [A]_1$ such that, for all 
integers $d$, the multiplication map $\times \ell: [A]_d \rightarrow [A]_{d+1}$ has maximal rank, that is, the map is injective or
surjective.  Such a linear form is called a {\em Lefschetz element} of $A$.

The weak Lefschetz property has been studied extensively for many reasons, especially for the relation to the Hilbert function (see, 
e.g., \cite{BMMNZ}, \cite{HMNW}, \cite{MZ}, and~\cite{ZZ}).   A convenient way to encode the Hilbert function of an artinian $K$-algebra
$A$ is the {\em $h$-vector}, a finite sequence $h(A) = (h_0, \ldots, h_e)$ of positive integers $h_i = \dim_K{[A]_i}$.  Using this
notation, one immediate consequence (\cite[Remark~3.3]{HMNW}) of $A$ having the weak Lefschetz property is that the $h$-vector of $A$ is
{\em strictly unimodal}.  Further, the positive part of the first difference of $h(A)$ is $h(A/\ell A)$, where $\ell$ is any Lefschetz 
element of $A$.

The weak Lefschetz property is known to be subtle to both deformations (see, e.g., \cite{CN-2010}, \cite{Mi}, and~\cite{MMN}) but also to field 
characteristic.  The latter, considering the weak Lefschetz property in positive characteristic, is an exciting and active direction of research.
Migliore, Mir\'o-Roig, and the second author~\cite{MMN}, as well as Zanello and Zylinski~\cite{ZZ}, began explorations into the connection
between the weak Lefschetz property and positive characteristic, and also posed several interesting questions.  

In~\cite{CN}, the authors found a connection between certain families of level artinian monomial almost complete intersections and lozenge 
tilings of hexagons; independently, Li and Zanello~\cite{LZ} found a similar connection for artinian monomial complete intersections (see also
Corollary~\ref{cor:ci}).  However, both were without combinatorial bijection until one was found by Chen, Guo, Jin, and Liu~\cite{CGJL}; Boyle, 
Migliore, and Zanello~\cite{BMZ} have pushed this connection further.  Brenner and Kaid~\cite{BK-2010} also consider artinian
monomial complete intersections in three variables with generators all of the same degree.  We also note that in their study of pure $O$-sequences 
Boij, Migliore, Mir\'o-Roig, the second author, and Zanello~\cite{BMMNZ} have explored the relation between the weak Lefschetz property and pure
$O$-sequences.

In this paper we extend the connection found by Chen, Guo, Jin, and Liu to a connection between artinian monomial almost complete intersections in
three variables and lozenge tilings of more general regions that we call punctured hexagons.  In Section~\ref{sec:toolchain} we gather a few 
useful tools for dealing with the weak Lefschetz property.  In Section~\ref{sec:aci} we introduce the algebras we are interested in:  artinian 
monomial almost complete intersections in three variables.  If the syzygy bundle is not semistable, then the algebra has the weak Lefschetz property
in characteristic zero (\cite{BK}).  Thus we focus on the algebras that have semistable syzygy bundles, which we classify numerically
(Proposition~\ref{pro:semistable}).  Then we prove that such an algebra has the weak Lefschetz property if and only if a particular map between
the peak homogeneous components of the algebra is a bijection (Corollary~\ref{cor:one-map}).  Using this, we show that to each of the studied 
algebras $A$ we can associate a zero-one matrix $Z_A$ such that $A$ has the weak Lefschetz property in positive characteristic $p$ if and only 
if $p$ is not a prime divisor of the determinant of $Z_A$ (Proposition~\ref{pro:wlp-zero-one}).  We also describe a matrix $N_A$ with binomial 
entries that also has the analogous property (Proposition~\ref{pro:wlp-binom}).  Moreover, we demonstrate that a rather simple algebraic argument
can be used to determine some of the prime divisors of the determinants for both $Z_A$ and $N_A$ (Proposition~\ref{pro:wlp-p}).

In Section~\ref{sec:ph} we organise the monomials generating the peak homogeneous components of such an algebra in a plane.  It turns out that the 
monomials fill a punctured hexagon (Theorem~\ref{thm:interlace-amaci}).  Using the well-known bijection between lozenge tilings and non-intersecting
lattice paths, and the Lindstr\"om-Gessel-Viennot theorem (\cite{GV}, \cite{GV-1989}, \cite{Li}) on non-intersecting lattice paths, we show
that the determinant of the binomial matrix $N_A$ is the enumeration of the signed lozenge tilings of the punctured hexagon, up to sign (Theorem~\ref{thm:nilp-matrix}).
Furthermore, using another well-known bijection between lozenge tilings and perfect matchings (see, e.g., \cite{Ku}), we argue that the determinant of
the zero-one matrix $Z_A$ is an enumeration of the signed perfect matchings of the associated bipartite graph, up to sign (Theorem~\ref{thm:bip-matrix}).

In Section~\ref{sec:signs} we use the aforementioned connections to show that the determinant of the zero-one matrix $Z_A$ and the binomial matrix 
$N_A$ are the same, up to sign (Theorem~\ref{thm:det-Z-N}).  Moreover, in a special case of Kasteleyn's theorem~\cite{Ka} about enumerating perfect
matchings,  when the puncture has an even side-length, then the determinant and the permanent of $Z_A$ are also the same, up to sign (Corollary~\ref{cor:det-Z-per-Z}).

In Section~\ref{sec:det} we prove first that the determinant of $N_A$ is non-zero when the puncture is of even side-length (Theorem~\ref{thm:M-even}),
thus establishing the weak Lefschetz property in many new cases.  We then find closed formulae for the determinants when the puncture is trivial
(Proposition~\ref{pro:M-zero}), when any one side of the hexagonal region has length zero (Proposition~\ref{pro:C-zero}), when a vertex of the puncture
touches one of the sides of the region (Proposition~\ref{pro:C-maximal}), and when a side of the puncture touches one of the sides of the region
(Proposition~\ref{pro:gamma-zero}).  We close with a complete description of when the region is symmetric.  In particular, we show that when certain
parity conditions hold the determinant is zero (Proposition~\ref{pro:symmetry-zero}) and we provide a conjecture for a closed formula of the determinant
when the same parity conditions fail (Conjecture~\ref{con:symmetry}).

In Section~\ref{sec:central} we explore two different ways to centralise the puncture.  We call the puncture {\em axis-central} when it is
central along each of the three axes, independently.  Using very involved computations, Ciucu, Eisenk\"olbl, Krattenthaler, and Zare~\cite{CEKZ} found
closed formulae for the enumerations and signed enumerations of regions with an axis-central puncture; therein axis-central is called simply ``central''.
We use these closed formulae to describe the permanents of the zero-one matrices $Z_A$ (Corollary~\ref{cor:axis-central-perm-Z}) and the determinants 
of both matrices (Corollary~\ref{cor:axis-central-det-Z}), $Z_A$ and $N_A$, when the puncture is axis-central.  We call the puncture 
{\em gravity-central} when its vertices are equidistant from the sides of the containing hexagon; this condition is equivalent to the associated algebra
being level, that is, its socle is concentrated in one degree.  Using this observation we provide further evidence for a conjecture by Migliore,
Mir\'o-Roig, and the second author~\cite{MMN} about the presence of the weak Lefschetz property for level artinian monomial almost complete intersections
in characteristic zero (Proposition~\ref{pro:level-wlp}).

In Section~\ref{sec:interesting} we describe a method, for any positive integer $n$, to generate a subfamily of algebras whose associated matrices have 
determinant $n$ (Proposition~\ref{pro:det-n}).  From this we generate a subfamily of algebras which {\em always} have the weak Lefschetz property, regardless
of the field characteristic (Corollary~\ref{cor:C-zero-unique});  we also describe a different subfamily of algebras which always have the weak Lefschetz property
(Proposition~\ref{pro:gamma-zero-unique}).  Moreover, we describe the unique algebras which retain certain properties yet have minimal multiplicity
(Example~\ref{exa:minimal-multiplicity}).

In Section~\ref{sec:splitting-type} we explicitly determine (in one case, depending on the presence of the weak Lefschetz property) the splitting type of all 
artinian monomial almost complete intersections.  In particular, we consider separately the cases when the syzygy bundle is non-semistable
(Proposition~\ref{pro:st-nss}) and semistable (Propositions~\ref{pro:st-nmod3} and~\ref{pro:st-mod3}).  Moreover, in the case of ideals associated to punctured
hexagons, we relate the weak Lefschetz property to a number of other problems in algebra, combinatorics, and algebraic geometry (Theorem~\ref{thm:equiv}).

Finally, in Appendix~\ref{sec:hyper-calculus} we provide a technique, a ``picture-calculus'', for working with hyperfactorials, a basic unit
for the aforementioned closed formulae.  We demonstrate that several nice polynomials can be written as ratios of products of hyperfactorials
(Proposition~\ref{pro:hyper-f} and Corollary~\ref{cor:hyper-eo-f}).  Further, this shows that MacMahon's formula for the number
of lozenge tilings of a (non-punctured) hexagon is a polynomial in one of the side-lengths when the other two are fixed (Corollary~\ref{cor:mac-poly}).

% -- Section
\section{Compiling the tool-chain} \label{sec:toolchain}

Let $R = K[x_1, \ldots, x_n]$ be the standard graded $n$-variate polynomial ring over the infinite field $K$, and let $A$ be
an artinian standard graded $K$-algebra over $R$.  Then the minimal free resolution of $A$ ends with the free module 
$\bigoplus_{i=1}^{m} R(-t_i)^{r_i},$ where $0 < t_1 < \cdots < t_m$ and $0 < r_i$ for all $i$.  In this case, $A$ is 
called {\em level} if $m = 1$, the {\em socle degrees} of $A$ are $t_i - n$, for all $i$, and the {\em socle type} of 
$A$ is the sum $\sum_{i=1}^{m} t_i$.

We recall that once multiplication by a general linear form is surjective, then it remains surjective.
\begin{proposition}{\cite[Proposition~2.1(a)]{MMN}} \label{pro:surj}
    Let $A = R/I$ be an artinian standard graded $K$-algebra, and let $\ell$ be a general linear form.  If the map 
    $\times \ell: [A]_{d} \rightarrow [A]_{d+1}$ is surjective, then $\times \ell: [A]_{d+1} \rightarrow [A]_{d+2}$ is surjective.
\end{proposition}

This generalises to modules generated in degrees that are sufficiently small.
\begin{lemma} \label{lem:mod-surj}
    Let $M$ be an $R$-module generated in degrees bounded by $e$, and let $\ell$ be a general linear form.  If the map 
    $\times\ell: [M]_d \rightarrow [M]_{d+1}$ is surjective, and $d \geq e$, then the map $\times\ell: [M]_{d+1} \rightarrow [M]_{d+2}$ is surjective.
\end{lemma}
\begin{proof}
    Consider the sequence
    \[
        [M]_d \stackrel{\times\ell}{\longrightarrow} [M]_{d+1} \rightarrow [M/\ell M]_{d+1} \rightarrow 0.
    \]
    Notice the first map is surjective if and only if $[M/\ell M]_{d+1} = 0.$  By assumption the map is surjective, so
    $[M/\ell M]_{d+1} = 0$.  Hence $[M/\ell M]_{d+2}$ is zero unless there is a generator of $M$ with degree beyond $d$.  However,
    the assumption is that no generators exist with degree beyond $d.$
\end{proof}

From this we get a result analogous to~\cite[Proposition~2.1(b)]{MMN} for non-level algebras.
\begin{proposition} \label{pro:inj}
    Let $A = R/I$ be an artinian standard graded $K$-algebra, and let $\ell$ be a general linear form.  If the map 
    $\times \ell: [A]_{d-1} \rightarrow [A]_{d}$ is injective, and $d$ is no greater than the smallest socle degree of $A$, then
    $\times \ell: [A]_{d-2} \rightarrow [A]_{d-1}$ is injective.
\end{proposition}
\begin{proof}
    The $K$-dual of $A$, $M$, is a shift of the canonical module of $A$ and is generated in degrees that are a linear
    shift of the socle degrees of $A$.  Consider now the map $\times \ell: [M]_i \rightarrow [M]_{i+1}$.  Using Lemma~\ref{lem:mod-surj}
    we see that once $i$ is at least as large as the largest degree in which $M$ is generated, and the map is surjective, then the map
    is surjective thereafter.  The result then follows by duality.
\end{proof}

Further recall that a monomial algebra has the weak Lefschetz property exactly when the sum of the variables is a Lefschetz element.
\begin{proposition}{\cite[Proposition~2.2]{MMN}} \label{pro:mono}
    Let $A = R/I$ be an artinian standard graded $K$-algebra with $I$ generated by monomials.  Then $A$ has the weak 
    Lefschetz property if and only if $x_1 + \cdots + x_n$ is a Lefschetz element of $A$.
\end{proposition}

Hence, the weak Lefschetz property can be decided for monomial ideals, in a small number of cases, by simple invariants.  The following lemma
is a generalisation of~\cite[Proposition~3.7]{LZ}.
\begin{lemma} \label{lem:wlp-p}
    Let $A = R/I$ be an artinian standard graded $K$-algebra with $I$ generated by monomials.  Suppose that $a$ is the least positive integer
    such that $x_i^a \in I$, for $1 \leq i \leq n$, and suppose that the Hilbert function of $R/I$ weakly increases to degree $s+1$.
    Then, for any positive prime $p$ such that $a \leq p^m \leq s+1$ for some positive integer $m$, $A$ fails to have the weak Lefschetz property
    in characteristic $p$.
\end{lemma}
\begin{proof}
    By Proposition~\ref{pro:mono}, we need only consider $\ell = x_1 + \cdots + x_n$.  Suppose the characteristic of $K$ is $p$, then by the
    Frobenius endomorphism $\ell \cdot \ell^{p^m-1} = \ell^{p^m} = x_1^{p^m} + \cdots + x_n^{p^m}$.  Moreover, as $a \leq p^m$, then $\ell^{p^m} = 0$
    in $A$ while $\ell \neq 0$ in $A$.  Hence $\times\ell^{p^m-1}: [A]_{1} \rightarrow [A]_{p^m}$ is not injective and thus $A$ does not
    have the weak Lefschetz property.
\end{proof}

Further, for monomial ideals, if the weak Lefschetz property holds in characteristic zero, then it holds for almost every characteristic.
\begin{lemma} \label{lem:wlp-ae}
    Let $I$ be an artinian monomial ideal in $R$.  If $R/I$ has the weak Lefschetz property when $\charf{K} = 0$, 
    then $R/I$ has the weak Lefschetz property for $\charf{K}$ sufficiently large.
\end{lemma}
\begin{proof}
    By Proposition~\ref{pro:mono}, we need only consider $\ell = x_1 + \cdots + x_n$.  As $R/I$ is artinian, then there are finitely many
    maps that need to be checked for the maximal rank property, and this in turn implies finitely many determinants that need to be computed.
    Further, because of the form of $\ell$, the matrices in question are all zero-one matrices.  Thus, the 
    determinants to be checked are integers.  Simply let $p$ be the smallest prime larger than all prime divisors of the determinants, 
    then the determinants are all non-zero modulo $p$ and so $R/I$ has the weak Lefschetz property if $\charf{K} \geq p$.
\end{proof}

And (pseudo-)conversely, again for monomial ideals, if the weak Lefschetz property holds in some positive characteristic, then it holds
for characteristic zero.
\begin{lemma} \label{lem:wlp-p-zero}
    Let $I$ be an artinian monomial ideal in $R$.  If $R/I$ has the weak Lefschetz property when $\charf{K} = p > 0$,
    then $R/I$ has the weak Lefschetz property for $\charf{K} = 0$.
\end{lemma}
\begin{proof}
    The proof is the same as that of Lemma~\ref{lem:wlp-ae} except we notice that if an integer $d$ is non-zero modulo a prime $p$, then $d$
    is not zero.
\end{proof}

Last, we note that any artinian ideal in two variables has the weak Lefschetz property.  This was proven for characteristic zero in~\cite[Proposition~4.4]{HMNW}
and then for arbitrary characteristic in~\cite[Corollary~7]{MZ}, though it was not specifically stated therein, as noted in~\cite[Remark~2.6]{LZ}.  We provide
a brief, direct proof of this fact to illustrate the weak Lefschetz property.  Unfortunately, the simplicity of this proof fails in three variables, even
for monomial ideals.
\begin{proposition} \label{pro:2-wlp}
    Let $R = K[x,y]$, where $K$ is an infinite field with {\em arbitrary} characteristic.  Every artinian algebra in $R$ has the weak Lefschetz property.
\end{proposition}
\begin{proof}
    Assume $I = (g_1, \ldots, g_t)$ and the given generators are minimal.  Let $s = \min\{\deg{g_i} \st 1 \leq i \leq t\}$.  
    Then $h(R/I)$, the $h$-vector of $R/I$, strictly increases by one from $h_0$ to $h_{s-1}$ and $h_{s-1} \geq h_s$, thus the positive part of
    the first difference of $h(R/I)$, $\DeltaPlus h(R/I)$, is $s$ ones.  Moreover, for a general linear form $\ell \in R$, $R/(I,\ell) \cong K[x]/J$ 
    where $J = (x^s)$ so $h(R/(I, \ell))$ is $s$ ones, that is, $\DeltaPlus h(R/I) = h(R/(I,\ell))$.  Hence $R/I$ has the weak Lefschetz property 
    with Lefschetz element $\ell$.
\end{proof}

% -- Section
\section{Almost complete intersections} \label{sec:aci}

Here we restrict to artinian monomial almost complete intersections in three variables.  These are the ideals discussed in \cite[Corollary~7.3]{BK}
and~\cite[Section~6]{MMN}.  

Let $K$ be an infinite field, and consider the ideal
\[
    I_{a,b,c,\alpha,\beta,\gamma} = (x^a, y^b, z^c, x^\alpha y^\beta z^\gamma)
\]
in $R = K[x,y,z]$, where $0 \leq \alpha < a, 0 \leq \beta < b,$ and $0 \leq \gamma < c$.  If $\alpha = \beta = \gamma = 0$, then we define
$I_{a,b,c,0,0,0}$ to be $(x^a, y^b, z^c)$ which is a complete intersection and is studied extensively in~\cite{LZ} and \cite{CGJL}.  Assume at
most one of $\alpha, \beta,$ and $\gamma$ is zero.
\begin{proposition}{\cite[Proposition~6.1]{MMN}} \label{pro:amaci-props}
    Let $I = I_{a,b,c,\alpha,\beta,\gamma}$ be defined as above.  Assume, without loss of generality, that $0 \leq \alpha \leq \beta \leq \gamma$.
    \begin{enumerate}
        \item If $\alpha = 0$, then $R/I$ has socle type $2$ with socle degrees $a+\beta+c-3$ and $a+b+\gamma -3$; thus $R/I$ is level if
            and only if $b - \beta = c - \gamma$.
        \item If $\alpha > 0$, then $R/I$ has socle type $3$ with socle degrees $\alpha + b + c -3$, $a+\beta+c-3$, and $a+b+\gamma -3$; thus
            $R/I$ is level if and only if $a - \alpha = b - \beta = c - \gamma$.
        \item Moreover, the minimal free resolution of $R/I$ has the form
            \begin{equation} \label{eqn:resolution}
                0
                \rightarrow
                \begin{array}{c} R(-a-b-\gamma) \\ \oplus \\ R(-a-\beta-c) \\ \oplus \\ R^n(-\alpha-b-c) \end{array}
                \rightarrow
                \begin{array}{c} R(-a-\beta-\gamma) \\ \oplus \\ R(-\alpha-b-\gamma) \\ \oplus \\  R(-\alpha-\beta-c) \\ \oplus \\
                                 R(-a-b) \\ \oplus \\ R(-a-c) \\ \oplus \\ R^n(-b-c) \end{array}
                \rightarrow
                \begin{array}{c} R(-\alpha-\beta-\gamma) \\ \oplus \\ R(-a) \\ \oplus \\ R(-b) \\ \oplus \\ R(-c) \end{array}
                \rightarrow
                R
                \rightarrow
                R/I
                \rightarrow
                0
            \end{equation}
            where $n = 1$ if $\alpha > 0$ and $n = 0$ if $\alpha = 0$.
    \end{enumerate}
\end{proposition}

Moreover, we see that in characteristic zero the weak Lefschetz property follows for certain choices of the parameters.
\begin{proposition}{\cite[Theorem~6.2]{MMN}} \label{pro:amaci-not-3}
    Let $K$ be an algebraically closed field of characteristic zero.  Then $R/I_{a,b,c,\alpha,\beta,\gamma}$ has the weak Lefschetz
    property if $a + b + c + \alpha + \beta + \gamma \not\equiv 0 \pmod{3}$.
\end{proposition}~

\subsection{Semi-stability}~

The syzygy module $\syz{I}$ of $I = I_{a,b,c,\alpha,\beta,\gamma}$ fits into the exact sequence
\begin{equation*}
        0
    \longrightarrow
        \syz{I}
    \longrightarrow 
        R(-\alpha-\beta-\gamma) \oplus R(-a) \oplus R(-b) \oplus R(-c)
    \longrightarrow 
        I_{a,b,c,\alpha,\beta,\gamma}
    \longrightarrow 
        0.
\end{equation*}
The sheafification $\widetilde{\syz{I}}$ is a rank 3 bundle on $\PP^2$, and it is called the {\em syzygy bundle} of $I$.  Recall that a vector bundle $E$ 
on projective space is said to be {\em semistable} if, for every coherent subsheaf $F \subset E$, the following inequality holds:
\[
    \frac{c_1(F)}{rk(F)} \leq \frac{c_1(E)}{rk(E)}.
\]

We analyse when $I_{a,b,c,\alpha,\beta,\gamma}$ has a semistable syzygy bundle.  (Note, the slightly awkward definition of $s$ in the following
is kept for consistency with~\cite[Section~7]{MMN}, the starting point of this work.)
\begin{proposition} \label{pro:semistable}
    Let $K$ be an algebraically closed field of characteristic zero.  Further, let $I = I_{a,b,c,\alpha,\beta,\gamma}$, and define
    the following rational numbers
    \begin{equation*}
        \begin{split}
            s :=& \frac{1}{3}(a + b + c + \alpha + \beta + \gamma)-2, \\
            A :=& s+2-a, \\
            B :=& s+2-b, \\ 
            C :=& s+2-c, \mbox{ and}\\
            M :=& s + 2 - (\alpha + \beta + \gamma).
        \end{split}
    \end{equation*}
    Then $I$ has a semistable syzygy bundle if and only if the following conditions all hold:
    \begin{enumerate}
        \item $0 \leq M$,
        \item $0 \leq A \leq \beta  + \gamma$, 
        \item $0 \leq B \leq \alpha + \gamma$, and
        \item $0 \leq C \leq \alpha + \beta$.
    \end{enumerate}
\end{proposition}
\begin{proof}
    Using~\cite[Corollary~7.3]{Br} we have that $I$ has a semistable syzygy bundle if and only if
    \begin{enumerate}
        \item[(a)] $\max\{a, b, c, \alpha + \beta + \gamma\} \leq s+2$,
        \item[(b)] $\min\{\alpha + \beta + c, \alpha + b + \gamma, a + \beta + \gamma\} \geq s+2$, and
        \item[(c)] $\min\{a+b, a+c, b+c\} \geq s+2$.
    \end{enumerate}

    Notice that condition (a) is equivalent to $A, B, C,$ and $M$ being non-negative.  Moreover, condition
    (b) is equivalent to the upper bounds on $A, B,$ and $C$.  We claim that condition (c) follows directly
    from condition (a).

    Indeed, by condition (a) we have that $C+M \geq 0$ and so $A+B+C+M = s+2 \geq A+B = 2(s+2) - a - b$, thus $a + b \geq s+2$.
    Similarly, we have $a + c \geq s+2$ and $b+c \geq s+2$.  Thus condition (c) holds if condition (a) holds.
\end{proof}

This gives further conditions on the parameters that force the weak Lefschetz property in characteristic zero (see~\cite[Theorem~3.3]{BK}).  
This extends~\cite[Lemma~6.7]{MMN}.
\begin{corollary} \label{cor:wlp-ss}
    Let $K$ be an algebraically closed field of characteristic zero, and let $I = I_{a,b,c,\alpha,\beta,\gamma}$.  If any of the conditions
    (i)-(iv) in Proposition~\ref{pro:semistable} fail, then $R/I$ has the weak Lefschetz property.
\end{corollary}

The above definitions of $s, A, B, C,$ and $M$ are not without purpose.  Before going further, we make a few comments
about the given parameters.
\begin{remark} \label{rem:sabcm}
    Suppose $s, A, B, C,$ and $M$ are defined as in Proposition~\ref{pro:semistable}.  Then clearly $s$ is an integer if
    and only if $a + b + c + \alpha + \beta + \gamma \equiv 0 \pmod{3}$; if $s$ is an integer, then so are $A, B, C,$ and $M$.
    Further, $A+B+C+M = s+2$ and $A+B+C=\alpha + \beta + \gamma$.
\end{remark}~

\subsection{Associated matrices}~
Given the minimal free resolution of $R/I$ (see (\ref{eqn:resolution})), we can easily compute the $h$-vector of $R/I$ as a weighted sum 
of binomial coefficients dependent only on the parameters $a,b,c,\alpha, \beta,$ and $\gamma$.  

We say $h(A)$ has {\em twin peaks} if there exists an integer $s$ such that $h_s = h_{s+1}$.  When $I_{a,b,c,\alpha,\beta,\gamma}$ has parameters
as in Proposition~\ref{pro:semistable} and $s$ is an integer, then the algebras $R/I_{a,b,c,\alpha,\beta,\gamma}$ always have twin peaks and
the peaks are bounded by the socle degrees.  This extends the results in~\cite[Lemma~7.1]{MMN} wherein the level algebras
$R/I_{a,b,c,\alpha,\beta,\gamma}$ with twin peaks are identified.
\begin{lemma} \label{lem:twin-peaks}
    Assume the parameters of $I = I_{a,b,c,\alpha,\beta,\gamma}$ satisfy the conditions in Proposition~\ref{pro:semistable} and suppose 
    $a + b + c + \alpha + \beta + \gamma \equiv 0 \pmod{3}$.  Then $R/I$ has twin peaks in degrees $s$ and $s+1$.  
    Moreover, $s+1$ is bounded above by the socle degrees of $R/I$.
\end{lemma}
\begin{proof}
    The upper bounds on $A, B,$ and $C$ are exactly those required to force the ultimate and penultimate terms in the minimal free resolution
    of $R/I$, given in Proposition~\ref{pro:amaci-props}(iii), to not contribute to the computation of the $h$-vector for degrees up to $s+1$.
    Moreover, as $A,B,C,$ and $M$ are non-negative, and using $\binom{n+1}{2} - \binom{n}{2} = n$ for $n \geq 0$, then
    \begin{equation*}
        \begin{split}
            h_{s+1} - h_s = & \left(\binom{s+3}{2} - \binom{A+1}{2} - \binom{B+1}{2} - \binom{C+1}{2} - \binom{M+1}{2}\right) \\
                            &  - \left(\binom{s+2}{2} - \binom{A}{2} - \binom{B}{2} - \binom{C}{2} - \binom{M}{2}\right) \\
                          = & s+2 - (A + B + C + M) \\
                          = & 0.
        \end{split}
    \end{equation*}

    Suppose, without loss of generality, that $\alpha \leq \beta \leq \gamma$.  The socle degrees of $R/I$ are $\alpha + b + c -3$, $a+\beta+c-3$,
    and $a+b+\gamma -3$, with the first removed if $\alpha = 0$.  The following argument shows that $\alpha + b + c - 3$ is at least $s+1$, however,
    with a simple changing of names it can be used to show that each of the socle degrees is at least $s+1$.

    As we are considering the socle degree $\alpha + b + c - 3$, we may assume $\alpha \geq 1$.  Notice that $\alpha + b + c - 3 = 2A + B + C + 2M + \alpha - 3$,
    which is at least $s+1 = A+B+C+M-1$ exactly when $A+M+\alpha \geq 2$.  If $A+M\geq 1$, then we are done.  Suppose $A+M = 0$, then $A = M = 0$ and 
    $b+c = \alpha+\beta+\gamma$.  Moreover, since $b > \beta$ and $c > \gamma$, then $\alpha+\beta+\gamma = b+c \geq \beta+\gamma + 2$.  Thus $\alpha \geq 2$.
\end{proof}

An immediate consequence of the previous lemma is that exactly one map need be considered for each algebra in order to determine the presence 
of the weak Lefschetz property.
\begin{corollary} \label{cor:one-map}
    Assume the parameters of $I = I_{a,b,c,\alpha,\beta,\gamma}$ satisfy the conditions in Proposition~\ref{pro:semistable} and suppose 
    $a + b + c + \alpha + \beta + \gamma \equiv 0 \pmod{3}$.  Then $R/I$ has the weak Lefschetz property if and only if the map 
    $\times(x+y+z): [R/I]_s \rightarrow [R/I]_{s+1}$ is injective (or surjective).
\end{corollary}
\begin{proof}
    This follows immediately from Lemma~\ref{lem:twin-peaks} by using Propositions~\ref{pro:surj}--\ref{pro:mono}.
\end{proof}

This leads to the definition of two matrices with determinants that determine the weak Lefschetz property.  The first is a zero-one matrix and the
second is a matrix of binomial coefficients.  
\begin{proposition} \label{pro:wlp-zero-one}
    Assume the parameters of $I = I_{a,b,c,\alpha,\beta,\gamma}$ satisfy the conditions in Proposition~\ref{pro:semistable} and suppose 
    $a + b + c + \alpha + \beta + \gamma \equiv 0 \pmod{3}$.  
    
    Then there exists a matrix $Z = Z_{a,b,c,\alpha,\beta,\gamma}$ such that
    \begin{enumerate}
        \item $Z$ is a square integer matrix of size $h_s$,
        \item $R/I$ has the weak Lefschetz property if and only if $\det{Z} \not\equiv 0 \pmod{\charf{K}}$, and
        \item the entries of $Z$ are given by
            \[
                (Z)_{i,j} = \left\{
                    \begin{array}{ll}
                        1 & n_j \mbox{ is a multiple of } m_i, \\[0.3em]
                        0 & \mbox{otherwise},
                    \end{array}
                \right.
            \]
            where $\{m_1, \ldots, m_{h_s}\}$ and $\{n_1, \ldots, n_{h_s}\}$ are the monomial bases of $[R/I]_s$ and $[R/I]_{s+1}$, respectively,
            and are given in lexicographic order.
    \end{enumerate}
\end{proposition}
\begin{proof}
    We notice that the map $\times(x+y+z): [R/I]_s \rightarrow [R/I]_{s+1}$ can be represented as a matrix $Z$ with rows and columns indexed
    by fixed monomial bases of $[R/I]_s$ and $[R/I]_{s+1}$, respectively.  This follows immediately from viewing $[R/I]_d$ as a vector space over $K$.

    Claim (i) follows from Lemma~\ref{lem:twin-peaks} wherein it is shown that $h_s = h_{s+1}$.  Since $Z$ is square, then the injectivity of 
    $\times(x+y+z): [R/I]_s \rightarrow [R/I]_{s+1}$ is equivalent to $Z$ being invertible, that is, equivalent to $\det{Z}$ being non-zero
    in $K$.  Thus, claim (ii) follows from Corollary~\ref{cor:one-map} wherein it is shown that the injectivity of the map 
    $\times(x+y+z): [R/I]_s \rightarrow [R/I]_{s+1}$ exactly determines the presence of the weak Lefschetz property for $R/I$.  
    Claim (iii) follows immediately from the construction of the map.
\end{proof}

The following generalises the results in~\cite[Theorem~7.2 and Corollary~7.3]{MMN}.
\begin{proposition} \label{pro:wlp-binom}
    Assume the parameters of $I = I_{a,b,c,\alpha,\beta,\gamma}$ satisfy the conditions in Proposition~\ref{pro:semistable}, and suppose 
    $a + b + c + \alpha + \beta + \gamma \equiv 0 \pmod{3}$.  
    
    Then there exists a matrix $N = N_{a,b,c,\alpha,\beta,\gamma}$ such that
    \begin{enumerate}
        \item $N$ is a square integer matrix of size $C+M$,
        \item $R/I$ has the weak Lefschetz property if and only if $\det{N} \not\equiv 0 \pmod{\charf{K}}$, and
        \item the entries of $N$ are given by 
            \[
                (N)_{i,j} = \left\{
                    \begin{array}{ll}
                        \displaystyle \binom{c}{A + j - i} & \mbox{if } 1 \leq i \leq C, \\[0.8em]
                        \displaystyle \binom{\gamma}{A+C-\beta + j - i} & \mbox{if } C + 1 \leq i \leq C+M.
                     \end{array}
                \right.
            \]
    \end{enumerate}
\end{proposition}
\begin{proof}
    Notice that $R/(I,x+y+z) \cong S/J$, where $S = K[x,y]$ and 
    \[
        J = (x^a, y^b, (x+y)^c, x^\alpha y^\beta (x+y)^\gamma).
    \]
    Thus the sequence
    \[
        [R/I]_d \xrightarrow{\times(x+y+z)} [R/I]_{d+1} \rightarrow [R/(I, x+y+z)]_{d+1} \rightarrow 0
    \]
    implies that $\times(x+y+z): [R/I]_s \rightarrow [R/I]_{s+1}$ is injective exactly when $[S/J]_{s+1} = 0$.
    Hence it suffices to show that all $s+2$ monomials of the form $x^iy^j$ where $i+j=s+1$ are in $J$.

    Clearly if $i \geq a$ or $j \geq b$, then $x^iy^j$ is in $J$.  This leaves $s+2 - (s+2-a) - (s+2-b) = s+2 - A - B = C+M$ monomials
    that are not trivially in $J$.  Thus there are $C+M$ equations and unknowns, all of which only involve the non-monomial
    terms (after reduction by the monomial terms).  Associated to this system of equations is a square integer matrix
    of size $C+M$, call it $N$.  Then $N$ is invertible if and only if $\det{N}$ is non-zero in $K$.  Thus, claims (i) and (ii) hold.

    There are $s+2-c = C$ ways to scale $(x+y)^c$ and $s+2-(\alpha+\beta+\gamma) = M$ ways to scale $x^\alpha y^\beta (x+y)^\gamma$
    to be degree $s+1$.  In both cases consider the binomial coefficient indexed by the degree of $y$.  Then $(N)_{i,j}$ is the
    coefficient on $x^{a-j}y^{A+j-1}$ in the scaling $x^{C-i}y^{i-1}(x+y)^c$ for $1 \leq i \leq C$, i.e., $\binom{c}{A+j-i}$, and in the scaling
    $x^{C+M-i}y^{i-C-1}x^\alpha y^\beta (x+y)^\gamma$ for $C+1 \leq i \leq C+M$, i.e., $\binom{\gamma}{A+C-\beta+j-i}$.  Thus claim (iii) holds.
\end{proof}

Clearly $\det{Z_{a,b,c,\alpha,\beta,\gamma}}$ and $\det{N_{a,b,c,\alpha,\beta,\gamma}}$ must both be either zero or have the same set of prime
divisors.  We can determine a few of the prime divisors from the known failure of the weak Lefschetz property.
\begin{proposition} \label{pro:wlp-p}
    Assume the parameters of $I = I_{a,b,c,\alpha,\beta,\gamma}$ satisfy the conditions in Proposition~\ref{pro:semistable}, and suppose 
    $a + b + c + \alpha + \beta + \gamma \equiv 0 \pmod{3}$.  If $K$ has positive characteristic $p$ and their exists
    a positive integer $m$ such that 
    \[ 
        \max\{a,b,c\} \leq p^m \leq s+1 = \frac{1}{3}(a + b + c + \alpha + \beta + \gamma)-1,
    \]
    then 
    \begin{enumerate}
        \item $R/I$ fails to have the weak Lefschetz property,
        \item $p$ is a prime divisor of the determinant of $Z_{a,b,c,\alpha,\beta,\gamma}$, and
        \item $p$ is a prime divisor of the determinant of $N_{a,b,c,\alpha,\beta,\gamma}$.
    \end{enumerate}
\end{proposition}
\begin{proof}
    By Lemma~\ref{lem:twin-peaks}, the Hilbert function of $R/I$ weakly increases to degree $s+1$, hence part~(i) follows by Lemma~\ref{lem:wlp-p}.
    Parts~(ii) and~(iii) then follow from Propositions~\ref{pro:wlp-zero-one} and~\ref{pro:wlp-binom}, respectively.
\end{proof}

In the next section we will see a nice combinatorial interpretation for both matrices as well as the defined values $s, A, B, C,$ and $M$.

% -- Section
\section{Punctured hexagons and friends} \label{sec:ph}

Recall the definition of $s, A, B, C,$ and $M$, and the conditions thereon, from Proposition~\ref{pro:semistable}.  In this section
we assume, without exception, that $I = I_{a,b,c,\alpha,\beta,\gamma}$ has parameters matching these conditions and further
that $a + b + c + \alpha + \beta + \gamma \equiv 0 \pmod{3}$. 

~\subsection{Punctured hexagons}~

Notice that every monomial in $[R]_d$ is of the form $x^i y^j z^k$ where $i,j,$ and $k$ are non-negative integers such that $i+j+k=d$.  Hence 
we can organise the monomials in $[R]_d$ into a triangle of side-length $d+1$ with $x^d$ at the lower-center, $y^d$ at the upper-right, and
$z^d$ at the upper-left.  (See Figure~\ref{fig:mono-tri}.)

\begin{figure}[!ht]
    \includegraphics[scale=1.5]{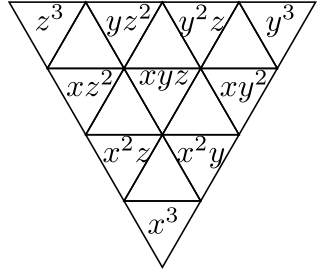}
    \caption{The monomial triangle for $[R]_3$}
    \label{fig:mono-tri}
\end{figure}

Notice that we can interlace the monomials of $[R]_{d-1}$ within the monomials of $[R]_{d}$.  If we stay consistent with our orientation (i.e.,
largest power of $x$ at the lower-center, largest power of $y$ at the upper-right, and largest power of $z$ at the upper-left), then two monomials
are adjacent if and only if one divides the other.  (See Figure~\ref{fig:interlace-tri}.)  We call such a figure the {\em interlaced basis
region of $[R]_{d-1}$ and $[R]_{d}$}.

\begin{figure}[!ht]
    \includegraphics[scale=1.5]{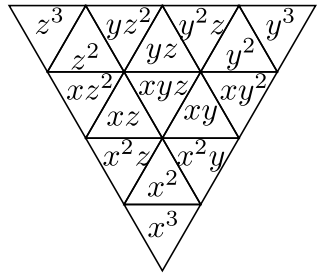}
    \caption{The interlaced basis region of $[R]_2$ and $[R]_3$}
    \label{fig:interlace-tri}
\end{figure}

If we compute the interlaced basis region of $[R/I_{a,b,c,\alpha,\beta,\gamma}]_{s}$ and $[R/I_{a,b,c,\alpha,\beta,\gamma}]_{s+1}$, then
we get a punctured hexagonal region.

\begin{theorem} \label{thm:interlace-amaci}
    Let $I = I_{a,b,c,\alpha,\beta,\gamma}$ satisfy the conditions in Proposition~\ref{pro:semistable}, and suppose 
    $a + b + c + \alpha + \beta + \gamma \equiv 0 \pmod{3}$.  Then the interlaced basis region $H_{a,b,c,\alpha,\beta,\gamma}$ 
    of $[R/I]_{s}$ and $[R/I]_{s+1}$ is in the shape of a hexagon with side-lengths (in clockwise cyclic order, starting at the bottom)
    \[
        (A, B+M, C, A+M, B, C+M)
    \]
    and with a puncture in the shape of an equilateral triangle of side-length $M$.  The puncture has sides parallel to the sides of the hexagon 
    of lengths $A+M, B+M,$ and $C+M$.  Moreover, the puncture is located $\alpha, \beta,$ and $\gamma$ units from the sides of length
    $A+M, B+M,$ and $C+M$, respectively.  (See Figure~\ref{fig:punctured-hexagon}.)
    \begin{figure}[!ht]
        \includegraphics[scale=0.667]{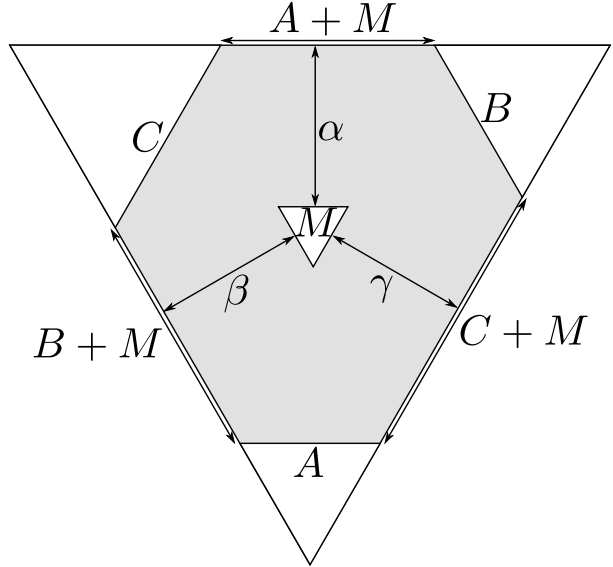}
        \caption{$H_{a,b,c,\alpha,\beta,\gamma}$, the interlaced basis region of $[R/I]_s$ and $[R/I]_{s+1}$}
        \label{fig:punctured-hexagon}
    \end{figure}
\end{theorem}
\begin{proof}
    The interlaced basis region of $[R/I]_s$ and $[R/I]_{s+1}$ corresponds to a spatial placement of the monomials of the associated 
    components of $R/I$.  As $I$ is a monomial ideal, we can easily get restrictions on the monomials $x^iy^jz^k$ in the region:
    \begin{enumerate}   
        \item The generator $x^a$ forces $0 \leq i < a$; this corresponds to the lower-center missing triangle which has side-length $s+2 - a = A$.
        \item The generator $y^b$ forces $0 \leq j < b$; this corresponds to the upper-right missing triangle which has side-length $s+2 - b = B$.
        \item The generator $z^c$ forces $0 \leq k < c$; this corresponds to the upper-left missing triangle which has side-length $s+2 - c = C$.
        \item The generator $x^\alpha y^\beta z^\gamma$ forces one of $i < \alpha, j < \beta$, or $k < \gamma$ to also hold;  this corresponds
            to the center missing triangle, which has side-length $s+2 - \alpha - \beta - \gamma = M$.  This further forces the particular
            placement of the puncture.
    \end{enumerate}
    Moreover, the conditions in Proposition~\ref{pro:semistable} force the regions to have non-negative side-lengths and to not overlap.
\end{proof}

\begin{remark} \label{rem:bijection}
    The ideals $I = I_{a,b,c,\alpha,\beta,\gamma}$ in Theorem~\ref{thm:interlace-amaci} are in bijection with their hexagonal regions 
    (assuming a fixed orientation and assuming a puncture of side-length zero is still considered to be in a particular position).  
    Suppose we have a punctured hexagonal region, as in Figure~\ref{fig:punctured-hexagon}, with parameters $A,B,C,M,\alpha,$ and 
    $\beta$.  Then $a = B+C+M$, $b = A+C+M$, $c = A+B+M$, and $\gamma = A+B+C-(\alpha + \beta)$.

    Moreover, we notice that, in characteristic zero, these ideals are exactly the artinian monomial almost complete intersections which
    do not immediately have the weak Lefschetz property from Proposition~\ref{pro:amaci-not-3} or Proposition~\ref{pro:semistable}.
\end{remark}

Notice that by Lemma~\ref{lem:twin-peaks} we have $h_s = h_{s+1}$, so the region $H_{a,b,c,\alpha,\beta,\gamma}$ has the same number of upward
pointing triangles as it has downward pointing triangles.  In particular, it may then be possible to tile the region by lozenges (i.e., rhombi 
with unit side-lengths and angles of $60^\circ$ and $120^\circ$; we also note a pair of alternate names used in the literature: calissons and
diamonds). 

~\subsection{Non-intersecting lattice paths}~

We follow~\cite[Section~5]{CEKZ} (similarly,~\cite[Section~2]{Fi}) to translate lozenge tilings of $H_{a,b,c,\alpha,\beta,\gamma}$ to 
families of non-intersecting lattice paths.  An example of a lozenge tiling and its associated family of non-intersecting lattice paths
is given in Figure~\ref{fig:hex-nilp}.

\begin{figure}[!ht]
    \begin{minipage}[b]{0.49\linewidth}
        \centering
        \includegraphics[scale=0.45]{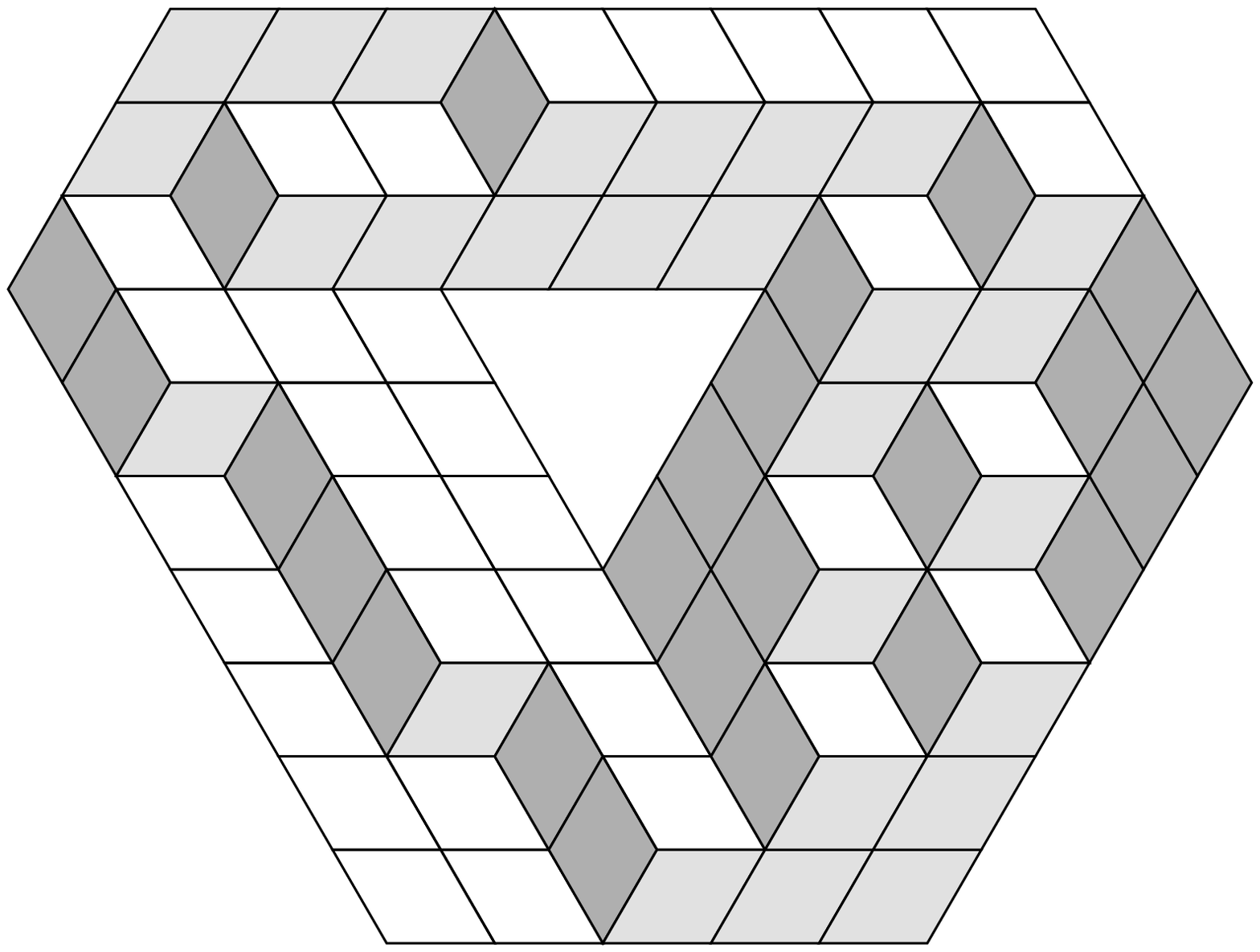}\\
        {\em Hexagon tiling by lozenges}
    \end{minipage}
    \begin{minipage}[b]{0.49\linewidth}
        \centering
        \includegraphics[scale=0.45]{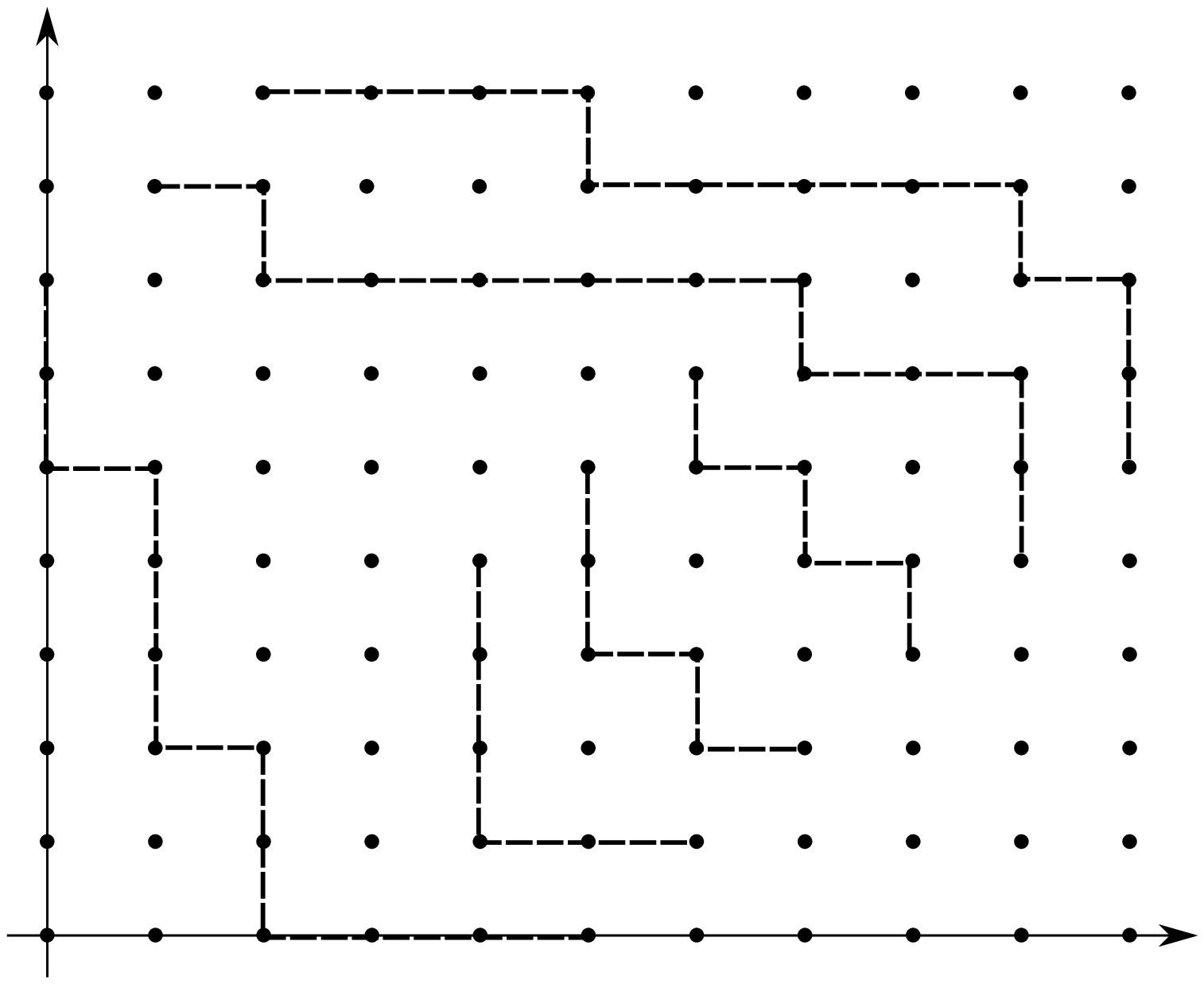}\\
        {\em Family of non-intersecting lattice paths}
    \end{minipage}
    \caption{Example of a lozenge tiling and its associated family of non-intersecting lattice paths}
    \label{fig:hex-nilp}
\end{figure}

In order to transform a lozenge tiling of a punctured hexagon $H_{a,b,c,\alpha,\beta,\gamma}$ into a family of non-intersecting
lattice paths, we follow three simple steps (see Figure~\ref{fig:hex-to-nilp}):
\begin{enumerate}   
    \item Mark the midpoints of the triangle edges parallel to the sides of length $C$ and $C+M$ with vertices.  Further,
        label the midpoints, always moving lower-left to upper-right,
        \begin{enumerate}
            \item along the hexagon side of length $C$ as $A_1, \ldots, A_C$,
            \item along the puncture as $A_{C+1}, \ldots, A_{C+M}$, and
            \item along the hexagon side of length $C+M$ as $E_1, \ldots, E_{C+M}$.
        \end{enumerate}
    \item Using the lozenges as a guide, we connect any pair of vertices that occur on a single lozenge.
    \item Thinking of motion parallel to the side of length $A$ as horizontal and motion parallel to the side of length $B$
        as vertical, we orthogonalise the lattice (and paths) and consider the lower-left vertex as the origin.
\end{enumerate}

\begin{figure}[!ht]
    \begin{minipage}[b]{0.49\linewidth}
        \centering
        \includegraphics[scale=0.45]{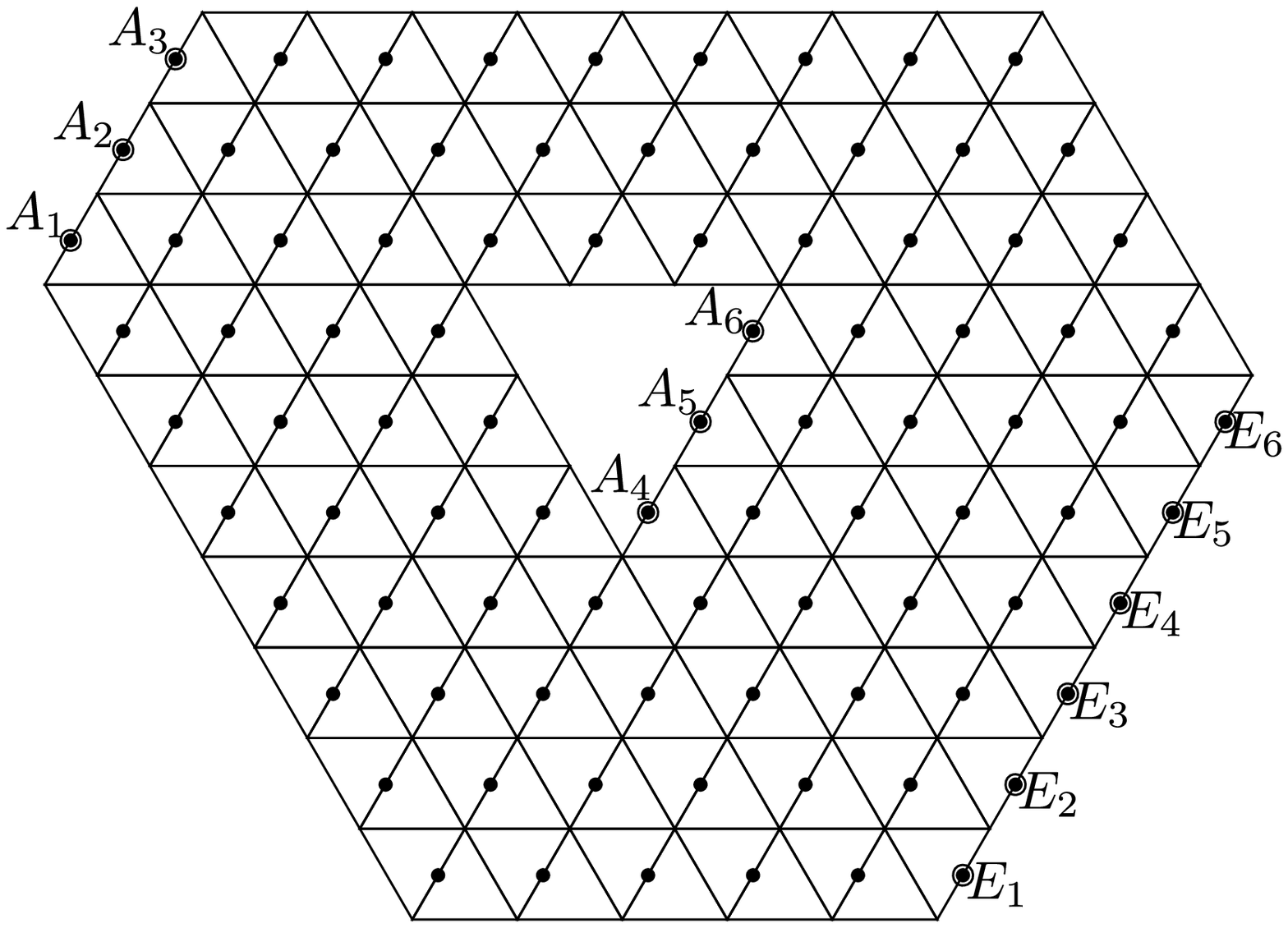}\\
        {\em (i) Mark midpoints with vertices and label particular vertices}
    \end{minipage}
    \begin{minipage}[b]{0.49\linewidth}
        \centering
        \includegraphics[scale=0.45]{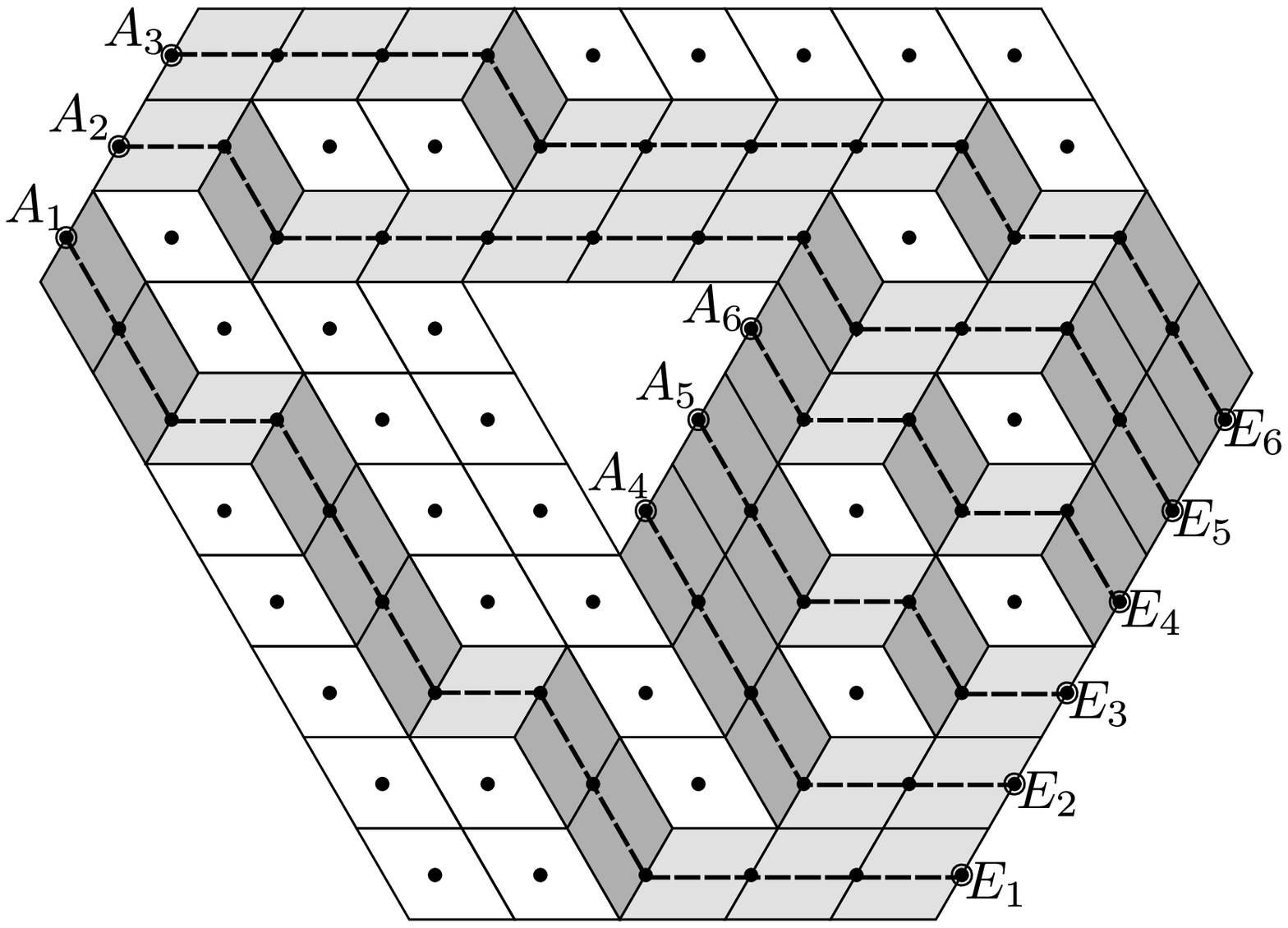}\\
        {\em (ii) Connect vertices using the tiling}
    \end{minipage} \\[0.5em]

    \begin{minipage}[b]{0.49\linewidth}
        \centering
        \includegraphics[scale=0.45]{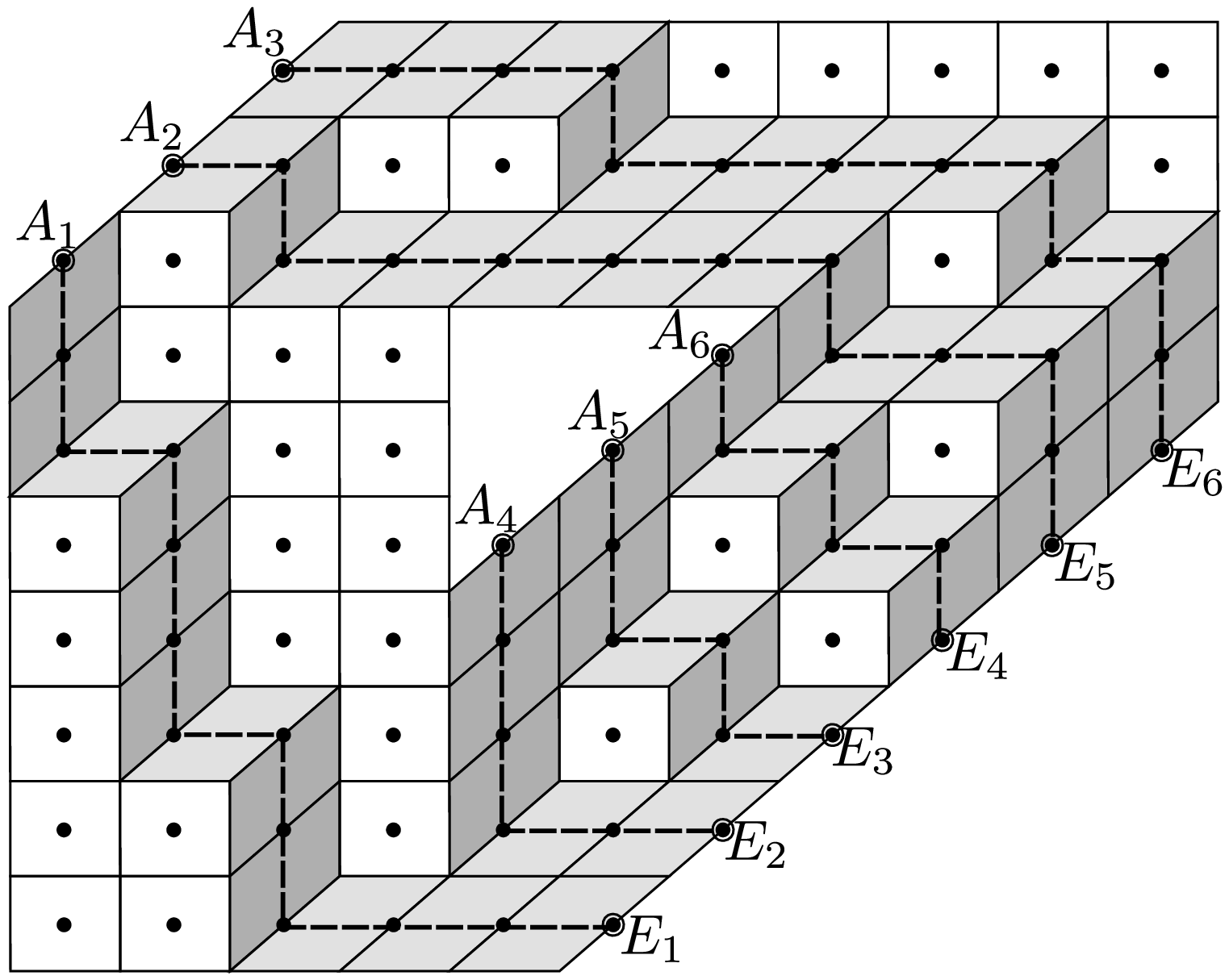}\\
        {\em (iii) Orthogonalise the path family}
    \end{minipage}
    \begin{minipage}[b]{0.49\linewidth}
        \centering
        \includegraphics[scale=0.45]{nilp-fin}\\
        {\em The family by itself}
    \end{minipage}
    \caption{Example of converting lozenge tilings to families of non-intersecting lattice paths}
    \label{fig:hex-to-nilp}
\end{figure}

Given the above transformation of $H_{a,b,c,\alpha,\beta,\gamma}$ to the integer lattice, we see that $A_i$ and $E_j$ have easy to compute
coordinates:
\[
    A_i = \left\{ 
            \begin{array}{ll}
                (i-1, B+M+i-1) & \mbox{if } 1 \leq i \leq C, \\[0.3em]
                (\beta + i - C - 1, B - \alpha + i - 1) & \mbox{if } C+1 \leq i \leq C+M,
            \end{array}
          \right.
\]
and
\[
    E_j = (A + j - 1, j - 1) \mbox{ for } 1 \leq j \leq C+M.
\]

Now we associate to each family of non-intersecting lattices paths a permutation and use it to assign a sign to the family of paths.
\begin{definition} \label{def:nilp-sign}
    Let $L$ be a family of non-intersecting lattice paths as above, and let $\lambda \in \PS_{C+M}$ be
    the permutation so that $A_i$ is connected to $E_{\lambda(i)}$.  We define the {\em sign} of $L$ to be the signature
    (or sign) of the permutation $\lambda$.  That is, $\sgn{L} := \sgn{\lambda}$.
\end{definition}

Now we are ready to use a beautiful theorem relating (signed) enumerations of families of non-intersecting lattice paths with 
determinants.  In particular, we use a theorem first given by Lindstr\"om in~\cite[Lemma~1]{Li} and stated independently in~\cite[Theorem~1]{GV-1989}
by Gessel and Viennot.  Stanley gives a very nice exposition of the topic in~\cite[Section~2.7]{St-EC}.

Here we give a specialisation of the theorem to the case when all edges have the same weight---one.  In particular, this result is
given in~\cite[Lemma~14]{CEKZ}.

\begin{theorem} \label{thm:lgv}
    Let $A_1, \ldots, A_n, E_1, \ldots, E_n$ 
    be distinct lattice points on $\NN_0^2$ where each $A_i$ is above and to the left of every $E_j$.  Then
    \[
        \det_{1\leq i, j \leq n} (P(A_i \rightarrow E_j)) = \sum_{\lambda \in \PS_n} \sgn(\lambda) P^+_\lambda(A\rightarrow E),
    \]
    where $P(A_i \rightarrow E_j)$ is the number of lattice paths from $A_i$ to $E_j$ and, for each permutation $\lambda \in \PS_n$,
    $P^+_\lambda(A \rightarrow E)$ is the number of families of non-intersecting lattice paths with paths going from $A_i$ to $E_{\lambda(i)}$.
\end{theorem}

Thus, we have an enumeration of the signed lozenge tilings of a punctured hexagon with signs given by the non-intersecting lattice paths.

\begin{theorem} \label{thm:nilp-matrix}
    The enumeration of signed lozenge tilings of $H_{a,b,c,\alpha,\beta,\gamma}$, with signs given by the signs of the associated
    families of non-intersecting lattice paths (Definition~\ref{def:nilp-sign}), is given by $\det N_{a,b,c,\alpha,\beta,\gamma}$,
    where the matrix $N_{a,b,c,\alpha,\beta,\gamma}$ is defined in Proposition~\ref{pro:wlp-binom}.
\end{theorem}
\begin{proof}
    Notice that the number of lattice paths from $(u,v)$ to $(x,y)$, where $u \leq x$ and $v \geq y$, is given by $\binom{x-u+v-y}{x-u}$ as there are
    $x-u + v-y$ steps and $x-u$ must be horizontal steps (equivalently, $v-y$ must be vertical steps).  Thus the claim follows immediately from the steps above.
\end{proof}

However, we need not consider all $(C+M)!$ permutations $\lambda \in \PS_{C+M}$ as the vast majority will always have $P^+_\lambda(A \rightarrow E) = 0$.
Given our choice of $A_i$ and $E_j$ the only possible choices of $\lambda$ are given by
\[
    \lambda_k = \left( 
        \begin{array}{cccc|cccc|cccc}
            1 & 2 & \cdots & k      & k+1   & k+2   & \cdots & C        & C+1 & C+2 & \cdots & C+M \\
            1 & 2 & \cdots & k      & M+k+1 & M+k+2 & \cdots & C+M      & k+1 & k+2 & \cdots & M+k 
        \end{array}
    \right),
\]
where $0 \leq k \leq C$ and $k$ corresponds to the number of lattice paths that go below the puncture.  In particular, the three parts
of $\lambda_k$ correspond to the paths going below, above, and starting from the puncture.  We call these permutations the {\em admissible permutations}
of $H_{a,b,c,\alpha,\beta,\gamma}$.

We will use this connection to compute determinants in Section~\ref{sec:det}, but first we look at an alternate combinatorial connection.

~\subsection{Perfect matchings}~

Lozenge tilings of a punctured hexagon can be associated to perfect matchings on a bipartite graph.  This connection
was first used by Kuperberg in~\cite{Ku} to study symmetries on plane partitions.
An example of a lozenge tiling and its associated perfect matching of edges is given in Figure~\ref{fig:hex-bip}.

\begin{figure}[!ht]
    \begin{minipage}[b]{0.49\linewidth}
        \centering
        \includegraphics[scale=0.45]{tile}\\
        {\em (i) Hexagon tiling by lozenges}
    \end{minipage}
    \begin{minipage}[b]{0.49\linewidth}
        \centering
        \includegraphics[scale=0.45]{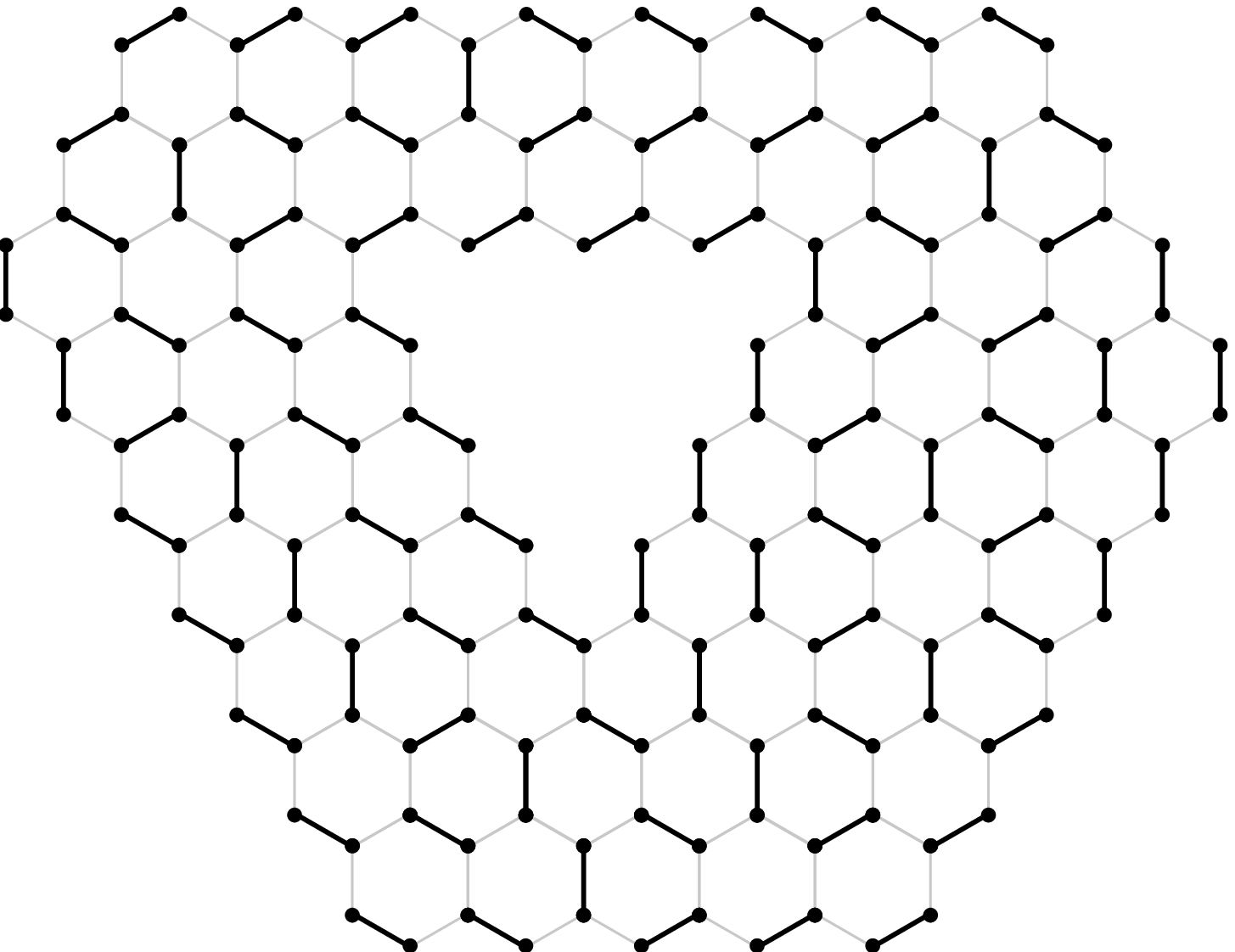}\\
        {\em (ii) Perfect matching of edges}
    \end{minipage}
    \caption{Example of a lozenge tiling and its associated perfect matching of edges}
    \label{fig:hex-bip}
\end{figure}

In order to transform a lozenge tiling of a punctured hexagon $H_{a,b,c,\alpha,\beta,\gamma}$ into a perfect matching of edges,
we follow three simple steps (see Figure~\ref{fig:hex-to-bip}):
\begin{enumerate}
    \item Put a vertex at the center of each triangle.
    \item Connect the vertices whose triangles are adjacent.
    \item Select the edges which the lozenges cover--this set is the perfect matching.
\end{enumerate}

\begin{figure}[!ht]
    \begin{minipage}[b]{0.49\linewidth}
        \centering
        \includegraphics[scale=0.45]{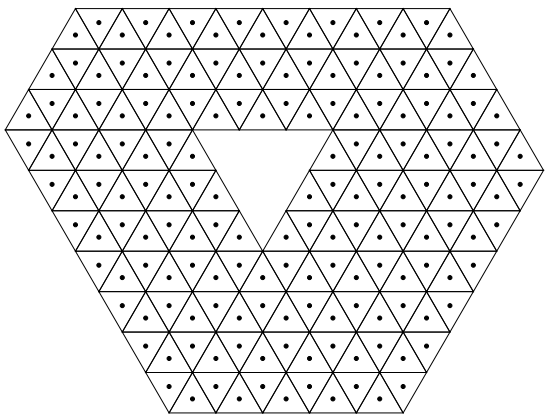}\\
        {\em (i) Put vertices in triangle centers}
    \end{minipage}
    \begin{minipage}[b]{0.49\linewidth}
        \centering
        \includegraphics[scale=0.45]{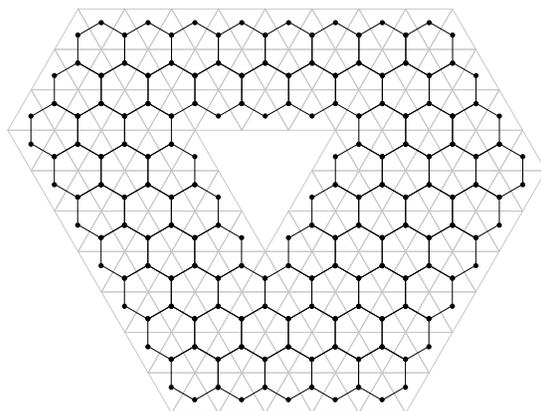}\\
        {\em (ii) Connect vertices of adjacent triangles}
    \end{minipage} \\[0.5em]

    \begin{minipage}[b]{0.49\linewidth}
        \centering
        \includegraphics[scale=0.45]{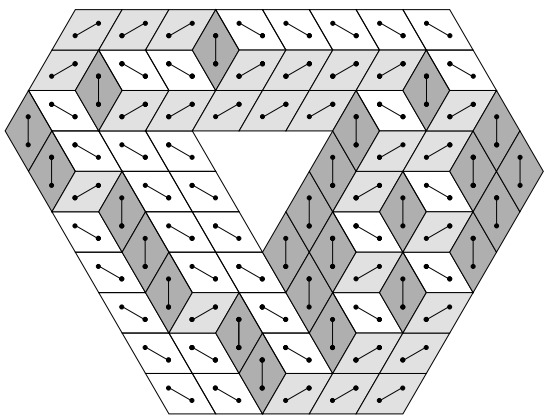}\\
        {\em (iii) Select edges covered by lozenges}
    \end{minipage}
    \begin{minipage}[b]{0.49\linewidth}
        \centering
        \includegraphics[scale=0.45]{bip-fin}\\
        {\em The perfect matching by itself}
    \end{minipage}
    \caption{Example of converting lozenge tilings to perfect matchings of edges}
    \label{fig:hex-to-bip}
\end{figure}

Notice that the graph associated to the punctured hexagon $H_{a,b,c,\alpha,\beta,\gamma}$ is a bipartite graph with
colour classes given by monomials in $[R/I]_{s}$ and $[R/I]_{s+1}$.  Thus we can represent this bipartite graph by a
bi-adjacency matrix with rows enumerated by the monomials in $[R/I]_{s}$ and columns enumerated by the monomials in
$[R/I]_{s+1}$.  We fix the order on the monomials to be the lexicographic order.  Clearly then the matrix 
$Z_{a,b,c,\alpha,\beta,\gamma}$ from Proposition~\ref{pro:wlp-zero-one} is the bi-adjacency matrix described here.

Consider the permanent of $Z = Z_{a,b,c,\alpha,\beta,\gamma}$, that is, 
\[
    \per{Z} = \sum_{\pi \in \PS_{h_s}} \prod_{i=1}^{h_s} (Z)_{i,\pi(i)}.
\]
As $Z$ has entries which are either zero or one, we see that all summands in $\per{Z}$ are either zero or one.  Moreover,
each non-zero summand corresponds to a perfect matching, as it corresponds to an isomorphism between the two colours classes 
of the bipartite graph, namely, the monomials in $[R/I]_{s}$ and $[R/I]_{s+1}$.  Thus, $\per{Z}$ enumerates the perfect matchings
of the bipartite graph associated to $H_{a,b,c,\alpha,\beta,\gamma}$, and hence $\per{Z}$ also enumerates the lozenge tilings
of $H_{a,b,c,\alpha,\beta,\gamma}$.

\begin{proposition} \label{pro:bip-matrix}
    The number of lozenge tilings of $H_{a,b,c,\alpha,\beta,\gamma}$ is $\per{Z_{a,b,c,\alpha,\beta,\gamma}}.$
\end{proposition}

Since each perfect matching is an isomorphism between the two colour classes, it can be seen as a permutation $\pi \in \PS_{h_s}$.
As with Definition~\ref{def:nilp-sign}, it is thus natural to assign a sign to a given perfect matching.

\begin{definition} \label{def:pm-sign}
    Let $P$ be a perfect matching of the bipartite graph associated to $H_{a,b,c,\alpha,\beta,\gamma}$, and let $\pi \in \PS_{h_s}$
    be the associated permutation (as described above).  We define th {\em sign} of $P$ to be the signature of the permutation
    $\pi$.  That is, $\sgn{P} := \sgn{\pi}$.
\end{definition}

Since the sign is the sign that is used in computing the determinant of the matrix $Z_{a,b,c,\alpha,\beta,\gamma}$, we get an
enumeration of the signed lozenge tilings of a punctured hexagon with signs given by the perfect matchings.

\begin{theorem} \label{thm:bip-matrix}
    The enumeration of signed perfect matchings of the bipartite graph associated to $H_{a,b,c,\alpha,\beta,\gamma}$, with signs given by the
    signs of the related perfect matchings, is given by $\det Z_{a,b,c,\alpha,\beta,\gamma}$, where the matrix 
    $Z_{a,b,c,\alpha,\beta,\gamma}$ is defined in Proposition~\ref{pro:wlp-zero-one}.
\end{theorem}

\begin{remark} \label{rem:dimers}
    Kasteleyn~\cite{Ka} provided, in 1967, a general method for computing the number of perfect matchings of a planar graph as a determinant.
    Moreover, he provided a classical review of methods and applications of enumerating perfect matchings.  Planar graphs, such as the
    ``honeycomb graphs'' described here, are studied for their connections to physics; in particular, honeycomb graphs model the bonds
    in dimers (polymers with only two structural units) and perfect matchings correspond to so-called {\em dimer coverings}.  Kenyon~\cite{Ke}
    gives a modern recount of explorations on dimer models, including random dimer coverings and their limiting shapes.
\end{remark}

\begin{remark} \label{rem:wlp-p}
    Recall that Proposition~\ref{pro:wlp-p} provides a numerical constraint that determines some of the prime divisors of the determinants
    of the matrices $Z_{a,b,c,\alpha,\beta,\gamma}$ and $N_{a,b,c,\alpha,\beta,\gamma}$ by means of some algebra deciding the weak Lefschetz 
    property for the algebra $R/I_{a,b,c,\alpha,\beta,\gamma}$.  Hence, by Theorems~\ref{thm:nilp-matrix} and~\ref{thm:bip-matrix}, we see that
    information from algebra can indeed be used to determine some of the prime divisors of the enumerations of signed lozenge tilings and of signed
    perfect matchings.
\end{remark}

Finally, we note that in~\cite{Pr}, Propp gives a history of the connections between lozenge tilings (of non-punctured hexagons), 
perfect matchings, plane partitions, non-intersecting lattice paths.

% -- Section
\section{Interlude of signs} \label{sec:signs}
In the preceding section we discussed three related combinatorial structures from which we can extract the primes $p$ for which
the algebras $R/I_{a,b,c,\alpha,\beta,\gamma}$ fail to have the weak Lefschetz property.  Therein we discussed two different ways
to assign a sign to a lozenge tiling:  by the associated family of non-intersecting lattice paths (Definition~\ref{def:nilp-sign})
and by the associated perfect matching (Definition~\ref{def:pm-sign}).  We now show that the two signs indeed agree.

Fix a hexagonal region $H = H_{a,b,c,\alpha,\beta,\gamma}$, and fix a lozenge tiling $T$ of $H$.  As discussed in Section~\ref{sec:ph},
we associate to the tiling $T$ a family of non-intersecting lattice paths $L_T$ and a perfect matching $P_T$.  Moreover, 
we introduced a permutation $\lambda_T \in \PS_{C+M}$ associated to $L_T$ (see Definition~\ref{def:nilp-sign}) and a permutation $\pi_T \in \PS_{h_s}$
associated to $P_T$ via $Z_{a,b,c,\alpha,\beta,\gamma}$ (see Definition~\ref{def:pm-sign}).

We first notice that ``rotating'' particular lozenge groups of $T$ do not change the permutation associated to the non-intersecting
lattice paths.  
\begin{lemma} \label{lem:perm-rotate}
    Let $T$ be a lozenge tiling of $H_{a,b,c,\alpha,\beta,\gamma}$.  Pick any triplet of lozenges in $T$ which is either an {\em up}
    or a {\em down} grouping, as in Figure~\ref{fig:rotatable},
    \begin{figure}[!ht]
        \includegraphics[scale=0.667]{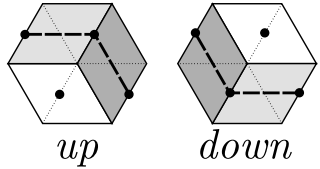}
        \caption{{\em up} and {\em down} lozenge groups with lattice path pieces superimposed}
        \label{fig:rotatable}
    \end{figure}
    and let $U$ be $T$ with the triplet exchanged for the other possibility (i.e., rotated $180^\circ$).  Then $U$ is a lozenge tiling
    of $H_{a,b,c,\alpha,\beta,\gamma}$ and $\lambda_T = \lambda_U$.  Moreover, $\pi_U = \tau \pi_T$, for some three-cycle $\pi \in \PS_{h_s}$.
\end{lemma}
\begin{proof}
    First, we note that if $T$ is a lozenge tiling of $H_{a,b,c,\alpha,\beta,\gamma}$ then clearly so is $U$ as the change does not
    modify any tiles besides the three in the triplet.

    Next, notice that exchanging the triplet in $T$ for its rotation only modifies the associated family of non-intersecting lattice
    paths in one path.  Moreover, it does not change the starting or ending points of the path, merely the order in which it gets there,
    that is, either right then down or down then right.  Thus, $\lambda_T = \lambda_U$.

    Last, suppose, without loss of generality, that our chosen triplet is an {\em up} lozenge group.  Label the three upward pointing 
    triangles in the triplet $i, j, k$ as in Figure~\ref{fig:rotatable-labeled}.  
    \begin{figure}[!ht]
        \includegraphics[scale=1]{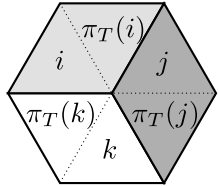}
        \caption{An {\em up} lozenge group with labeling}
        \label{fig:rotatable-labeled}
    \end{figure}
    Thus we see that $\pi_U(i) = \pi_T(k)$, $\pi_U(j) = \pi_T(i)$, $\pi_U(k) = \pi_T(j)$, and $\pi_U(m) = \pi_T(m)$ for $m$ not $i, j,$ or $k$.
    Hence $\pi_U = \tau \pi_T$ where $\tau$ is the three-cycle $(\pi_T(k), \pi_T(j),\pi_T(i))$.
\end{proof}

It follows that two lozenge tilings that have the same $\lambda$ permutation have $\pi$ permutations with the same sign.  

\begin{proposition} \label{pro:sign-agree}
    For each $H_{a,b,c,\alpha,\beta,\gamma}$ there exists a constant $i \in \{1, -1\}$ such that for all lozenge 
    tilings $T$ of $H_{a,b,c,\alpha,\beta,\gamma}$ the expression $\sgn{L_T} = i \cdot \sgn{P_T}$ holds, where $L_T$ is the family of 
    non-intersecting lattice paths associated to $T$ and $P_T$ is the family of perfect matchings associated to $T$.
\end{proposition}
\begin{proof}
    \noindent {\em Step 1}:

    Let $T$ and $U$ be two lozenge tilings of $H_{a,b,c,\alpha,\beta,\gamma}$ with $\lambda_T = \lambda_U$.
    As $\lambda_T = \lambda_U$, then the families of non-intersecting lattice paths associated to $T$ and $U$ start and end at the same
    places.  Hence $T$ can be modified by a series of, say $n$, rotations, as in Lemma~\ref{lem:perm-rotate}, to $U$.  Thus 
    \[
        \pi_U = \tau_n \tau_{n-1} \cdots \tau_1 \pi_T,
    \]
    where $\tau_1, \ldots, \tau_n \in \PS_{h_s}$ are three cycles by Lemma~\ref{lem:perm-rotate}.  As $\sgn\tau_i = 1$ for $1 \leq i \leq n$,
    and $\sgn$ is a group homomorphism, we see that $\sgn \pi_T = \sgn \pi_U$.

    \vspace*{0.5\baselineskip}
    \noindent {\em Step 2}:
    
    By the comments following Theorem~\ref{thm:nilp-matrix} we only need to consider the admissible permutations $\lambda_0, \ldots, \lambda_C$.
    Moreover, $\sgn \lambda_k = (-1)^{M(C-k)}$ so $\sgn \lambda_k = (-1)^M\sgn \lambda_{k+1}$.

    Let $T$ and $U$ be two lozenge tilings of $H = H_{a,b,c,\alpha,\beta,\gamma}$ with $\lambda_T = \lambda_k$ and $\lambda_U = \lambda_{k+1}$. 
    That is, $\sgn \lambda_T = (-1)^M \sgn \lambda_U$.
    First, $\alpha \geq C-k$ by the existence of $T$ as $C-k$ paths go above the puncture and so must go through a gap of size
    $\alpha$, and similarly $\beta \geq k+1$ by the existence of $U$.

    By {\em Step 1}, we may pick $T$ and $U$ however we wish, as long as $\lambda_T = \lambda_k$ and $\lambda_U = \lambda_{k+1}$.  In particular, let
    $T$, and similarly $U$, be defined as follows (see Figure~\ref{fig:maxmin-tiling}):
    \begin{figure}[!ht]
        \begin{minipage}[b]{0.45\linewidth}
            \centering
            \includegraphics[scale=0.75]{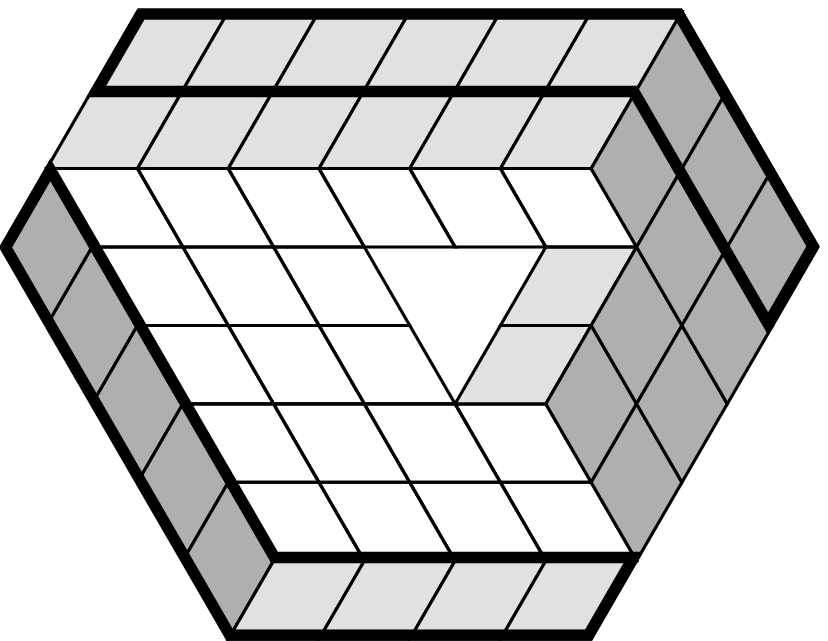}\\
            {\em (i) The tiling $T$}
        \end{minipage}
        \begin{minipage}[b]{0.45\linewidth}
            \centering
            \includegraphics[scale=0.75]{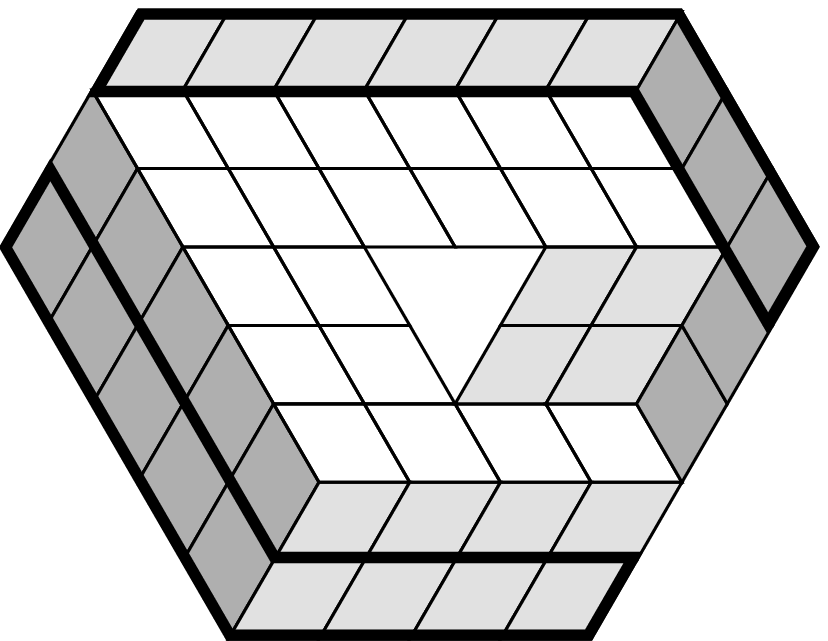}\\
            {\em (ii) The tiling $U$}
        \end{minipage}
        \caption{An example of tilings $T$ and $U$ of $H_{9,8,9,4,3,3}$, for $k = 1$, which are ``minimal'' below the puncture and ``maximal'' everywhere else;
                both tilings have the regions of similarity highlighted.}
        \label{fig:maxmin-tiling}
    \end{figure}
    \begin{enumerate}
        \item The $C-k$ paths above the puncture ($C-k-1$ for $U$) always move right before moving down.
        \item The $M$ paths from the puncture always move right before moving down.
        \item The $k$ paths below the puncture ($k+1$ for $U$) always move down before moving right.
    \end{enumerate}
    With the idea of {\em up} and {\em down} triplets from Lemma~\ref{lem:perm-rotate}, we can say a path is ``minimal'' if it contains no {\em up} triplets 
    and a path is ``maximal'' if it contains no {\em down} triplets.  Thus, $T$ and $U$ are ``minimal'' below the puncture and ``maximal'' everywhere
    else.
    
    Given this choice, $T$ and $U$ have exactly the same paths for the top $C-k-1$ paths above the puncture and the bottom $k$ paths below the puncture.
    Hence we can trim off these paths to make two new tilings, $T'$ and $U'$, of $H' = H_{B+M+1, A+M+1, c, \alpha - (C-k-1), \beta - k, \gamma}$.  
    Notice that $H$ and $H'$ have the same $A, B, M,$ and $\gamma$, only $C, \alpha,$ and $\beta$ have changed; in particular, $C' = 1$.  
    See Figure~\ref{fig:T-U-difference} parts (i) and (ii) for an example of the tilings $T'$ and $U'$ with their region-of-difference highlighted in bold.

    \begin{figure}[!ht]
        \begin{minipage}[b]{0.45\linewidth}
            \centering
            \includegraphics[scale=0.5]{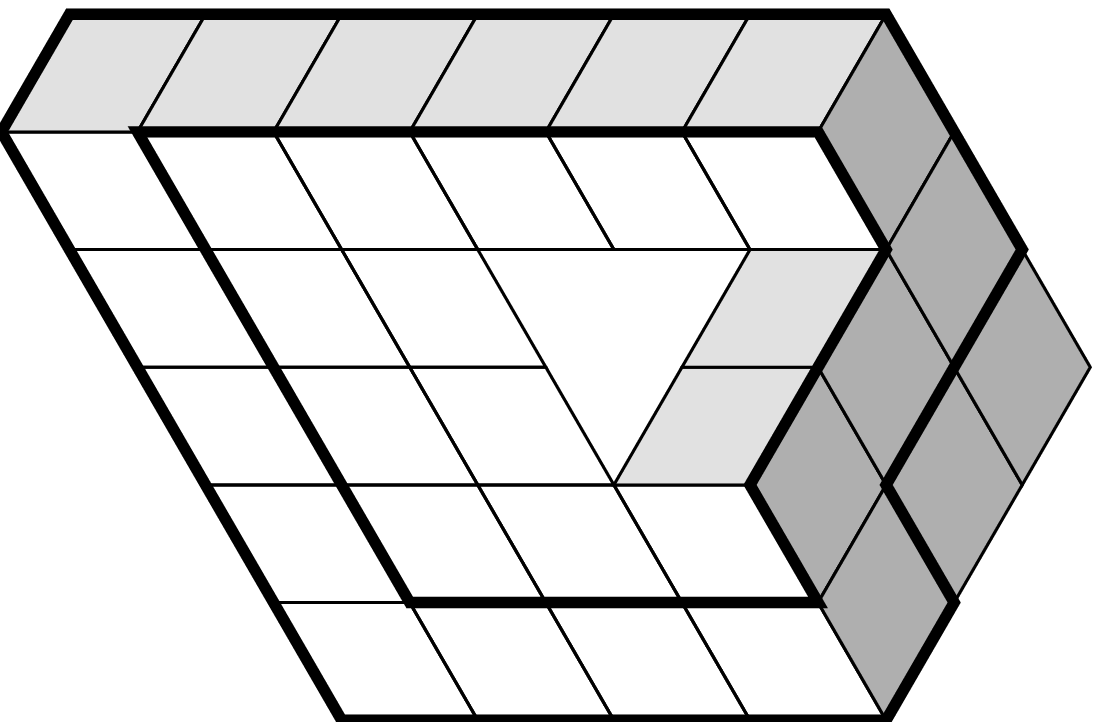}\\
            {\em (i) The tiling $T'$}
        \end{minipage}
        \begin{minipage}[b]{0.45\linewidth}
            \centering
            \includegraphics[scale=0.5]{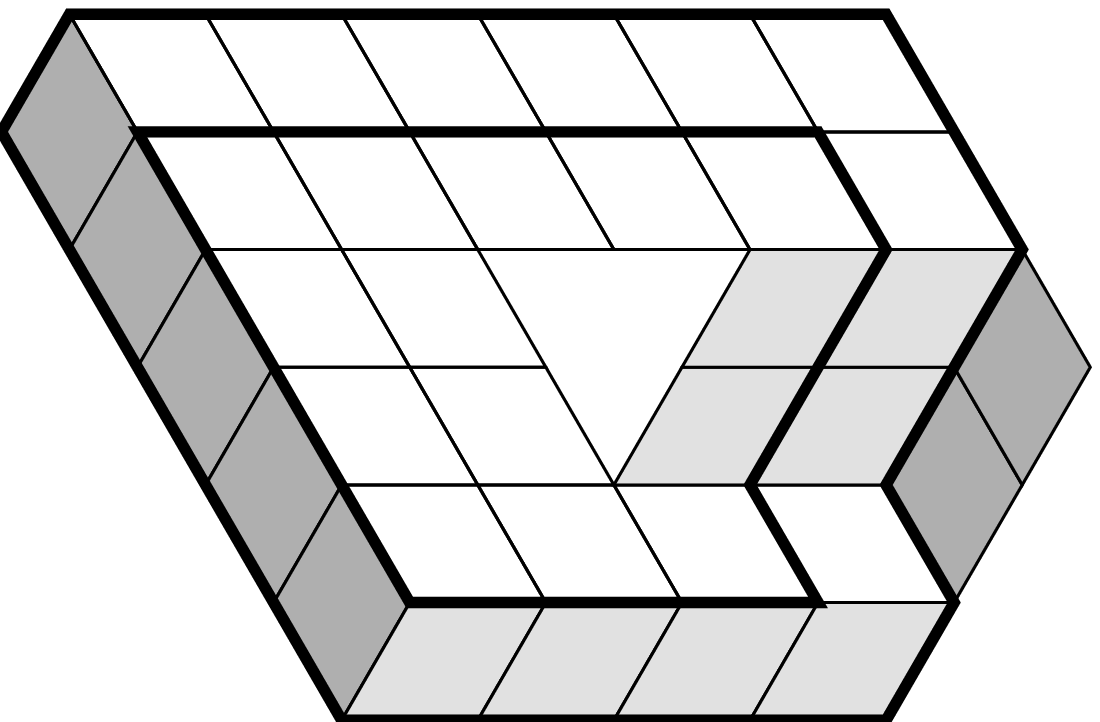}\\
            {\em (ii) The tiling $U'$}
        \end{minipage}
        \caption{The punctured hexagon $H_{7,6,9,3,2,3}$; both tilings have the region-of-difference highlighted.}
        \label{fig:T-U-difference}
    \end{figure}

    Clearly then $T'$ and $U'$ differ in four ways: (i) the upper path in $T'$ except the small overlap near the end, (ii) the lower path in $U'$, 
    (iii) the position of the bend in the puncture-paths, and (iv) the part past the bend of the bottom puncture-path in $T'$.  The difference 
    between $T'$ and $U'$ is exactly $2(A+B+M)+M-1$ tiles; moreover the region-of-difference forms a cycle so that there exists a 
    $(2(A+B+M)+M-1)$-cycle, $\sigma$, such that $\pi_{T'} = \sigma \pi_{U'}$.  We then have 
    \[
        \sgn \pi_{T'} = (-1)^{2(A+B+M)+M-1-1}\sgn \pi_{U'} = (-1)^M\sgn \pi_{U'}.
    \]
    That is, $\sgn \pi_{T} = (-1)^M\sgn \pi_{U}$.  Since $\sgn \lambda_T = (-1)^M \sgn \lambda_U$, the claim follows.
\end{proof}

We conclude that $Z_{a,b,c,\alpha,\beta,\gamma}$ and $N_{a,b,c,\alpha,\beta,\gamma}$ have the same determinant, up to sign.
\begin{theorem} \label{thm:det-Z-N}
    Consider the punctured hexagon $H_{a,b,c,\alpha,\beta,\gamma}$.  Then 
    \[
        |\det{Z_{a,b,c,\alpha,\beta,\gamma}}| = |\det{N_{a,b,c,\alpha,\beta,\gamma}}|.
    \]
\end{theorem}
\begin{proof}
    Combine Theorems~\ref{thm:nilp-matrix} and~\ref{thm:bip-matrix} via Proposition~\ref{pro:sign-agree}.
\end{proof}

Moreover, when the puncture is of even length, the determinant and permanent of $Z_{a,b,c,\alpha,\beta,\gamma}$ are the same.
\begin{corollary} \label{cor:det-Z-per-Z}
    Consider the punctured hexagon $H_{a,b,c,\alpha,\beta,\gamma}$.  If $M$ is even, then
    \[
        \per{Z_{a,b,c,\alpha,\beta,\gamma}} = |\det{Z_{a,b,c,\alpha,\beta,\gamma}}|.
    \]
\end{corollary}
\begin{proof}
    A simple analysis of the proof of Proposition~\ref{pro:sign-agree} implies that when $M$ is even then $\sgn{\pi_T} = \sgn{\pi_U}$
    for all tilings $T$ and $U$ of $H_{a,b,c,\alpha,\beta,\gamma}$.  Thus, the enumeration of signed lozenge tilings of 
    $H_{a,b,c,\alpha,\beta,\gamma}$ is, up to sign, the enumeration of (unsigned) lozenge tilings of $H_{a,b,c,\alpha,\beta,\gamma}$.  
    Thus, the claim follows from Proposition~\ref{pro:bip-matrix} and Theorem~\ref{thm:bip-matrix}.
\end{proof}

\begin{remark} \label{rem:det-Z-per-Z}
    We make a pair of remarks regarding the preceding corollary.
    \begin{enumerate}
        \item The corollary can be viewed as a special case of Kasteleyn's theorem on enumerating perfect matchings~\cite{Ka}.  To see this, notice that
            when $M$ is even, then all ``faces'' of the bipartite graph have size congruent to $2 \pmod{4}$. 
        \item The corollary extends~\cite[Theorem~1.2]{CGJL}, where punctured hexagons with trivial puncture (i.e., $M = 0$) are considered.
            We further note that~\cite[Section~3.4]{Ke} provides, independently, essentially the same proof as~\cite{CGJL}, and the proof of 
            Lemma~\ref{lem:perm-rotate} builds on this technique.
    \end{enumerate}
\end{remark}

We conclude this section with some observations on the signs introduced here.

Let $T$ be a lozenge tiling of $H_{a,b,c,\alpha,\beta,\gamma}$, and let $L_T$ and $P_T$ be the associated family of non-intersecting lattice
paths and perfect matching, respectively.  By Proposition~\ref{pro:sign-agree}, we may assume that $\sgn L_T = \sgn P_T$.  Thus we may assign
to $T$ the sign $\sgn T = \sgn L_T$.

Recall that there are $C$ admissible permutations $\lambda_0, \ldots, \lambda_C$ (see the discussion after Theorem~\ref{thm:nilp-matrix}) associated
to $H_{a,b,c,\alpha,\beta,\gamma}$.  Further, $\sgn{\lambda_k} = (-1)^{M(C-k)}$ and so if $M$ is even then $\sgn{\lambda_k} = 1$ for all $k$.
Hence, we need only consider $M$ odd.  In this case, $\sgn{\lambda_k} = 1$ if and only if $C-k$ is even.  Thus, the sign of $T$ is $(-1)^{C-k}$.

\begin{figure}[!ht]
    \begin{minipage}[b]{0.32\linewidth}
        \centering
        \includegraphics[scale=0.33]{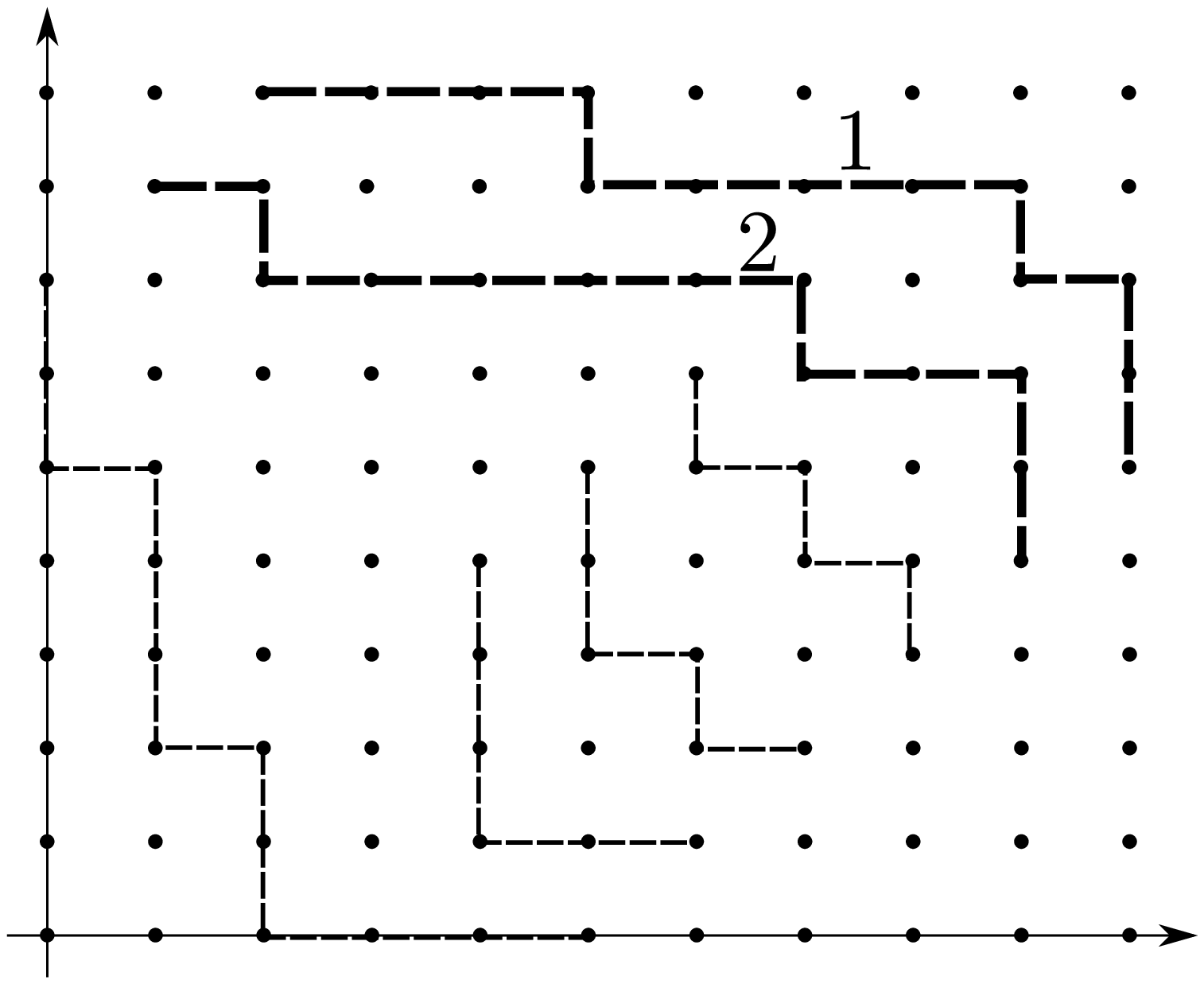}\\
        {\em (i) The sign of a family of non-intersecting lattice paths}
    \end{minipage}
    \begin{minipage}[b]{0.32\linewidth}
        \centering
        \includegraphics[scale=0.33]{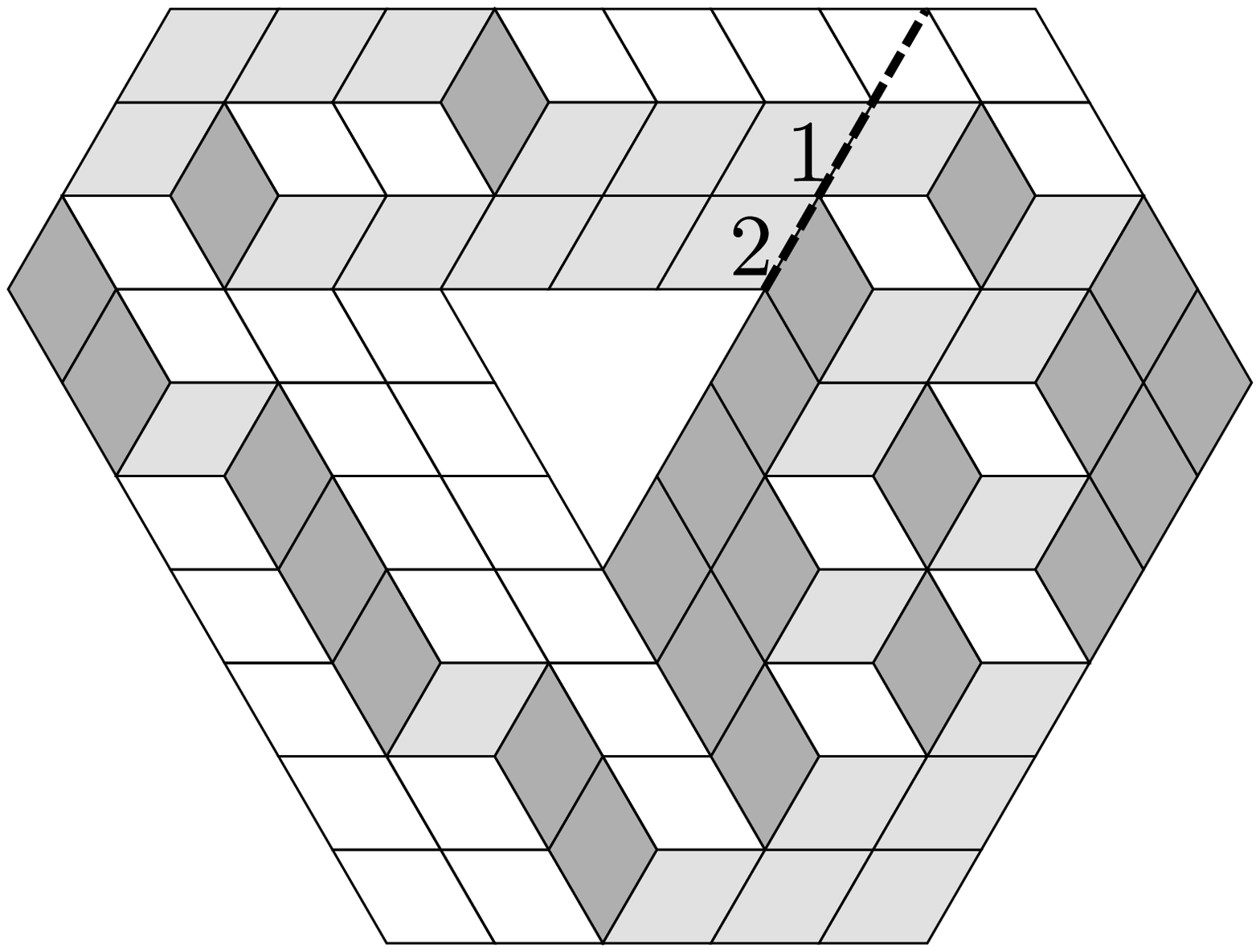}\\
        {\em (ii) The sign of a lozenge tiling}
    \end{minipage}
    \begin{minipage}[b]{0.32\linewidth}
        \centering
        \includegraphics[scale=0.33]{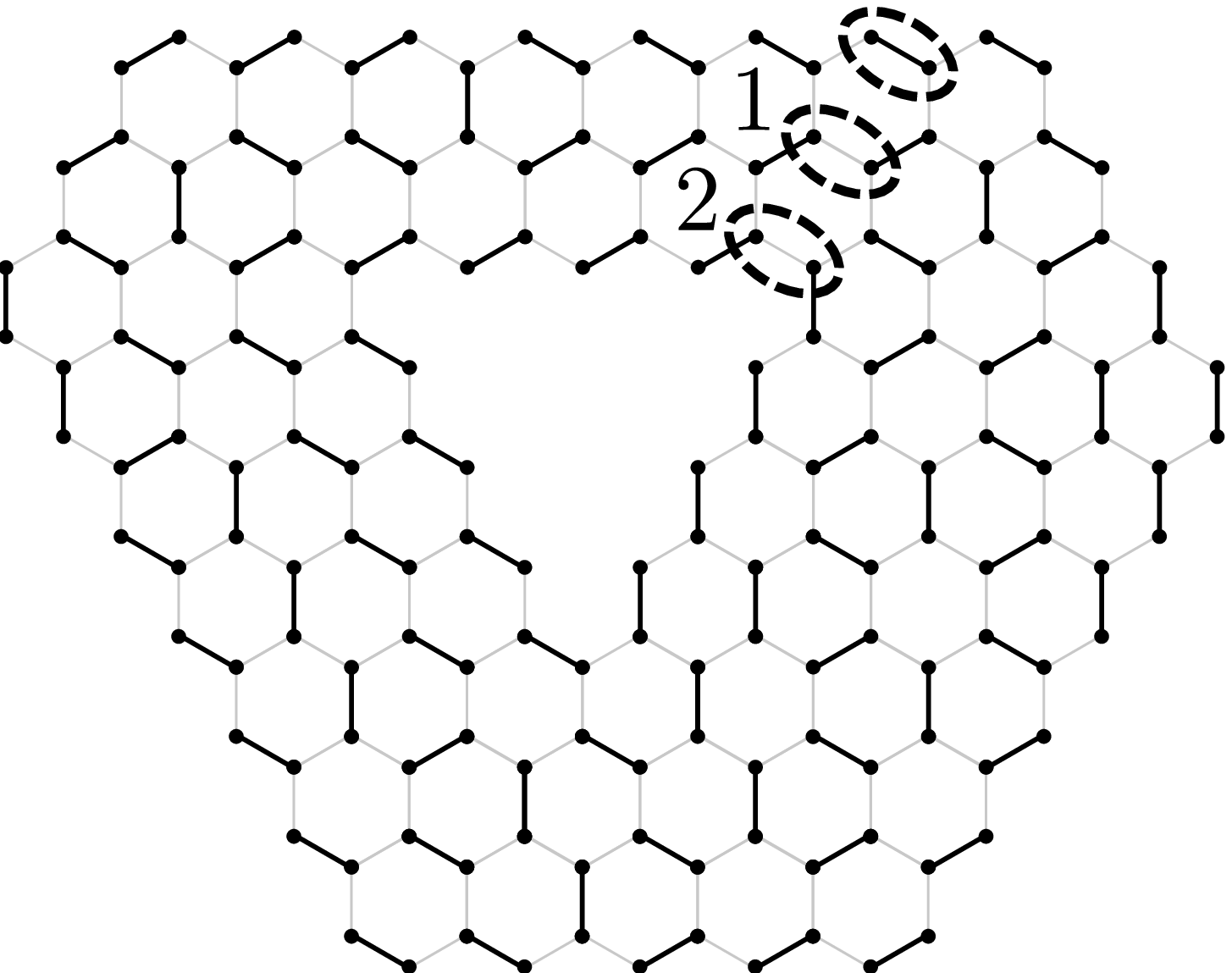}\\
        {\em (iii) The sign of a perfect matching}
    \end{minipage}
    \caption{Example of interpreting the sign}
    \label{fig:interp-signs}
\end{figure}

By definition of $\lambda_k$, $C-k$ is the number of lattice paths in the family that go above the puncture; see Figure~\ref{fig:interp-signs}(i).
For the lozenge tiling $T$, $C-k$ is the number of edges of lozenges of $T$ that touch the line formed by extending the edge of the puncture 
parallel to the side of length $C$ to the side of length $A+M$; see Figure~\ref{fig:interp-signs}(ii).  Note that this interpretation is in line with
the definition of the statistic $n(\cdot)$ in~\cite[Section~2]{CEKZ}.  Last, for the perfect matching, $C-k$ is the number of {\em non}-selected
edges that correspond to those on the edge described for lozenge tilings; see Figure~\ref{fig:interp-signs}(iii).

% -- Section
\section{Determinants} \label{sec:det}

We continue to use the notation introduced in Proposition~\ref{pro:semistable} and Theorem~\ref{thm:interlace-amaci}.
Throughout this section we assume that $A, B, C,$ and $M$ meet conditions (i)-(iv) of Proposition~\ref{pro:semistable} and 
$a+b+c+\alpha+\beta+\gamma \equiv 0 \pmod{3}$. 

We will discuss properties of the determinant of the matrix $N_{a,b,c,\alpha,\beta,\gamma}$ given in Proposition~\ref{pro:wlp-binom} using 
Theorem~\ref{thm:nilp-matrix}.  In particular, we are chiefly interested in whether the determinant is zero and if we can compute an upper
bound on the prime divisors.  In some cases we can explicitly compute the determinant.

~\subsection{A few properties}~

First, a brief remark about the polynomial nature of the determinants.
\begin{remark} \label{rem:polynomial}
    The argument in~\cite[Section~6]{CEKZ} demonstrates that for fixed $A, B,$ and $C$ and $\alpha, \beta,$ and $\gamma$ satisfying certain restraints,
    then the determinant of $N_{a,b,c,\alpha,\beta,\gamma}$ is polynomial in $M$, the side-length of the puncture of $H_{a,b,c,\alpha,\beta,\gamma}$, for
    $M$ of a fixed parity.  This argument centers around an alternate bijection between the lozenge tilings and non-intersecting lattice paths.  
    
    We note that the argument is completely independent of the restrictions on $\alpha, \beta,$ and $\gamma$.  Thus, their argument can be easily seen to
    generalise to show that, for fixed $A,B,C,\alpha,\beta,$ and $\gamma$, the determinant of $N_{a,b,c,\alpha,\beta,\gamma}$ is polynomial in $M$, for $M$
    of a fixed parity.
\end{remark}

We demonstrate that every punctured hexagonal region $H_{a,b,c,\alpha,\beta,\gamma}$ has at least one tiling.
\begin{lemma} \label{lem:tilings-exist}
    Every region $H_{a,b,c,\alpha,\beta,\gamma}$ has at least one lozenge tiling.
\end{lemma}
\begin{proof}
    In this case, it is easier to show there exists a family $L$ of non-intersecting lattice paths.  In particular, it is sufficient to show that the sum 
    of the maximum number of paths that can go above and below the puncture is at least $C$.  By analysis of $H_{a,b,c,\alpha,\beta,\gamma}$, we see that
    at most $\min\{C, \beta, B+C-\alpha\}$ paths can go below the puncture and at most $\min\{C, \alpha, A+C-\beta\}$ paths can go above the puncture.
    However, as $0 \leq A,B,C$ and $C \leq \alpha + \beta$, then $\min\{C, \beta, B+C-\alpha\} + \min\{C, \alpha, A+C-\beta\} \geq C$.
\end{proof}

Thus when $M$ is even, the determinant is always positive.
\begin{theorem} \label{thm:M-even}
    If $a+b+c$ is even, then $M$ is even and $\det{N_{a,b,c,\alpha,\beta,\gamma}} > 0$.  Thus 
    \[
        I_{a,b,c,\alpha,\beta,\gamma} = (x^a, y^b, z^c, x^\alpha y^\beta z^\gamma)
    \]
    has the weak Lefschetz property in characteristic zero and when the characteristic is sufficiently large.
\end{theorem}
\begin{proof}
    Recall the definition of the admissible partitions $\lambda_k$, for $0 \leq k \leq C$ (see the discussion following Theorem~\ref{thm:nilp-matrix}).
    Since $M$ is even, then $\sgn{\lambda_k} = 1$ for $0 \leq k \leq C$ and hence $\det{N_{a,b,c,\alpha,\beta,\gamma}}$ is the number of tilings of 
    $H_{a,b,c,\alpha,\beta,\gamma}$.  Thus, by Lemma~\ref{lem:tilings-exist}, $\det{N_{a,b,c,\alpha,\beta,\gamma}} > 0$.
\end{proof}~

\subsection{Mahonian determinants}\label{ss:macmahon}~

MacMahon computed the number of plane partitions (finite two-dimensional arrays that weakly decrease in all columns and rows) in an $A \times B \times C$ box
as (see, e.g., \cite[Page 261]{Pr})
\[
    \Mac(A,B,C) := \frac{\HF(A) \HF(B) \HF(C) \HF(A+B+C)}{\HF(A+B) \HF(A+C) \HF(B+C)},
\]
where $A,B,$ and $C$ are non-negative integers and $\HF(n) := \prod_{i=0}^{n-1}i!$ is the {\em hyperfactorial} of $n$.  David and Tomei proved in~\cite{DT} 
that plane partitions in an $A \times B \times C$ box are in bijection with lozenge tilings in an non-punctured hexagon of side-lengths $(A,B,C,A,B,C)$. 
We note that Propp states on~\cite[Page 258]{Pr} that Klarner was likely the first to have observed this.
See Figure~\ref{fig:pp-tile} for an illustration of the connection.

\begin{figure}[!ht]
    \begin{minipage}[b]{0.49\linewidth}
        \centering
        \includegraphics[scale=0.6]{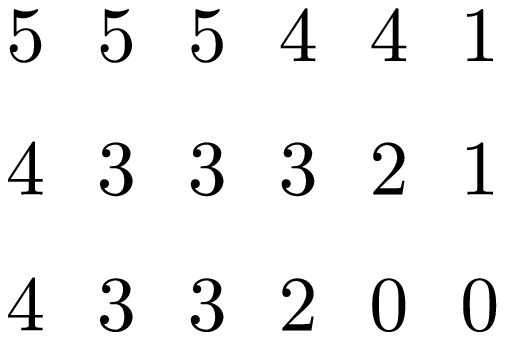}\\[1\baselineskip]~
    \end{minipage}
    \begin{minipage}[b]{0.49\linewidth}
        \centering
        \includegraphics[scale=0.3]{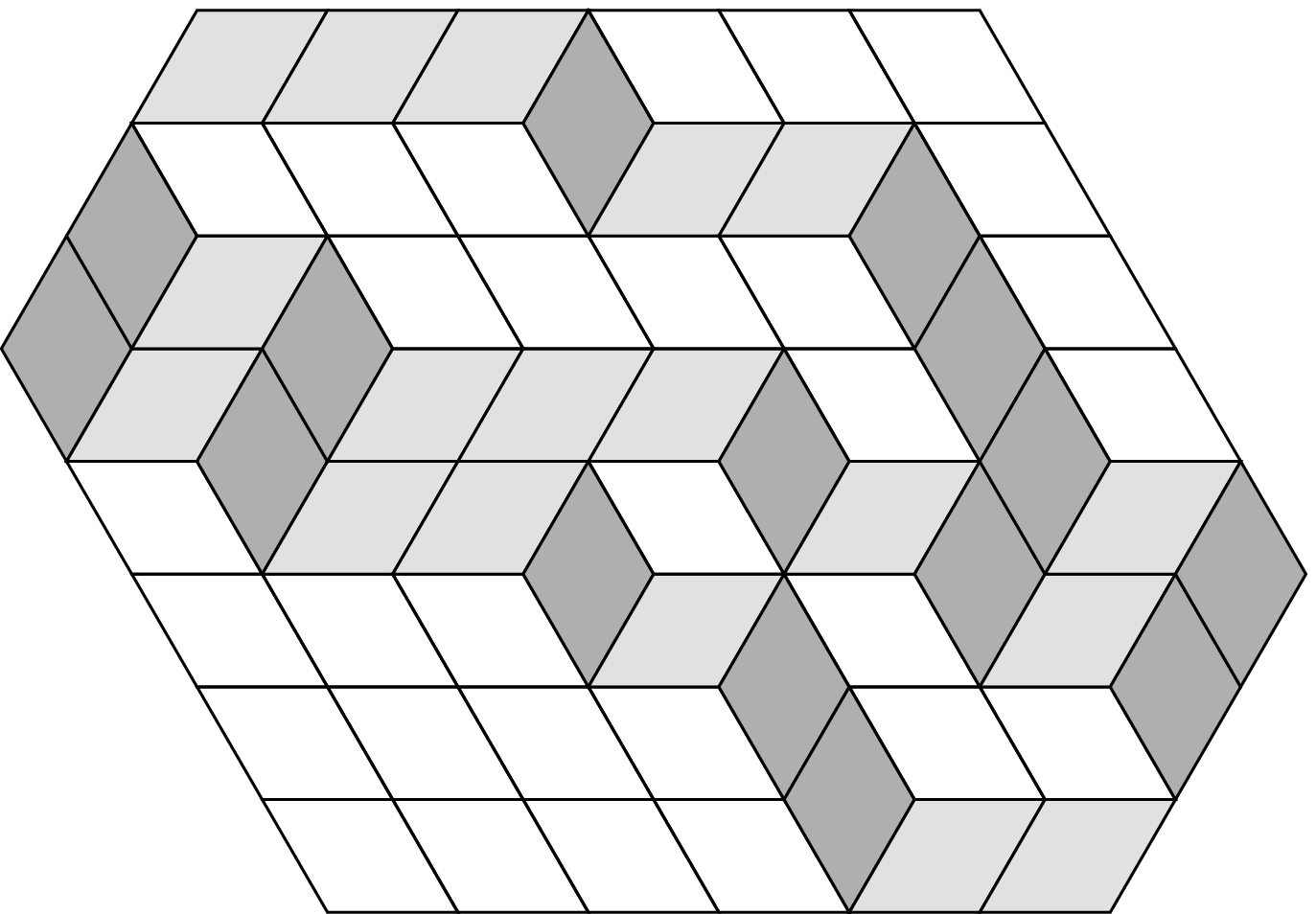}\\
    \end{minipage}
    \caption{An example of a $3 \times 6 \times 5$ plane partition and its associated lozenge tiling (with light grey as the top faces of the boxes)}
    \label{fig:pp-tile}
\end{figure}

We can use MacMahon's formula to compute the determinant of $N_{a,b,c,\alpha,\beta,\gamma}$ in many cases.  Also, note that the prime divisors 
of $\Mac(A,B,C)$ are sharply bounded above by $A+B+C-1$.
A first case is when the puncture is trivial.  This extends~\cite[Theorem~4.5]{CN} where the level algebras of this family are considered.
\begin{proposition} \label{pro:M-zero}
    If $a+b+c=2(\alpha+\beta+\gamma)$, then $M = 0$ and $\det{N_{a,b,c,\alpha,\beta,\gamma}}$ is \[\Mac(A,B,C).\]  Thus, $I_{a,b,c,\alpha,\beta,\gamma}$
    has the weak Lefschetz property if the characteristic of $K$ is zero or at least $A+B+C=\alpha+\beta+\gamma = \frac{1}{2}(a+b+c)$.
\end{proposition}
\begin{figure}[!ht]
    \includegraphics[scale=0.667]{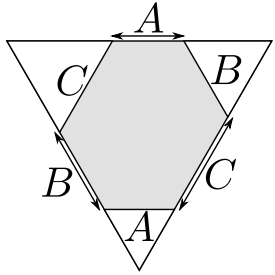}
    \caption{When the puncture has side-length zero, the region is a simple hexagon.}
    \label{fig:hex-no-puncture}
\end{figure}
\begin{proof}
    When $M = 0$ then there is no puncture in the region $H_{a,b,c,\alpha,\beta,\gamma}$.  Hence the region is a simple hexagon with side-lengths 
    $(A,B,C,A,B,C)$, exactly the region to which MacMahon's formula applies.
\end{proof}

This result allows us to recover some earlier results about complete intersections.  
\begin{corollary} \label{cor:ci}
    If $a+b+c$ is even, then the complete intersection $J = (x^a, y^b, z^c)$ has the weak Lefschetz property if and only if the characteristic of $K$ is
    not a prime divisor of $\Mac(A,B,C)$.  That is, the algebra $R/J$ has the weak Lefschetz property if the characteristic of $K$ is zero or at 
    least $A+B+C=\alpha+\beta+\gamma = \frac{1}{2}(a+b+c)$.
\end{corollary}
\begin{proof}
    Set $\alpha = \frac{1}{2}(-a+b+c),$ $\beta = \frac{1}{2}(a-b+c),$ $\gamma = \frac{1}{2}(a+b-c),$ and consider $I = (x^a, y^b, z^c, x^\alpha y^\beta z^\gamma)$.
    Then Proposition~\ref{pro:M-zero} applies to $I$ and the mixed term, $x^\alpha y^\beta z^\gamma$, has total degree $s+2$.  Thus we have that 
    $[R/I]_i \cong [R/J]_i$ for $i \leq s+1$.  That is, the twin peaks of $R/I$ are isomorphic to the twin peaks of the complete intersection $R/J$.
    Hence $R/J$ has the weak Lefschetz property if and only if $R/I$ has the weak Lefschetz property, and Proposition~\ref{pro:M-zero} gives the claim.
\end{proof}

In particular, the corollary recovers~\cite[Theorem~3.2(1)]{LZ} when combined with Proposition~\ref{pro:wlp-binom} and~\cite[Theorem~1.2]{CGJL} 
when combined with Corollary~\ref{cor:det-Z-per-Z}.  Further, the special case in~\cite[Theorem~4.2]{LZ} can be recovered if we set 
$a = \beta + \gamma, b = \alpha + \gamma,$ and $c = \alpha + \beta$. 

MacMahon's formula can be used again in another special case, when $C = 0$.  (Notice if $A$ or $B$ is zero, then we can simply relabel the sides to 
ensure $C$ is zero.) We notice this extends~\cite[Theorem~4.3]{CN} where the level algebras of this family are considered.
\begin{proposition} \label{pro:C-zero}   
    If $c = \frac{1}{2}(a+b+\alpha+\beta+\gamma)$, then $C = 0$ and $\det{N_{a,b,c,\alpha,\beta,\gamma}}$ is \[\Mac(M, A-\beta, B-\alpha).\]  Thus,
    $I_{a,b,c,\alpha,\beta,\gamma}$ has the weak Lefschetz property if the characteristic of $K$ is zero or at least $A+B+M-\alpha-\beta=c-\alpha-\beta$.
\end{proposition}
\begin{figure}[!ht]
    \includegraphics[scale=0.667]{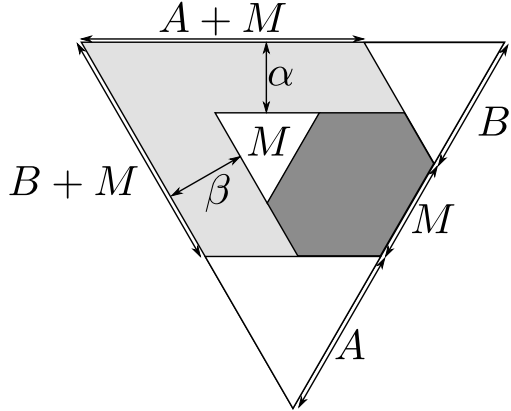}
    \caption{When $C$ is zero, the lightly shaded region has tiles that are fixed, leaving the only variation in the darkly shaded region.}
    \label{fig:hex-C-zero}
\end{figure}
\begin{proof}
    In this case, it is easier to consider families of non-intersecting lattice paths.  In particular, since $C = 0$, then the only
    starting points, the $A_i$, are those on the puncture.  Further, since lattice paths must move only right and down, then we can focus
    on the isolated region between the puncture and the bottom-right edge.  If we convert this region back into a punctured hexagon, then
    it is just a hexagon without a puncture and with side-lengths $(M, A+C-\beta, B+C-\alpha, M, A+C-\beta, B+C-\alpha)$.
\end{proof}

\begin{remark} \label{rem:C-zero}
    Notice that in the preceding proof, we show that the only possible lattice paths come from the puncture to the opposite
    edge.  Converting this back to the language of lozenge tilings, we see this means that a large region of the figure has fixed tiles leaving
    only a small region in which variation can occur.  See Figure~\ref{fig:hex-C-zero} for an illustration of this.

    Further, given the condition in Proposition~\ref{pro:C-zero}, we see that the pure power of $z$, $z^c$, has total degree $c = s+2$.  Thus,
    if we let $I = I_{a,b,c,\alpha,\beta,\gamma}$, then we have that $[R/I]_i \cong [R/J]_i$ for $i \leq s+1$, where $J = (x^a, y^b, x^\alpha y^\beta z^\gamma)$.  
    Thus, the twin peaks of $R/I$ are isomorphic to the twin peaks of the non-artinian algebra $R/J$.
\end{remark}

\begin{corollary} \label{cor:injective-pure-z-missing}
    Let $J = (x^a, y^b, x^\alpha y^\beta z^\gamma)$ and $c = \frac{1}{2}(a+b+\alpha+\beta+\gamma)$, with parameters still suitably restricted.  Then the map
    \[
        [R/J]_i \stackrel{\times (x+y+z)}{\longrightarrow} [R/J]_{i+1}
    \]
    is injective for $i \leq c$.
\end{corollary}

Further, MacMahon's formula can be used when $C$ is maximal, that is, $C = \alpha + \beta$. 
\begin{proposition} \label{pro:C-maximal}
    If $c = \frac{1}{2}(a+b+\gamma) - \alpha-\beta$, then $C = \alpha + \beta$ and $\det{N_{a,b,c,\alpha,\beta,\gamma}}$ is 
    \[\Mac(A,B,C+M).\]
    Thus, $I_{a,b,c,\alpha,\beta,\gamma}$ has the weak Lefschetz property if the characteristic of $K$ is zero or at least $A+B+C+M=s+2=c+\alpha+\beta$.
\end{proposition}
\begin{proof}
    In this case, it is easier to consider families of non-intersecting lattice paths.  In particular, since $C = \alpha + \beta$, then
    $\gamma = A+B$ and so the puncture has a point touching the side labeled $C$; see Figure~\ref{fig:hex-C-maximal}.
    \begin{figure}[!ht]
        \includegraphics[scale=0.667]{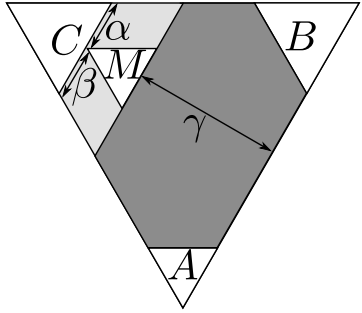}
        \caption{When $C$ is maximal, the lightly shaded region has tiles which are fixed, leaving the only variation in the darkly shaded region.}
        \label{fig:hex-C-maximal}
    \end{figure}
    Thus the lattice paths starting from $A_1, \ldots, A_\beta$ have the first $M$ moves being down and the lattice paths starting from 
    $A_{\beta+1}, \ldots, A_C$ have the first $M$ moves being right.  However, we then see that each $A_i$ ``starts'' on the same line, the line
    running through the lower-right side of the puncture.  If we convert the region-of-interest back into a punctured hexagon, then it is a simple
    hexagon with side-lengths $(A, B, C+M, A, B, C+M)$.
\end{proof}

The next case considered, when the mixed term is in two variables, needs a special determinant calculation which may be of independent interest.
\begin{lemma} \label{lem:split-binom-det}
    Let $T$ be an $n$-by-$n$ matrix defined as follows
    \[
        (T)_{i,j} = \left\{
            \begin{array}{ll}
                \displaystyle \binom{p}{q + j - i}     & \mbox{if } 1     \leq j \leq m, \\[0.8em]
                \displaystyle \binom{p}{q + r + j - i} & \mbox{if } m + 1 \leq j \leq n, \\
            \end{array}
        \right.
    \]
    where $p,q,r,$ and $m$ are non-negative integers and $1 \leq m \leq n$.  Then
    \[
        \det{T} = \Mac(m,q,r) \Mac(n-m, p-q-r, r) \frac{\HF(q+r)\HF(p-q)\HF(n+r)\HF(n+p)}{\HF(n+p-q)\HF(n+q+r)\HF(p)\HF(r)}.
    \]
\end{lemma}
\begin{proof}
    In this case, we can use~\cite[Equation~(12.5)]{CEKZ} to evaluate $\det{T}$ to be
    \[
        \prod_{1\leq i < j \leq n} (L_j - L_i) \prod_{i=1}^n \frac{(p+i-1)!}{(n+p-L_i)!(L_i-1)!},
    \]
    where $L_j = q+j$ if $1 \leq j \leq m$ and $L_j = q+r+j$ if $m + 1 \leq j \leq n$.  If we split the products in the previously displayed equation
    relative to the split in $L_j$, then we get the following equations:
    \begin{equation*}
        \begin{split}
            \prod_{1\leq i < j \leq n} (L_j - L_i) 
                = & \left(\prod_{1\leq i < j \leq m} (j - i)\right) \left(\prod_{m< i < j \leq n} (j - i)\right) \left(\prod_{1\leq i \leq m < j \leq n} (r+j-i)\right) \\[0.3em]
                = & \left(\HF(m)\right) \left(\HF(n-m)\right) \left(\frac{\HF(n+r) \HF(r)}{\HF(n+r-m) \HF(m+r)}\right)
        \end{split}
    \end{equation*}
    and 
    \begin{equation*}
        \begin{split}
            \prod_{i=1}^n \frac{(p+i-1)!}{(n+p-L_i)!(L_i-1)!}
                = & \left( \prod_{i=1}^{n}(p+i-1)! \right) \left( \prod_{i=1}^m \frac{1}{(n+p-q-i)!(q+i-1)!}\right) \\[0.3em]
                  & \left( \prod_{i=m+1}^{n} \frac{1}{(n+p-q-r-i)!(q+r+i-1)!}\right) \\[0.3em]
                = & \left( \frac{\HF(n+p)}{\HF(p)} \right) \left( \frac{\HF(n+p-m-q)\HF(q)}{\HF(n+p-q)\HF(m+q)}  \right) \\[0.3em]
                  & \left( \frac{\HF(p-q-r)\HF(m+q+r)}{\HF(n+p-m-q-r)\HF(n+q+r)}  \right).
        \end{split}
    \end{equation*}
    Bringing these equations together we have that $\det{T}$ is 
    {\footnotesize \[
        \frac{\HF(m)\HF(q)\HF(r)\HF(m+q+r)}{\HF(m+r)\HF(m+q)} \frac{\HF(n-m)\HF(p-q-r)\HF(n+p-m-q)}{\HF(n+r-m)\HF(n+p-m-q-r)} \frac{\HF(n+r)\HF(n+p)}{\HF(p)\HF(n+p-q)\HF(n+q+r)},
    \] }
    which, after minor manipulation, yields the claimed result.
\end{proof}

\begin{remark} \label{rem:split-binom-det}
    Lemma~\ref{lem:split-binom-det} generalises the result of~\cite[Lemma~2.2]{LZ} where the case $r = 1$ is discussed.  Further, when $r = 0$, 
    then $\det{T} = \Mac(n, p-q, q)$, as expected (see the running example, $\det \binom{a+b}{a-i+j}$, in~\cite{Kr}).
\end{remark}

The case when the mixed term has only two variables follows immediately.
\begin{proposition} \label{pro:gamma-zero}
    If $\gamma = 0$, then $|\det{N_{a,b,c,\alpha,\beta,\gamma}}|$ is 
    \[
        \Mac(\beta-A,A,M) \Mac(\alpha - B, B, M) \frac{\HF(A+M)\HF(B+M)\HF(C+M)\HF(A+B+C+M)}{\HF(a)\HF(b)\HF(c)\HF(M)}. 
    \]
    Thus, the type $2$ ideal 
    \[
        I_{a,b,c,\alpha,\beta,0} = (x^a, y^b, z^c, x^\alpha y^\beta)
    \]
    has the weak Lefschetz property if the characteristic of $K$ is zero or at least $A+B+C+M$.
\end{proposition}
\begin{figure}[!ht]
    \includegraphics[scale=0.667]{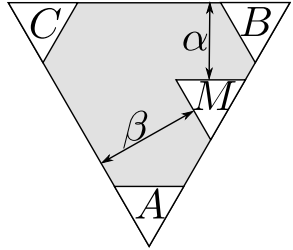}
    \caption{When $\gamma$ is zero, the starting points $A_{C+1}, \ldots, A_{C+M}$ coincide with the $M$ consecutive ending points 
        $E_{A-\beta+1}, \ldots, E_{A-\beta+M}$.}
    \label{fig:hex-gamma-zero}
\end{figure}
\begin{proof}
    As $\gamma = 0$, $N = N_{a,b,c,\alpha,\beta,\gamma}$ has entries given by
    \[
        (N)_{i,j} = \left\{
            \begin{array}{ll}
                \binom{c}{A + j - i} & \mbox{if } 1 \leq i \leq C, \\[0.3em]
                \left\{
                    \begin{array}{ll}
                        1 & \mbox{if } j = i + \beta - A - C \\
                        0 & \mbox{if } j \neq i + \beta -A - C\\
                    \end{array}
                \right\} & \mbox{if } C + 1 \leq i \leq C+M.
                \end{array}
            \right..
    \]
    Further, if we define the matrix $T$ by
    \[
        (T)_{i,j} = \left\{
            \begin{array}{ll}
                \binom{c}{A + j - i}     & \mbox{if } 1             \leq j \leq \beta - A, \\[0.3em]
                \binom{c}{A + M + j - i} & \mbox{if } \beta - A + 1 \leq j \leq C \\
            \end{array}
        \right.,
    \]
    then $|\det{N}| = |\det{T}|$ due to the structure of the lower-part of $N$.  Thus, if we let $p = c, q = A, r = M, m = \beta - A,$ and $n = C$, then by 
    Lemma~\ref{lem:split-binom-det} we have the desired determinant evaluation.

    Moreover, $\alpha + M$ and $\beta + M$ are smaller than $A+B+C+M$, so the prime divisors of $\det{N}$ are strictly bounded above by $A+B+C+M$.
\end{proof}

\begin{remark} \label{rem:gamma-zero}
    Proposition~\ref{pro:gamma-zero} deserves a pair of comments:
    \begin{enumerate}
        \item The evaluation of the determinant includes two Mahonian terms and a third non-Mahonian term.
            It should be noted that both hexagons associated to the Mahonian terms actually show up in the punctured hexagon.
            \begin{figure}[!ht]
                \includegraphics[scale=0.667]{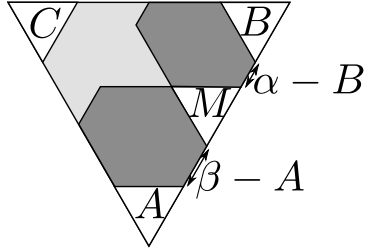}
                \caption{The darkly shaded hexagons correspond to the two Mahonian terms in the determinantal evaluation.}
                \label{fig:hex-gamma-zero-highlight}
            \end{figure}
            See Figure~\ref{fig:hex-gamma-zero-highlight} where the darkly shaded hexagons correspond to the Mahonian terms.  It is not clear (to us)
            where the third term comes from, though it may be of interest to note that if one subtracts $M$ from each hyperfactorial, before the evaluation,
            then what remains is $\Mac(A,B,C)$.  
        \item We notice the proposition also extends~\cite[Lemma~6.6]{MMN} where it was shown that the associated almost complete intersection always
            has the weak Lefschetz property in characteristic zero (i.e., the determinant is non-zero).  That is, all level type $2$ artinian monomial almost
            complete intersections in $R$ have the weak Lefschetz property in characteristic zero.
    \end{enumerate}
\end{remark}~

\subsection{Exploring symmetry}~

When $a = b$ (equivalently, $A = B$) and $\alpha = \beta$, then $H_{a,a,c,\alpha,\alpha,\gamma}$ is symmetric; see Figure~\ref{fig:hex-symmetric}.
In this case, $c$ is even exactly when $M = \frac{1}{3}(2a+c-4\alpha-2\gamma)$ is even; similarly, $\gamma$ is even exactly when
$C = \frac{1}{3}(2a - 2c + 2\alpha + \gamma)$ is even.  Moreover, $\alpha = A + \frac{1}{2}(C-\gamma)$.
\begin{figure}[!ht]
    \includegraphics[scale=0.667]{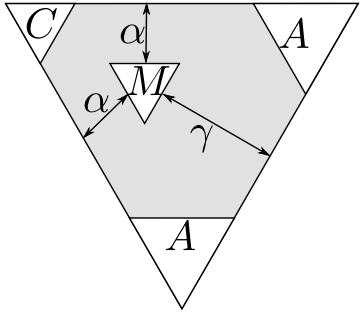}
    \caption{When $a=b$ and $\alpha = \beta$, then $H_{a,a,c,\alpha,\alpha,\gamma}$ is symmetric.}
    \label{fig:hex-symmetric}
\end{figure}

When $C$ and $M$ are odd, we can exploit symmetry to show $\det{N_{a,a,c,\alpha,\alpha,\gamma}}$ is $0$.  This result extends the evaluation 
in~\cite[Corollary~7.4]{MMN} and offers a (more) direct combinatorial proof, rather than one based on linear algebra.
\begin{proposition} \label{pro:symmetry-zero}
    If $c$ and $\gamma$ are odd, $a=b$, and $\alpha = \beta$, then $H_{a,a,c,\alpha,\alpha,\gamma}$ is symmetric with an odd puncture (i.e., $M$ odd;
    see Figure~\ref{fig:hex-symmetric}) and $\det{N_{a,a,c,\alpha,\alpha,\gamma}}$ is 0.  Thus, 
    \[  
        I_{a,a,c,\alpha,\alpha,\gamma} = (x^a, y^a, z^c, x^\alpha y^\alpha z^\gamma)
    \]
    never has the weak Lefschetz property, regardless of characteristic.
\end{proposition}
\begin{proof}
    Recall the admissible partitions of $H_{a,a,c,\alpha,\alpha,\gamma}$ are $\lambda_0, \ldots, \lambda_C$.  For $0 \leq i \leq \frac{C-1}{2}$
    we see that $P^+_{\lambda_i}(A \rightarrow E) = P^+_{\lambda_{C-i}}(A \rightarrow E)$ by symmetry, and further that 
    $\sgn{\lambda_{i}} = -\sgn{\lambda_{C-i}}$, as $\sgn{\lambda_k} = (-1)^{M(C-k)}$ and $C$ is odd.  Hence, 
    $\det{N_{a,b,c,\alpha,\beta,\gamma}} = \sum_{i = 0}^C{\sgn{\lambda_i} P^+_{\lambda_i}(A \rightarrow E)} = 0.$
\end{proof}

From the preceding proof we see that if we consider $c$ even instead of $c$ odd (i.e., $M$ even instead of $M$ odd), then 
$\det{N_{a,a,c,\alpha,\alpha,\gamma}}$ is even, when $\gamma$ is odd (i.e., $C$ is odd).

Recall the definitions of $A, B, C,$ and $M$ from Proposition~\ref{pro:semistable}, $H_{a,a,c,\alpha,\alpha,\gamma}$ from Theorem~\ref{thm:interlace-amaci},
and $N_{a,b,c,\alpha,\beta,\gamma}$ from Proposition~\ref{pro:wlp-binom}.  If $C$ or $M$ is even, then the region $H_{a,a,c,\alpha,\alpha,\gamma}$ is symmetric
and we offer the following conjecture for a closed form for $\det{N_{a,a,c,\alpha,\alpha,\gamma}}$.  Notice that in this case $\alpha = A + \frac{1}{2}(C-\gamma)$.
\begin{conjecture} \label{con:symmetry}
    Suppose $a=b$ and $\alpha = \beta$ so $H_{a,a,c,\alpha,\alpha,\gamma}$ is symmetric.  If $c$ or $\gamma$ is even, then
    $\det{N_{a,b,c,\alpha,\beta,\gamma}}$ is
    \[
        (-1)^{M\clfr{C}{2}} 
        \times \frac{ \HF(M+C) \HF(M+\gamma) \HF(M+A+\flfr{C}{2}) \HF(M + A + \clfr{C}{2}) \HF(M+2A+C) }
                    { \HF(M+2A) \HF^2(M+A+C) \HF^2(M + \frac{C+\gamma}{2}) }
    \]\vspace{0.02in}
    \[
        \times \frac{ \HF(\flfr{M}{2}) \HF(\flfr{M}{2} + A) \HF(\flfr{M}{2} + \frac{C+\gamma}{2}) \HF(\flfr{M}{2} + A + \frac{C-\gamma}{2}) }
                    { \HF(\flfr{M+C}{2}) \HF(\flfr{M+\gamma}{2}) \HF(\flfr{M+C}{2} + A) \HF(\flfr{M-\gamma}{2} + A) }
    \]\vspace{0.025in}
    \[
        \times \frac{ \HF(\clfr{M}{2}) \HF(\clfr{M}{2} + A) \HF(\clfr{M}{2} + \frac{C+\gamma}{2}) \HF(\clfr{M}{2} + A + \frac{C-\gamma}{2}) }
                    { \HF(\clfr{M+C}{2}) \HF(\clfr{M+\gamma}{2}) \HF(\clfr{M+C}{2} + A) \HF(\clfr{M-\gamma}{2} + A) }
    \]\vspace{0.025in}
    \[
        \times \frac{ \HF(A - \flfr{\gamma}{2}) \HF(\flfr{C}{2}) \HF(\flfr{\gamma}{2}) \HF(A - \clfr{\gamma}{2}) \HF(\clfr{C}{2}) \HF(\clfr{\gamma}{2}) }
                    { \HF(\gamma) \HF^2(A + \frac{C-\gamma}{2}) }.
    \]

    Further, the ideal 
    \[  
        I_{a,a,c,\alpha,\alpha,\gamma} = (x^a, y^a, z^c, x^\alpha y^\alpha z^\gamma)
    \]
    has the weak Lefschetz property when the characteristic of $K$ is zero or at least $2A+C+M$.
\end{conjecture}

\begin{remark} \label{rem:symmetry}
    The above symmetry conjecture deserves a few remarks.
    \begin{enumerate}
        \item Note that by Remark~\ref{rem:polynomial}, $\det{N_{a,b,c,\alpha,\beta,\gamma}}$ is polynomial in $M$.  Further, the conjectured form of
            the determinant would imply that the polynomial factors completely into linear terms and has degree $AC + \left\lfloor \frac{\gamma}{2}(C-\frac{\gamma}{2}) \right\rfloor$.
        \item If Conjecture~\ref{con:symmetry} were shown to hold, then it would
            complete the $(-1)$-enumeration of symmetric punctured hexagons when combined with Proposition~\ref{pro:symmetry-zero},
        \item As expected, the conjecture corresponds to Proposition~\ref{pro:C-zero} when $C = 0$, to Proposition~\ref{pro:C-maximal} when 
            $A = \frac{1}{2}\gamma$ (this implies $\alpha = \frac{1}{2}C$ and so $C = 2\alpha$, which is maximal), and to Proposition~\ref{pro:gamma-zero}
            when $\gamma = 0$.  Moreover, when $A=C=\gamma$, then $H_{a,a,c,\alpha,\alpha,\gamma}$ has an {\em axis-central} puncture (see 
            Section~\ref{sub:axis}) and the conjecture corresponds to Corollary~\ref{cor:axis-central-det-Z}.
        \item When $C$ is even and $M$ is odd, then using the $f_{a,b}(c)$ from Proposition~\ref{pro:hyper-f} and $f^e_{a,b}(c)$ and $f^o_{a,b}(c)$ from
            Corollary~\ref{cor:hyper-eo-f}, we can rewrite the monic (as a polynomial in $M$) part of the conjecture as
            \begin{center}$
                    \frac{f^o_{\frac{C+\gamma}{2}, \frac{C+\gamma}{2}}(M) 
                    \cdot f^e_{\frac{|C-\gamma|}{2},\frac{|C-\gamma|}{2}}(M+\min(C,\gamma))
                    \cdot f^e_{\min(C, \gamma), \max(C, \gamma)}(M+2A-\gamma)
                    \cdot f_{|A-\gamma|, |A-\gamma|}(M+C-\gamma + 2\min(A, \gamma))}
                    {f_{\left|\frac{1}{2}C+\gamma-A\right|, \left|\frac{1}{2}C+\gamma-A\right|}(M+\min(2A-\gamma, C+\gamma))}.
            $\end{center}
            (Carefully note that the input parameter in each of the polynomials $f$ above is odd as $M$ is odd.)
    \end{enumerate}
\end{remark}

We give an example of using the symmetry conjecture.
\begin{example} \label{exa:symmetry}
    Consider $A = B = 8$, $C = 6$, $\gamma = 2$, and $M$ even.  Then $\alpha = \beta = 10$, $a = b = 14 + M$, and $c = 16 + M$.  Moreover,
    the region $H_{14+M, 14+M, 16+M, 10, 10, 2}$ is symmetric and does not fall into the case of Remark~\ref{rem:symmetry}(iii).

    Supposing Conjecture~\ref{con:symmetry} holds, then $H_{14+M, 14+M, 16+M, 10, 10, 2}$ has a $(-1)$-enumeration of
    \[
        \frac{1}{-2^{34} 3^{16} 5^6 7^6} \times (M+1)(M+3)^3(M+4)^2(M+5)^3(M+7)
    \]
    \[
           \times (M+12)^2(M+13)^4(M+14)^6(M+15)^5(M+16)^6(M+17)^3(M+18)^4(M+19)(M+20)^2.
    \]
    Thus, $I_{14+M, 14+M, 16+M, 10, 10, 2} = (x^{14+M}, y^{14+M}, z^{16+M}, x^{10} y^{10} z^2)$ has the weak Lefschetz property when the 
    characteristic of the ground field is $0$ or at least $M+21$.
\end{example}

So far, in every case where we can bound the prime divisors of $\det{N_{a,b,c,\alpha,\beta,\gamma}}$ from above, we can do so linearly in the 
parameters (actually, always by at most $s+2$).  This may, however, not always be the case.  We provide the following example to demonstrate
that this is true, but also as a contrast to the symmetry conjecture, where some restrictions lead to a (conjectured) closed form.

\begin{example} \label{exa:non-linear-bound}
    Consider the level and type $3$ algebra given by $R/I$, where
    \[
        I_{1+t, 4+t, 7+t, 1,4,7} = (x^{1+t}, y^{4+t}, z^{7+t}, x y^4 z^7)
    \]
    and $t \geq 4$.  By Remark~\ref{rem:polynomial}, we have that $\det{N} = \det{N_{1+t, 4+t, 7+t, 1,4,7}}$ is a polynomial
    in $t$.  Hence we can use interpolation to determine the polynomial in terms of $t$; in particular, $\det{N_{1+t, 4+t, 7+t, 1,4,7}}$ is 
    \[
        \frac{4}{\HF(7)} \cdot \left\{ \begin{array}{ll}
            (t-3) (t-2) (t-1)^3 t^3 (t+1)^2 (t+2) (t+4) (t+6) (t^2 + 6t - 1) & \mbox{ if $t$ is odd;} \\[0.3em]
            (t-2)^2 (t-1)^2 t^4 (t+1)^2 (t+2) (t+5) (t+7) (t^2 + 2t - 9) & \mbox{ if $t$ is even.}
        \end{array} \right.
    \]

    In 1857, Bouniakowsky conjectured that for every irreducible polynomial $f \in \ZZ[t]$ of degree at least 2 with common divisor
    $d = \gcd\{f(i) \st i \in \ZZ\}$, there exists infinitely many integers $t$ such that $\frac{1}{d}f(t)$ is prime.  We note that the 
    weaker Fifth Hardy-Littlewood conjecture, which states that $t^2 + 1$ is prime for infinitely many positive integers $t$, is a special
    case of the Bouniakowsky conjecture.

    When $t$ is odd, the determinant has the quadratic factor $t^2 + 6t - 1$.  If we let $t = 2k+1$, then this factor becomes $2(2k^2 + 8k +3)$, which
    is an irreducible polynomial over $\ZZ[k]$ with common divisor $2$ (when $k = 4$ then the polynomial evaluates to $134 = 2\cdot 67$).  Hence the
    quadratic factor of the determinant is prime for infinitely many odd integers $t$, assuming the Bouniakowsky conjecture.  Similarly the quadratic
    factor of the determinant for $t$ even is prime for infinitely many even integers $t$, again assuming the Bouniakowsky conjecture.

    Hence, assuming the Bouniakowsky conjecture, for large enough $t$, the upper bound on the prime divisors of the determinant grows quadratically in $t$.
\end{example}

The above example falls in to the case of Proposition~\ref{pro:level-wlp}(ii)(a) or the second open case immediately 
following the proposition, depending on the parity of $t$.

% -- Section
\section{Centralising the puncture} \label{sec:central}
In this section we consider two subtlety different ways to centralise the puncture of a punctured hexagon.  The first, {\em axis-central}, forces
the puncture to be centered along each axis, individually.  The second, {\em gravity-central}, forces the puncture to be the same distance, 
simultaneously, from the three sides of the hexagon that are parallel to the puncture-sides.

Throughout this section we assume, in addition to the conditions in Proposition~\ref{pro:semistable} and $a+b+c+\alpha+\beta+\gamma \equiv 0 \pmod{3}$,
that $I_{a,b,c,\alpha,\beta,\gamma}$ has type $3$, that is, $\alpha$, $\beta$, and $\gamma$ are non-zero. 

~\subsection{Axis-central}\label{sub:axis}~

We define a punctured hexagon $H_{a,b,c,\alpha,\beta,\gamma}$ to have an {\em axis-central} puncture if the puncture is ``central'' as defined 
in~\cite[Section~1]{CEKZ}.  Specifically, for each side of the puncture, the puncture-side should be the same distance from the parallel hexagon-side
as the puncture-vertex opposite the puncture-side is from the other parallel hexagon-side; see Figure~\ref{fig:axis-central}(i).  However, when $c$ has
a different parity than both $a$ and $b$, then an adjustment has to be made; in particular, translate the puncture parallel to the hexagon-side of length $C$
one-half unit toward the side of length $A$; see Figure~\ref{fig:axis-central}(b).
\begin{figure}[!ht]
    \begin{minipage}[b]{0.45\linewidth}
        \centering
        \includegraphics[scale=0.667]{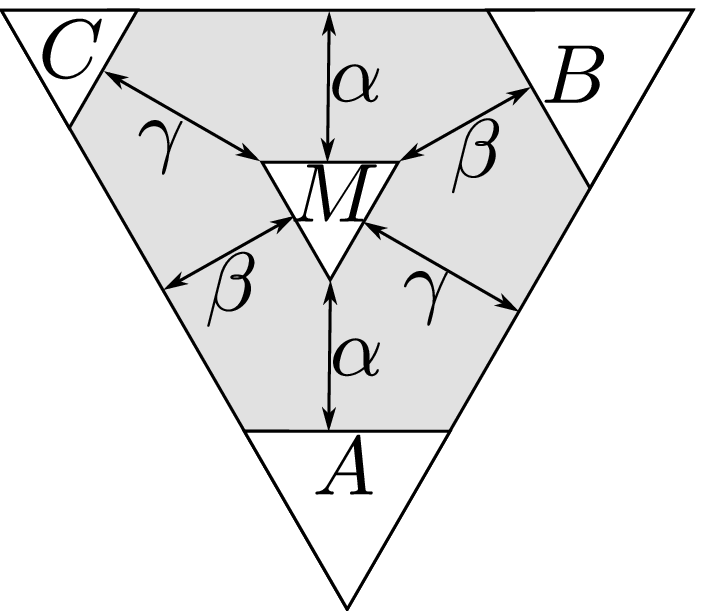}\\
        {\em (i) The parity of $c$ agrees with $a$ and $b$.}
    \end{minipage}
    \begin{minipage}[b]{0.48\linewidth}
        \centering
        \includegraphics[scale=0.667]{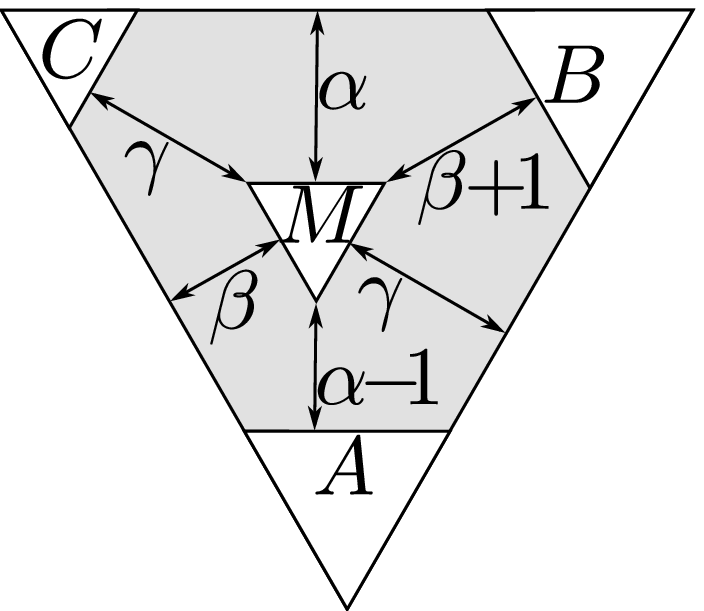}\\
        {\em (ii) The parity of $c$ differs from $a$ and $b$.}
    \end{minipage}
    \caption{A punctured hexagon with an axis-central puncture.}
    \label{fig:axis-central}
\end{figure}

When $H_{a,b,c,\alpha,\beta,\gamma}$ has an axis-central puncture, then the ideal has a nice form.  Suppose first that $a$, $b$, and $c$ have the 
same parity.  Then $\alpha = a - M - \alpha$ so $a = 2\alpha + M$; similarly, $b = 2\beta + M$ and $c = 2\gamma + M$.  Thus, if we set $t = M$, then
\[
    I_{2\alpha+t,2\beta+t, 2\gamma+t, \alpha,\beta,\gamma} = (x^{2\alpha+t}, y^{2\beta+t}, z^{2\gamma+t}, x^\alpha y^\beta z^\gamma).
\]
The conditions in Proposition~\ref{pro:semistable} simplify to $\alpha \leq \beta + \gamma$, $\beta \leq \alpha + \gamma$, $\gamma \leq \alpha + \beta$,
and $t \geq 0$.

Now, suppose the parity of $c$ differs from that of both $a$ and $b$.  Then $\alpha = a - M - \alpha + 1$, $\beta = b - M - \beta - 1$, and $\gamma = c- M - \gamma$,
so $a = 2\alpha + M-1$, $b = 2\beta +M+1$, and $c = 2\gamma +M$.  Thus, if we set $t = M$, then 
\[
    I_{2\alpha+t-1,2\beta+t+1, 2\gamma+t, \alpha,\beta,\gamma} = (x^{2\alpha+t-1}, y^{2\beta+t+1}, z^{2\gamma+t}, x^\alpha y^\beta z^\gamma).
\]
The conditions in Proposition~\ref{pro:semistable} simplify to $\alpha \leq \beta + \gamma+1$, $\beta \leq \alpha + \gamma-1$, $\gamma \leq \alpha + \beta$,
and $t \geq 0$.

Much to our fortune, the determinants of $N_{a,b,c,\alpha,\beta,\gamma}$ have been calculated for punctured hexagons with axis-central punctures.  We recall
the four theorems here, although we forgo the exact statements of the determinant evaluations; the explicit evaluations can be found in~\cite{CEKZ}.
\begin{theorem}{\cite[Theorems 1, 2, 4, \& 5]{CEKZ}} \label{thm:CEKZ-1245}
    Let $A, B, C,$ and $M$ be non-negative integers and let $H$ be the associated hexagon with an axis-central puncture.  Then
    \begin{enumerate}
        \item[(1)] The number of lozenge tilings of $H$ is $\CEKZ_1(A,B,C,M)$ if $A, B,$ and $C$ share a common parity.
        \item[(2)] The number of lozenge tilings of $H$ is $\CEKZ_2(A,B,C,M)$ if $A, B,$ and $C$ do not share a common parity.
        \item[(4)] The number of signed lozenge tilings of $H$ is
        \begin{enumerate}
            \item[(i)] $\CEKZ_4(A,B,C,M)$ if $A, B,$ and $C$ are all even, and
            \item[(ii)] $0$ if $A, B,$ and $C$ are all odd.
        \end{enumerate}
        \item[(5)] The number of signed lozenge tilings of $H$ is $\CEKZ_5(A,B,C,M)$ if $A, B,$ and $C$ do not share a common parity.
    \end{enumerate}

    Moreover, the four functions $\CEKZ_i$ are polynomials in $M$ which factor completely into linear terms.  Further, each can be expressed
    as a quotient of products of hyperfactorials and, in each case, the largest hyperfactorial term is $\HF(A+B+C+M)$.
\end{theorem}

Thus, we calculate the permanent of $Z_{a,b,c,\alpha,\beta,\gamma}$.
\begin{corollary} \label{cor:axis-central-perm-Z}
    Let $H_{a,b,c,\alpha,\beta,\gamma}$ be a hexagon with an axis-central puncture.  Then
    \[
        \per{Z_{a,b,c,\alpha,\beta,\gamma}} = \left\{ \begin{array}{ll}
            \CEKZ_1(A,B,C,M) & \mbox{if $a,b,$ and $c$ share a common parity;} \\[0.3em]
            \CEKZ_2(A,B,C,M) & \mbox{otherwise.}
        \end{array} \right.
    \]
\end{corollary}
\begin{proof}
    This follows from Proposition~\ref{pro:bip-matrix} and Theorem~\ref{thm:CEKZ-1245}.
\end{proof}

Moreover, we calculate the determinant of $N_{a,b,c,\alpha,\beta,\gamma}$, and thus can completely classify when the algebra
$R/I_{a,b,c,\alpha,\beta,\gamma}$ has the weak Lefschetz property.
\begin{corollary} \label{cor:axis-central-det-Z}
    Let $H_{a,b,c,\alpha,\beta,\gamma}$ be a hexagon with an axis-central puncture.  If $M$ is even, then
    \[  
        \det{N_{a,b,c,\alpha,\beta,\gamma}} = \left\{ \begin{array}{ll}
            \CEKZ_1(A,B,C,M) & \mbox{if $a,b,$ and $c$ share a common parity;} \\[0.3em]
            \CEKZ_2(A,B,C,M) & \mbox{otherwise.}
        \end{array} \right.
    \]
    If $M$ is odd, then
    \[  
        \det{N_{a,b,c,\alpha,\beta,\gamma}} = \left\{ \begin{array}{ll}
            \CEKZ_4(A,B,C,M) & \mbox{if $a,b,c,$ and $s+2$ share a common parity;} \\[0.3em]
            \raisebox{-0.6em}{0} & \mbox{if $a,b,$ and $c$ share a common parity}  \\[-0.25em] & \mbox{different from the parity of $s+2$;} \\[0.3em]
            \CEKZ_5(A,B,C,M) & \mbox{if $a,b,$ and $c$ do not share a common parity.}
        \end{array} \right.
    \]

    Thus, $R/I_{a,b,c,\alpha,\beta,\gamma}$ always fails to have the weak Lefschetz property if $a,b,c,$ and $M$ are odd, regardless
    of the field characteristic.  Otherwise, $R/I_{a,b,c,\alpha,\beta,\gamma}$ has the weak Lefschetz property if the field characteristic
    is zero or at least $A+B+C+M$.
\end{corollary}
\begin{proof}
    This follows from Theorem~\ref{thm:nilp-matrix} and Theorem~\ref{thm:CEKZ-1245}.
\end{proof}

As we will see in the following subsection, having a gravity-central puncture is equivalent to the associated algebra being level. 
\begin{question} \label{que:axis-central-algebraically}
    Consider the punctured hexagon $H_{a,b,c,\alpha,\beta,\gamma}$.  Is there an algebraic property $P$ of algebras such that
    $H_{a,b,c,\alpha,\beta,\gamma}$ has an axis-central puncture if and only if $R/I_{a,b,c,\alpha,\beta,\gamma}$ has property $P$?
\end{question}~

\subsection{Gravity-central}~

We define a punctured hexagon $H_{a,b,c,\alpha,\beta,\gamma}$ to have a {\em gravity-central} puncture if the vertices of the puncture
are each the same distance from the perpendicular side of the hexagon; see Figure~\ref{fig:gravity-central}.
\begin{figure}[!ht]
    \includegraphics[scale=0.667]{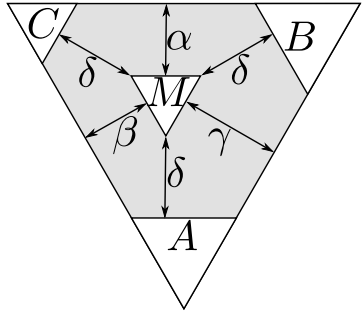}
    \caption{A punctured hexagon with a gravity-central puncture.}
    \label{fig:gravity-central}
\end{figure}
That is, we have that $B+C-\alpha = A+C -\beta = A+B - \gamma$, which simplifies to the relation $a-\alpha = b-\beta = c-\gamma$,
and this is exactly the condition in Proposition~\ref{pro:amaci-props}(ii) for $R/I_{a,b,c,\alpha,\beta,\gamma}$ to be level and type $3$.  Thus,
if we let $t$ be this common difference, then we can rewrite $I_{a,b,c,\alpha,\beta,\gamma}$ as
\[
    I_{\alpha+t,\beta+t, \gamma+t, \alpha,\beta,\gamma} = (x^{\alpha+t}, y^{\beta+t}, z^{\gamma+t}, x^\alpha y^\beta z^\gamma).
\]
Without loss of generality, assume $0 < \alpha \leq \beta \leq \gamma$.  Then the conditions in Proposition~\ref{pro:semistable} simplify to 
$t \geq \frac{1}{3}(\alpha+\beta+\gamma)$ and $\gamma \leq 2(\alpha + \beta)$.

The ideals $I_{\alpha+t,\beta+t,\gamma+t,\alpha,\beta,\gamma}$ are studied extensively in~\cite[Sections~6 \&~7]{MMN}.  In particular, \cite[Conjecture~6.8]{MMN} 
makes a guess as to when $R/I_{\alpha+t,\beta+t,\gamma+t,\alpha,\beta,\gamma}$ has the weak Lefschetz property in characteristic zero.  We recall the conjecture 
here, though we present it in a different but equivalent form.
\begin{conjecture} \label{conj:level-wlp}
    Consider the ideal $I_{\alpha+t,\beta+t, \gamma+t, \alpha,\beta,\gamma}$ in $R$ where $K$ has characteristic zero,
    $0 < \alpha \leq \beta \leq \gamma \leq 2(\alpha+\beta)$, $t \geq \frac{1}{3}(\alpha+\beta+\gamma)$, and $\alpha + \beta + \gamma$ is divisible by three.

    If $(\alpha,\beta,\gamma,t)$ is not $(2,9,13,9)$ or $(3,7,14,9)$, then $R/I_{\alpha+t,\beta+t, \gamma+t, \alpha,\beta,\gamma}$ fails to have the weak 
    Lefschetz property if and only if $t$ is even, $\alpha + \beta + \gamma$ is odd, and $\alpha = \beta$ or $\beta = \gamma$.
\end{conjecture}

\begin{remark}
    \cite[Conjecture~6.8]{MMN} is presented in a format that does not elucidate the reasoning behind it.  We present the conjecture differently so it says that 
    the weak Lefschetz property fails in two exceptional cases and also when a pair of parity conditions and a symmetry condition hold.
\end{remark}

We add further support to the conjecture.
\begin{proposition} \label{pro:level-wlp}
    Let $I = I_{\alpha+t,\beta+t, \gamma+t, \alpha,\beta,\gamma}$ be as in Conjecture~\ref{conj:level-wlp}.  Then
    \begin{enumerate}
        \item $R/I$ fails to have the weak Lefschetz property when $t$ is even, $\alpha + \beta + \gamma$ is odd, and $\alpha = \beta$ or $\beta = \gamma$;
        \item $R/I$ has the weak Lefschetz property when 
        \begin{enumerate}
            \item $t$ and $\alpha + \beta + \gamma$ have the same parity, or
            \item $t$ is odd and $\alpha = \beta = \gamma$ is even.
        \end{enumerate}
    \end{enumerate}
\end{proposition}
\begin{proof}
    Part (i) follows from Proposition~\ref{pro:symmetry-zero} (also by~\cite[Corollary~7.4]{MMN}).  
    Part (ii)(a) implies $M$ is even and so follows by Theorem~\ref{thm:M-even}.
    Part (ii)(b) follows from \cite[Theorem~4]{CEKZ}, which is recalled here in Theorem~\ref{thm:CEKZ-1245}(4)(i).
\end{proof}

We note that Conjecture~\ref{conj:level-wlp} remains open in two cases, both of which are conjectured to have the weak Lefschetz property:
\begin{enumerate}
    \item $t$ even, $\alpha + \beta + \gamma$ odd, and $\alpha < \beta < \gamma$;
    \item $t$ odd, $\alpha + \beta + \gamma$ even, and $\alpha < \beta$ or $\beta < \gamma$.
\end{enumerate}

\begin{remark} \label{rem:level-symmetry}
    Notice that the second open case in the above statement is solved for the cases when $\alpha = \beta$ or $\beta = \gamma$
    if Conjecture~\ref{con:symmetry} is true. 
\end{remark}~

\subsection{Axis- and gravity-central}~

Suppose $a,b,$ and $c$ have the same parity.  Then the punctured hexagons that are both axis- and gravity-central are precisely those
with $a = b = c = \alpha + t$ and $\alpha = \beta = \gamma$.  In this case, we strengthen~\cite[Corollary~7.6]{MMN}.
\begin{corollary} \label{cor:a-t}
    Consider the level, type $3$ algebra $A$ given by
    \[
        R/I_{\alpha + t, \alpha + t, \alpha + t, \alpha, \alpha, \alpha} = R/(x^{\alpha + t}, y^{\alpha + t}, z^{\alpha + t}, x^\alpha y^\alpha z^\alpha),
    \]
    where $t \geq \alpha$.  Then $A$ fails to have the weak Lefschetz property in characteristic zero if and only if $\alpha$ is odd and $t$ is even.
\end{corollary}
In~\cite{Kr-DPP}, Krattenthaler described a bijection between cyclically symmetric lozenge tilings of the punctured hexagon considered in the previous corollary 
and descending plane partitions with specified conditions.

If $c$ has a different parity than $a$ and $b$, then $\alpha - 1 = \beta + 1 = \gamma$.  Thus for $\alpha \geq 3$ and $M$ non-negative we have 
that the ideals of the form
\[
    I_{2\alpha + M, 2\alpha + M-2, 2\alpha + M-1, \alpha, \alpha, \alpha} = (x^{2\alpha + M}, y^{2\alpha + M-2}, z^{2\alpha + M-1}, x^\alpha y^{\alpha-2} z^{\alpha-1}),
\]
are precisely those that are both axis- and gravity-central.

% -- Section
\section{Interesting families and examples} \label{sec:interesting}
In this section, we give several interesting families and examples.  

~\subsection{Large prime divisors}~

Throughout the two preceding sections, when we could bound the prime divisors of $\det{N}$ above, we bounded them above by (at most) $s+2$. 
However, this need not always be the case, as demonstrated in Example~\ref{exa:non-linear-bound}.  We provide here a few exceptional-looking
though surprisingly common cases.

\begin{example} \label{exa:large-primes}
    Recall that $s+2 = \frac{1}{3}(a+b+c+\alpha+\beta+\gamma)$.  In each case, we specify the parameter set by a sextuple $(a,b,c,\alpha,\beta,\gamma)$.
    \begin{enumerate}
        \item Consider the parameter set $(4,6,6,1,1,3)$.  Then $s+2 = 7$ and $\det{N} = 11$.  This is the smallest $s+2$ so that
            $\det{N}$ has a prime divisor greater than $s+2$.
        \item For the parameter set $(20, 20, 20, 3, 8, 13)$, we get $s+2 = 28$ and
            \[
                \det{N}  = 2\cdot 3^{2}\cdot 5^{3}\cdot 7\cdot 11\cdot 17^{2}\cdot 19^{6}\cdot 23^{5}\cdot 20554657.
            \]
            Hence $\det{N}$ is divisible by a prime that is over $700000$ times large than $s+2$.  Moreover, $20554657$ is
            greater than the multiplicity of the associated algebra.
        \item Consider the parameter set $(7,12,13,1,7,2)$.  Then $s+2 = 14$ and $\det{N} = 13\cdot17\cdot23$.  This is the smallest
            $s+2$ so that $\det{N}$ has more than one prime divisor greater than $s+2$.
        \item Last, for the parameter set $(8,12,15,2,8,5)$, we get $s+2 =17$ and $\det{N} = 2\cdot11\cdot13^2\cdot179\cdot197$.
            In this case, notice that $\det{N}$ has two prime divisors both greater than $a+b+c+\alpha+\beta+\gamma$, the sum
            of the generating degrees of $R/I_{a,b,c,\alpha,\beta,\gamma}$.
   \end{enumerate}
\end{example}

Given the previous example and Example~\ref{exa:non-linear-bound}, it seems unlikely that there is a reasonably simple closed formula 
for the determinant of $N_{a,b,c,\alpha,\beta,\gamma}$ in general, as opposed to the case of a symmetric region (see Conjecture~\ref{con:symmetry}).

~\subsection{Fixed determinants}~

For any positive integer $n$, there is an infinite family of punctured hexagons with exactly $n$ tilings.  Note the algebras
are type $2$ if $\beta$ is zero or $ c= n+\beta+1$ and type $3$ otherwise.
\begin{proposition} \label{pro:det-n}
    Let $n$ be a positive integer.  If $\beta \geq 0$ and $c \geq n+\beta+1$, then 
    \[
        \det{N_{c-\beta-1, \beta+2, c, c-n-\beta-1, \beta, n}} = n.
    \]
    Hence the ideal
    \[
        I_{c-\beta-1, \beta+2, c, c-n-\beta-1, \beta, n} = (x^{c-\beta-1}, y^{\beta + 2}, z^c, x^{c-n-\beta-1} y^\beta z^n)
    \]
    has the weak Lefschetz property when the characteristic of $K$ is either zero or not a prime divisor of $n$.
\end{proposition}
\begin{proof}
    In this case, $s = c-2$, $A = \beta + 1$, $B = c - \beta - 2$, $C = 0$, and $M = 1$.  
    
    Using Proposition~\ref{pro:C-zero} we have that
    \[
        \det{N_{c-\beta-1, \beta+2, c, c-n-\beta-1, \beta, n}} = \Mac(M, A-\beta, B-\alpha) = \Mac(1, 1, n-1) = n.
    \]
    Alternatively, from Proposition~\ref{pro:wlp-binom} we have that
    \[
        N_{c-\beta-1, \beta+2, c, c-n-\beta-1, \beta, n} = \left( \binom{\gamma}{A+C-\beta} \right) = \left( \binom{n}{1} \right) = \left( n \right).
    \]
    Clearly then the determinant is $n$.
\end{proof}

Thus for any prime $p$, Proposition~\ref{pro:det-n} provides infinitely many monomial almost complete intersections
that fail to have the weak Lefschetz property exactly when the field characteristic is $p$.

A result of Proposition~\ref{pro:det-n} is an infinite (in fact, two dimensional) family whose members have a unique tiling.  Note that
the algebras are type $2$ if $\beta$ is zero or $c = \beta + 2$ and type $3$ otherwise.
\begin{corollary} \label{cor:C-zero-unique}
    If $\beta \geq 0$ and $c\geq \beta+2$, then $\det{N_{c-\beta-1, \beta+2, c, c-\beta-2, \beta, 1}}$ is 1.  That is, 
    \[
        I_{a,b,c,\alpha,\beta,\gamma} = (x^{c-\beta-1}, y^{\beta+2}, z^c, x^{c-\beta-2} y^\beta z)
    \]
    has the weak Lefschetz property independent of the field characteristic.
\end{corollary}

Another family whose members have a unique tiling comes from Proposition~\ref{pro:gamma-zero}.  Note that it is a three dimensional family but
also that all of the associated algebras are type $2$.
\begin{proposition} \label{pro:gamma-zero-unique}
    If $a = b = \alpha + \beta + c$ and $\gamma = 0$, then $A = B = 0$ (see Figure~\ref{fig:hex-A-B-gamma-zero}) and $\det{N_{a,b,c,\alpha,\beta,\gamma}}$ is 1. 
    That is, 
    \[
        I_{a,b,c,\alpha,\beta,\gamma} = (x^{\alpha + \beta + c}, y^{\alpha + \beta + c}, z^c, x^\alpha y^\beta)
    \] 
    has the weak Lefschetz property independent of the field characteristic.
\end{proposition}
\begin{figure}[!ht]
    \includegraphics[scale=0.667]{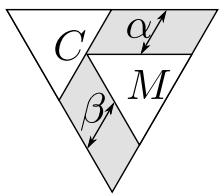}
    \caption{When $A = B = \gamma = 0$, then $H_{a,b,c,\alpha,\beta,\gamma}$ has a unique tiling.}
    \label{fig:hex-A-B-gamma-zero}
\end{figure}
\begin{proof}
    This follows from Proposition~\ref{pro:gamma-zero}.
\end{proof}

Several questions were asked in~\cite{MMN}, two of which we can answer in the affirmative.
\begin{remark} \label{rem:q-and-a}
    Question~8.2(2c) asked if there exist non-level almost complete intersections which never have the weak Lefschetz property.  The almost complete intersection
    \[ R/I_{5,5,3,2,2,1} = R/(x^5, y^5, z^3, x^2y^2z) \] is non-level and never has the weak Lefschetz property, regardless of field characteristic, as 
    $\det{N_{5,5,3,2,2,1}} = 0$ by Proposition~\ref{pro:symmetry-zero}.

    Further, we notice here that Question~8.2(2b) in~\cite{MMN} is answered in the affirmative by the comments following Question~7.12 in~\cite{MMN}.  In particular, 
    $I_{11, 18, 22, 2, 9, 13}$ is a level almost complete intersection which has odd socle degree (39) and never has the weak Lefschetz property, as
    $\det{N_{11,18,22,2,9,13}} = 0$.
\end{remark}~

\subsection{Minimal multiplicity}~

The Huneke-Srinivasan Multiplicity Conjecture, which was proven by Eisen\-bud and Schr\-ey\-er~\cite[Corollary~0.3]{ES}, shows that the multiplicity of a 
Cohen-Macaulay module gives nice bounds on the possible shifts of the Betti numbers.  Moreover, as the algebras $A$ can be viewed as finite dimensional vector 
spaces, then the multiplicity {\em is} the dimension of $A$ as a vector space.  Thus, algebras that have minimal multiplicity while retaining a particular 
property are the smallest, in the above sense, examples one can generate.

\begin{example} \label{exa:minimal-multiplicity}
    Possibly of interest are a few cases of minimal multiplicity with regard to the weak Lefschetz property.

    The following examples never have the weak Lefschetz property, that is, the determinant of their associated matrix $N_{a,b,c,\alpha,\beta,\gamma}$ is 0.
    Note that both examples are type $3$.
    \begin{enumerate}
        \item The unique level ideal with minimal multiplicity is \[ I_{3,3,3,1,1,1} = (x^3, y^3, z^3, xyz).\]  Its Hilbert function is $(1,3,6,6,3)$ and so it 
            has multiplicity $19$.  It is worth noting that this ideal is extensively studied in~\cite[Example~3.1]{BK} and is the basis for an exploration of
            the subtlety of the Lefschetz properties in~\cite{CN-2010}.
        \item The unique non-level ideal with minimal multiplicity is \[ I_{5,5,3,2,2,1} = (x^5, y^5, z^3, x^2y^2z).\]  Its Hilbert function is
            $(1,3,6,9,12,12,9,4,1)$ and so it has multiplicity $57$.  Further, this ideal is the example given in Remark~\ref{rem:q-and-a}.
    \end{enumerate}

    Moreover, the following examples always have the weak Lefschetz property, regardless of the base field characteristic.  That is to say, the determinant of their
    associated matrix $N_{a,b,c,\alpha,\beta,\gamma}$ is 1.
    \begin{enumerate}
        \item The two level ideals with minimal multiplicity are \[ I_{1,2,3,0,1,2} = (x, y^2, z^3, yz^2) \mbox{ and } I_{1,3,3,0,1,1} = (x, y^3, z^3, yz).\]  Both
            ideals have Hilbert function $(1,2,2)$ and thus multiplicity $5$.  However, both ideals are isomorphic to ideals in $K[y,z]$.
        \item The unique level, type $2$ ideal without $x$ as a generator and with minimal multiplicity is \[ I_{2,2,3,1,1,0} = (x^2, y^2, z^3, xy).\]  Its Hilbert
            function is $(1,3,3,2)$ and so it has multiplicity $9$.
        \item The unique level, type $3$ ideal with minimal multiplicity is \[ I_{3,3,6,1,1,4} = (x^3, y^3, z^6, xyz^4).\]  Its Hilbert function is
            $(1,3,6,8,9,9,7,3)$ and so it has multiplicity $46$.
        \item The unique non-level, type $2$ ideal with minimal multiplicity is \[ I_{2,2,3,0,1,1} = (x^2, y^2, z^3, yz). \]  Its Hilbert function is $(1,3,3,1)$ 
            and so it has multiplicity $8$.
        \item The unique non-level, type $3$ ideal with minimal multiplicity is \[ I_{2,2,4,1,1,2} = (x^2, y^2, z^4, xyz^2). \]  Its Hilbert function is
            $(1,3,4,4,2)$ and so it has multiplicity $14$.
    \end{enumerate}

    Notice that example (ii) and (iv) in the above enumeration differ only slightly in the mixed term yet one is level and the other is not.  It should also
    be noted that all of the above examples were found via an exhaustive search in the finite space of possible ideals using Macaulay2~\cite{M2}.
\end{example}

% -- Section
\section{Splitting type and regularity} \label{sec:splitting-type}

Throughout this section we assume $K$ is an algebraically closed field of characteristic zero.

Recall the definition of the ideals given in Section~\ref{sec:aci};  consider
\[
    I = I_{a,b,c,\alpha,\beta,\gamma} = (x^a, y^b, z^c, x^\alpha y^\beta z^\gamma),
\]
where $0 \leq \alpha < a, 0 \leq \beta < b, 0 \leq \gamma < c,$ and at most one of $\alpha, \beta,$ and $\gamma$ is zero.  In this section we consider
the splitting type of the syzygy bundles of the artinian algebras $R/I$, regardless of any extra conditions on the parameters.

Recall, also from Section~\ref{sec:aci}, that the syzygy module $\syz{I}$ of $I$ is defined by the exact sequence
\begin{equation*}
        0
    \longrightarrow
        \syz{I}
    \longrightarrow 
        R(-\alpha-\beta-\gamma) \oplus R(-a) \oplus R(-b) \oplus R(-c)
    \longrightarrow 
        I_{a,b,c,\alpha,\beta,\gamma}
    \longrightarrow 
        0
\end{equation*}
and the syzygy bundle $\widetilde{\syz{I}}$ on $\PP^2$ of $I$ is the sheafification of $\syz{I}$.  Its restriction to the line $H \cong \PP^1$ defined by
$\ell = x+y+z$ splits as $\SO_H(-p) \oplus \SO_H(-q) \oplus \SO_H(-r)$.  The arguments in~\cite[Proposition~2.2]{MMN} (recalled here in
Proposition~\ref{pro:mono}) imply that $(p, q, r)$ is the splitting type of the restriction of $\widetilde{\syz{I}}$ to a general line.  Thus,
we call $(p,q,r)$ the {\em generic splitting type} of $\syz{I}$.

In order to compute the generic splitting type of $\syz{I}$, we use the observation that $R/(I, \ell) \cong S/J$, where $S = K[x,y]$, and 
$J = (x^a, y^b, (x+y)^c, x^\alpha y^\beta (x+y)^\gamma)$.  Define $\syz{J}$ by the exact sequence
\begin{equation} \label{eqn:syz-J}
        0
    \longrightarrow
        \syz{J}
    \longrightarrow 
        S(-\alpha-\beta-\gamma) \oplus S(-a) \oplus S(-b) \oplus S(-c)
    \longrightarrow 
        J
    \longrightarrow 
        0
\end{equation}
using the possibly non-minimal set of generators $\{x^a, y^b, (x+y)^c, x^\alpha y^\beta (x+y)^\gamma\}$ of $J$.  Then 
$\syz{J} \cong S(-p) \oplus S(-q) \oplus S(-r)$.  The Castelnuovo-Mumford regularity of a homogeneous ideal $I$ is denoted
by $\reg{I}$.

\begin{remark} \label{rem:splitting-type}
    For later use, we record the following facts on the generic splitting type $(p,q,r)$ of $\syz{I_{a,b,c,\alpha,\beta,\gamma}}$.
    \begin{enumerate}
        \item As the sequence in~(\ref{eqn:syz-J}) is exact, we see that $p + q + r = a+b+c+\alpha+\beta+\gamma$.
        \item Further, if any of the generators of $J$ are extraneous, then the degree of that generator is one of $p, q,$ and $r$.
        \item As regularity can be read from the Betti numbers of $R/J$, we get that $\reg{J} + 1 = \max\{p,q,r\}$.
    \end{enumerate}
\end{remark}

Before moving on, we prove a useful lemma.
\begin{lemma} \label{lem:reg-2AMACI}
    Let $S = K[x,y]$, where $K$ is a field of characteristic zero, and let $a,b,\alpha,\beta,$ and $\gamma$ be non-negative integers with
    $\alpha + \beta + \gamma < a+b$.  Without loss of generality, assume that $0 < a-\alpha \leq b - \beta$.  Then
    $\reg{(x^a, y^b, x^\alpha y^\beta (x+y)^\gamma)}$ is
    \[
            \left\{ \begin{array}{ll}
                a + \beta + \gamma - 1 & \mbox{ if } \alpha = 0 \mbox{ and } 0 < \gamma \leq b - \beta - a; \\
                \alpha + b - 1 & \mbox{ if } 0 < \alpha, \gamma \leq b - \beta + \alpha - a, \mbox{and } 0 < \beta \mbox{ or } 0 < \gamma; and \\
                \left\lceil \frac{1}{2}(a+b+\alpha+\beta+\gamma)\right\rceil-1 & \mbox{ if } \gamma > b - \beta + \alpha - a.\\
            \end{array} \right.
    \]

    Further still, we always have $\reg{(x^a, y^b, x^\alpha y^\beta (x+y)^\gamma)} \leq \left\lceil \frac{1}{2}(a+b+\alpha+\beta+\gamma)\right\rceil-1$.
\end{lemma}
\begin{proof}
    We proceed in three steps.

    First, consider $\gamma = 0$, $0 < \alpha$, and $0 < \beta$.  Then by the form of the minimal free resolution of the quotient algebra 
    $S/(x^a, y^b, x^\alpha y^\beta)$ we have that $\reg{(x^a, y^b, x^\alpha y^\beta)} = \alpha + b - 1$.

    Second, consider $\gamma > 0$ and $\alpha = \beta = 0$.  By ~\cite[Proposition~4.4]{HMNW}, the algebra $S/(x^a, y^b)$ has the strong 
    Lefschetz property in characteristic zero.  Thus the Hilbert function of $S/(x^a, y^b, (x+y)^\gamma)$ is
        \[
            \dim_K{[S/(x^a, y^b, (x+y)^\gamma)]_j} = \max\{0, \dim_K{[S/(x^a, y^b)]_j} - \dim_K{[S/(x^a,y^b)]_{j-\gamma}}\}.
        \]  
    By analysing when the difference becomes non-positive, we get that the regularity is $a + \gamma - 1$ if $\gamma \leq b-a$ and
    $\left\lceil \frac{1}{2}(a+b+\gamma)\right\rceil-1$ if $\gamma > b-a$.

    Third, consider $\gamma > 0$ and $0 < \alpha$ or $0 < \beta$.  Notice that
        \[
            (x^a, y^b, x^\alpha y^\beta (x+y)^\gamma):x^\alpha y^\beta = (x^{a-\alpha}, y^{b-\beta}, (x+y)^\gamma).
        \]
    Considering the short exact sequence
        \[
            0 \rightarrow [S/(x^{a-\alpha}, y^{b-\beta}, (x+y)^\gamma)](-\alpha-\beta) \stackrel{\times x^\alpha y^\beta}{\longrightarrow}
                S/(x^a, y^b, x^\alpha y^\beta (x+y)^\gamma) \rightarrow S/(x^a, y^b, x^\alpha y^\beta) \rightarrow 0,
        \]
    where the first map is multiplication by $x^\alpha y^\beta$, we obtain
        \[
            \reg{(x^a, y^b, x^\alpha y^\beta (x+y)^\gamma)} = \max\{\alpha + \beta + \reg{(x^{a-\alpha}, y^{b-\beta}, (x+y)^\gamma)}, 
            \reg{(x^a, y^b, x^\alpha y^\beta)}\}.
        \]
    The claims then follows by simple case analysis.
\end{proof}

Recall that the semistability of $\syz{I_{a,b,c,\alpha,\beta,\gamma}}$ is completely determined by the parameters
$a,b,c,\alpha,\beta,\gamma$ in Proposition~\ref{pro:semistable}.

~\subsection{Non-semistable syzygy bundle}~

We first consider the case when the syzygy bundle is not semistable.  We distinguish three cases.  It turns out that in two cases, at least one of the
generators of $J$ is extraneous.
\begin{proposition} \label{pro:st-nss}
    Let $K$ be a field of characteristic zero and suppose $I = I_{a,b,c,\alpha,\beta,\gamma}$ is an ideal of $R$.  Let 
    $J = (x^a, y^b, (x+y)^c, x^\alpha y^\beta (x+y)^\gamma)$ be an ideal of $S$.  We assume,
    without loss of generality, that $a \leq b \leq c$ so that $C \leq B \leq A$.  
    \begin{enumerate}
        \item If $M < 0$, then the generator $x^\alpha y^\beta (x+y)^\gamma$ of $J$ is extraneous.  The generic splitting type of $\syz{I}$ is 
            $(a+c, b, \alpha+\beta+\gamma)$ if $c \leq b-a$ and 
            $(\left\lfloor \frac{1}{2}(a+b+c) \right\rfloor, \left\lceil \frac{1}{2}(a+b+c) \right\rceil, \alpha+\beta+\gamma)$ if $c > b-a$.
        \item If $M \geq 0$ and $C < 0$, then the generator $(x+y)^c$ of $J$ is extraneous.  The generic splitting type of $\syz{I}$ is 
            $(a+b+\alpha+\beta+\gamma - r - 1, r+1, c)$, where $r = \reg{(x^a, y^b, x^\alpha y^\beta (x+y)^\gamma)}$ (which is given in 
            Lemma~\ref{lem:reg-2AMACI}).
        \item If $M \geq 0, C \geq 0$, and $A > \beta + \gamma$, then the only destabilising sub-bundle of $\syz{I}$ is 
            $\syz{(x^a, x^\alpha y^\beta z^\gamma)}$ and so the generic splitting type of $\syz{I}$ is
            $(\left\lfloor \frac{1}{2}(\alpha+b+c) \right\rfloor,$ $\left\lceil \frac{1}{2}(\alpha+b+c) \right\rceil, a + \beta + \gamma)$.
    \end{enumerate}
\end{proposition}
\begin{proof}
    Assume $M < 0$, then $\frac{1}{2}(a+b+c) < \alpha+\beta+\gamma$ and when $c \geq a+b$ then 
    \[
        a+b-1 \leq \frac{1}{2}(a+b+c) - 1 < \alpha + \beta + \gamma.
    \]
    By Lemma~\ref{lem:reg-2AMACI} the regularity of $(x^a, y^b, (x+y)^c)$ is $a+b-1$ when $c \geq a+b$ and $\lceil \frac{1}{2}(a+b+c)\rceil - 1$ 
    otherwise; hence we have that $x^\alpha y^\beta (x+y)^\gamma$ is contained in $(x^a, y^b, (x+y)^c)$ and the first claim follows.

    Assume $M \geq 0$ and $C < 0$, then $2(\alpha+\beta+\gamma) \leq a+b+c$, $c \geq \frac{1}{2}(a+b+\alpha+\beta+\gamma)$, and 
    when $\alpha+\beta+\gamma \geq a+b$ then $2(\alpha+\beta+\gamma) \leq a+b+c$ implies $c \geq a+b$.  By Lemma~\ref{lem:reg-2AMACI},
    the regularity of $(x^a, y^b, x^\alpha y^\beta (x+y)^\gamma)$ is $a+b-1$ if $\alpha+\beta+\gamma \geq a+b$ and at most 
    $\lceil \frac{1}{2}(a+b+\alpha+\beta+\gamma) \rceil - 1$ otherwise; hence we have that $(x+y)^c$ is contained in 
    $(x^a, y^b, x^\alpha y^\beta (x+y)^\gamma)$ and the second claim follows.

    Last, assume $M \geq 0, C \geq 0$, and $A > \beta + \gamma$.  Note that since $A+B+C = \alpha+\beta+\gamma$ we then have that $B+C < \alpha$ 
    and, in particular, $B < \alpha + \gamma$ and $C < \alpha + \beta$.  Using Brenner's combinatorial criterion for the semi-stability of syzygy
    bundles of monomial ideals (see~\cite[Corollary~6.4]{Br}), we see that that $\mathcal{S} = \syz{(x^a, x^\alpha y^\beta z^\gamma)} \cong R(-r)$,
    where $r = a +\beta+\gamma$, is the only destabilising sub-bundle of $\syz{I}$.  Further, $(\syz{I})/\mathcal{S}$ is a semistable rank two vector
    bundle, so by Grauert-M\"ulich theorem, the quotient has generic splitting type $(p,q)$ where $0 \leq q-p \leq 1$.  Thus, if we consider
    the short exact sequence
    \[
        0 \longrightarrow \mathcal{S} \longrightarrow \syz{I} \longrightarrow (\syz{I})/\mathcal{S} \longrightarrow 0,
    \]
    then the third claim follows after restricting to $\ell$.
\end{proof}

In the third case, when $A > \beta + \gamma$, the associated ideal $J \subset S$ may be minimally generated by four polynomials, unlike in the other two cases.
\begin{example} \label{exa:st-nss-4mingen}
    Consider the ideals
    \[
        I_{4,5,5,3,1,1} = (x^4, y^5, z^5, x^3yz) \mbox{ and } J = (x^4, y^5, (x+y)^5, x^3y(x+y))
    \]
    in $R$ and $S$, respectively.  Notice that in this case, $0 \leq C \leq B \leq A$, $0 \leq M$, and $A > \beta + \gamma$ so
    the syzygy bundle of $R/I_{4,5,5,3,1,1}$ is non-semistable and its generic splitting type is determined in Proposition~\ref{pro:st-nss}(iii).  
    Further, $J$ is minimally generated by the four polynomials  $x^4,$ $y^5,$ $xy^3(2x+y)$, and $x^3y^2.$
\end{example}~

\subsection{Semistable syzygy bundle}~

Order the entries of the generic splitting type $(p,q,r)$ of the semistable syzygy bundle $\widetilde{\syz{I}}$ such that $p \leq q \leq r$.
Then by Grauert-M\"ulich theorem we have that $r - q$ and $q - p$ are both non-negative and at most 1.  Moreover, \cite[Theorem~2.2]{BK} specialises
in our case.
\begin{theorem} \label{thm:wlp-semistable-splitting-type}
    Let $I = I_{a,b,c,\alpha,\beta,\gamma}$.  If $R/I$ has the weak Lefschetz property, then $p = q$ or $q = r$ and $r - p \leq 1$; otherwise 
    $q = p + 1$ and $r = p + 2$.
\end{theorem}

When $a+b+c+\alpha+\beta+\gamma \not\equiv 0 \pmod{3}$, then the generic splitting type of $\syz{I}$ and regularity of $J$ can be computed easily.
\begin{proposition} \label{pro:st-nmod3}
    Let $R = K[x,y,z]$ where $K$ is a field of characteristic zero.
    Suppose $I = I_{a,b,c,\alpha,\beta,\gamma}$ is an ideal of $R$ with a semistable syzygy bundle and let $J = (x^a, y^b, (x+y)^c, x^\alpha y^\beta (x+y)^\gamma)$ be an 
    ideal of $S$.  Let $k = \left\lfloor \frac{1}{3}(a+b+c+\alpha+\beta+\gamma) \right\rfloor$.  Then $\reg{J} = k$ and the generic splitting type of
    $\syz{I}$ is 
    \[
        \left\{ \begin{array}{ll}
            (k,k,k+1) & \mbox{ if } a+b+c+\alpha+\beta+\gamma = 3k+1, \mbox{ and} \\[0.3em]
            (k,k+1,k+1) & \mbox{ if } a+b+c+\alpha+\beta+\gamma = 3k+2.
        \end{array} \right.
    \]
\end{proposition}
\begin{proof}
    Let $(p,q,r)$ be the generic splitting type of $\syz{I}$, so $a+b+c+\alpha+\beta+\gamma = 3(s+2) = p+q+r$.  By Proposition~\ref{pro:amaci-not-3}, $R/I$ has the 
    weak Lefschetz property so $p = q$, $q = r$, and $r - p \leq 1$.  Clearly if $p = q = r$ then $p+q+r=3p$ is 0 modulo 3 so cannot be
    $a+b+c+\alpha+\beta+\gamma$.  
    
    If $p = q < r$, then $r = p+1$ and $p+q+r = 3p+1$.  This matches the case when $a+b+c+\alpha+\beta+\gamma = 3k+1$, so $p = k$ and the splitting
    type of $\syz{I}$ is $(k,k,k+1)$.  Similarly, if $p < q = r$, then $q = r = p + 1$ and $p+q+r = 3p+2$.  This matches the case when 
    $a+b+c+\alpha+\beta+\gamma = 3k+2$, so $p = k$ and the splitting type of $\syz{I}$ is $(k,k+1,k+1)$.  
    
    In both cases, we have that $k-1 \leq \reg{J} \leq k$ by Remark~\ref{rem:splitting-type}(iii).  However, we see that 
    $\dim_K{[R/I]_{k-2}} < \dim_K{[R/I]_{k-1}}$ so $\dim_K{[R/(I,x+y+z)]_{k-1}} = \dim_K{[S/J]_{k-1}} > 0$ and thus $\reg{J} > k-1$.  Hence
    $\reg{J} = k$.
\end{proof}

The generic splitting type of $I_{a,b,c,\alpha,\beta,\gamma}$, when the ideal is associated to a punctured hexagon, depends on thew ideal having
the weak Lefschetz property.
\begin{proposition} \label{pro:st-mod3}
    Let $R = K[x,y,z]$ where $K$ is a field of characteristic zero.
    Suppose $I = I_{a,b,c,\alpha,\beta,\gamma}$ is an ideal of $R$ with a semistable syzygy bundle (see Proposition~\ref{pro:semistable}) and
    $a+b+c+\alpha+\beta+\gamma \equiv 0 \pmod{3}$.  Let $J = (x^a, y^b, (x+y)^c, x^\alpha y^\beta (x+y)^\gamma)$ be an ideal of $S$ and
    let $s+2 = \frac{1}{3}(a+b+c+\alpha+\beta+\gamma)$.  Then
    \begin{enumerate}
        \item If $R/I$ has the weak Lefschetz property, then the generic splitting type of $\syz{I}$ is $(s+2,s+2,s+2)$ and $\reg{J} = s+1$.
        \item If $R/I$ does not have the weak Lefschetz property, then the generic splitting type of $\syz{I}$ is $(s+1, s+2, s+3)$ and $\reg{J} = s+2$.
    \end{enumerate}
\end{proposition}
\begin{proof}
    Let $(p,q,r)$ be the generic splitting type of $\syz{I}$, so $a+b+c+\alpha+\beta+\gamma = 3(s+2) = p+q+r$.

    Assume that $R/I$ has the weak Lefschetz property.  Suppose $p \neq q$, then $q = r = p + 1$ and $p+q+r = 3p + 2$, similarly, if $q \neq r$, 
    then $p = q$ and $r = p+1$ so $p +q +r = 3p+1$; neither case is 0 modulo 3, hence cannot be $3(s+2)$.  Thus $p = q = r = s+2$.  Further
    we then see that $\reg{J} = s+1$ by Remark~\ref{rem:splitting-type}(iii).

    Now assume $R/I$ fails to have the weak Lefschetz property.  Then $p + q + r = 3p + 3 = 3(s+2)$ so $p+1 = s+2$ and $p = s+1$.  Thus, the generic
    splitting type of $\syz{I}$ must be $(s+1, s+2, s+3)$.  As $R/I$ has twin-peaks at $s+1$ and $s+2$ by Corollary~\ref{cor:one-map}, we see that 
    $\reg{J} \leq s+1$ if and only if $R/I$ has the weak Lefschetz property; so $\reg{J} \geq s+2$.  However, by Remark~\ref{rem:splitting-type}(iii)
    we have that $\reg{J} + 1 \leq s+3$ so $\reg{J} \leq s+2$, hence $\reg{J} = s+2$.
\end{proof}

This proposition can be combined with the results in the previous sections to compute the generic splitting type of many of syzygy bundles of the
artinian algebras $R/I_{a,b,c,\alpha,\beta,\gamma}$.
\begin{example} \label{exa:syzygy}
    Consider the ideal $I_{7,7,7,3,3,3} = (x^7, y^7, z^7, x^3 y^3 z^3)$ which never has the weak Lefschetz property, by Proposition~\ref{pro:symmetry-zero}.
    The generic splitting type of $\syz{I_{7,7,7,3,3,3}}$ is $(9, 10, 11)$.  Notice that the similar ideal $I_{6,7,8,3,3,3} = (x^6, y^7, z^8, x^3 y^3 z^3)$ 
    has the weak Lefschetz property in characteristic zero as $\det{N_{6,7,8,3,3,3}} = -1764$ and the generic splitting type of $\syz{I_{6,7,8,3,3,3}}$
    is $(10,10,10)$.  
\end{example}

If $I = I_{a,b,c,\alpha,\beta,\gamma}$ is not associated to a punctured hexagon, then we have seen in Proposition~\ref{pro:amaci-not-3} and
Corollary~\ref{cor:wlp-ss} that $R/I$ has the weak Lefschetz property in characteristic zero.  We summarise part of our results by pointing out
that in the case when $I$ is associated to a punctured hexagon then deciding the presence of the weak Lefschetz property is equivalent to determining
other invariants of the algebra.
\begin{theorem} \label{thm:equiv}
    Let $R = K[x,y,z]$ where $K$ is a field of arbitrary characteristic.  Let $I = I_{a,b,c,\alpha,\beta,\gamma}$ be associated to a punctured hexagon;
    in particular, $a+b+c+\alpha+\beta+\gamma \equiv 0 \pmod{3}$  and $\syz{I}$ is semistable (see Proposition~\ref{pro:semistable}).  
    Set $s = \frac{1}{3}(a+b+c+\alpha+\beta+\gamma) - 2$.  
    
    Then the following conditions are equivalent:
    \begin{enumerate}
        \item The algebra $R/I$ has the weak Lefschetz property;
        \item the regularity of $S/J$ is $s$;
        \item the determinant of $N_{a,b,c,\alpha,\beta,\gamma}$ (i.e., the enumeration of signed lozenge tilings of the punctured hexagon
            $H_{a,b,c,\alpha,\beta,\gamma}$) modulo the characteristic of $K$ is non-zero; and
        \item the determinant of $Z_{a,b,c,\alpha,\beta,\gamma}$ (i.e., the enumeration of signed perfect matchings of the bipartite graph
            associated to $H_{a,b,c,\alpha,\beta,\gamma}$) modulo the characteristic of $K$ is non-zero.
    \end{enumerate}
    Moreover, if the characteristic of $K$ is zero, then there is one further equivalent condition:
    \begin{enumerate}
        \item[(v)] The generic splitting type of $\syz{I}$ is $(s+2,s+2,s+2)$.
    \end{enumerate}
\end{theorem}
\begin{proof}
    Combine Corollary~\ref{cor:one-map}, Propositions~\ref{pro:wlp-zero-one} and~\ref{pro:wlp-binom}, Theorems~\ref{thm:nilp-matrix} 
    and~\ref{thm:bip-matrix}, and Proposition~\ref{pro:st-mod3}.
\end{proof}
This relates the weak Lefschetz property to a number of other problems in algebra, combinatorics, and algebraic geometry.

~\subsection{Jumping lines}~
Recall that a {\em jumping line} is a line, $L = 0$, over which the syzygy bundle splits differently than in the generic case, $x+y+z = 0$.  Since 
$I = I_{a,b,c,\alpha,\beta,\gamma}$ is a monomial ideal it is sufficient to consider the two cases $z = 0$ and $y+z=0$.

\begin{proposition} \label{pro:jumping-lines}
    Let $R = K[x,y,z]$ where $K$ is a field of characteristic zero and let $I = I_{a,b,c,\alpha,\beta,\gamma}$ be an ideal of $R$.
    The splitting type of $\syz{I}$ on the line $z = 0$ is $(c, \alpha + b, a + \beta)$ if $\gamma = 0$
    and $(c, \alpha + \beta + \gamma, a+b)$ if $\gamma > 0$.  And the splitting type of $\syz{I}$ on the line $y+z = 0$ is 
    $(c, a + \beta + \gamma, \alpha + b)$ if $\beta + \gamma < b \leq c$ and $(c, \alpha + \beta + \gamma, a + b)$ if $b \leq \min\{c, \beta + \gamma\}$.
\end{proposition}
\begin{proof}
    All four cases follow immediately by analysing the monomial algebra $S/J$ isomorphic to $R/(I, L)$, where $L = 0$ is the splitting line, and using
    Lemma~\ref{lem:reg-2AMACI} to compute the regularities.
\end{proof}

% -- Appendix
\appendix
\section{Hyperfactorial calculus} \label{sec:hyper-calculus}

Throughout this manuscript, the {\em hyperfactorial} function $\HF$ on non-negative integers $n$, defined by
\[
    \HF(n) := \prod_{i=0}^{n-1}i!,
\]
has been a key ingredient in many of the formulae.  In this appendix we will highlight the uses and structure of the hyperfactorial and describe a
useful ``picture-calculus'' approach to working with hyperfactorials.

Notice that, for $n \geq 0$, $\HF(n)$ can also be seen as $\prod_{k=1}^{n-1} k^{n-k}.$  Thus if we place the numbers $1$ to $n-1$ in a right-triangular
grid with legs of length $n-1$ (see Figure~\ref{fig:hyper-tri-6}), then the hyperfactorial of $n$ is the product of all the numbers in the grid.
\begin{figure}[!ht]
    \includegraphics[scale=1.0]{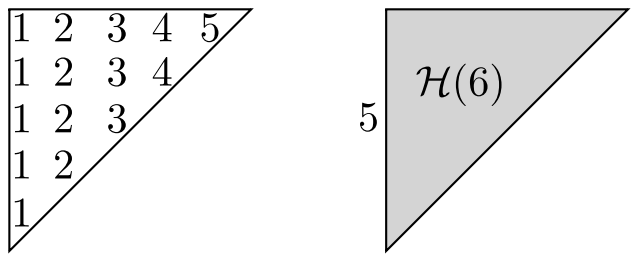}
    \caption{$\HF(6) = 34560$ represented as a triangular grid, both specifically and as a more generic shape}
    \label{fig:hyper-tri-6}
\end{figure}

Using this pictorial representation of the hyperfactorial, various identities involving hyperfactorials become more transparent.
The first identity is simple, but very useful.

\begin{proposition} \label{pro:hyper-f}
    Let $a$ and $b$ be non-negative integers with $a \leq b$.  Then the polynomial $f_{a,b}(c) \in \ZZ[c]$, given by
    \[
         f_{a,b}(c) := \prod_{i=1}^{a}(c+i)^i \prod_{i=1}^{b-a}(c+a+i)^a \prod_{i=1}^{a}(c+b+i)^{a-i},
    \]
    is equal to 
    \[
        \frac{\HF(a+b+c) \HF(c)}{\HF(a+c) \HF(b+c)},
    \]
    for positive integers $c$.
\end{proposition}
\begin{proof}
    We proceed with a proof by picture-calculus; note that in each case we choose the $-1$ inherent to the hyperfactorial to go with the 
    term which contains $c$ and that we represent the numbers that are present by a grey shaded region.  

    In Figure~\ref{fig:hyper-f}(i), we consider $\HF(a+b+c)$ divided by $\HF(b+c)$; note that we align the triangles at their bottom points.
    We then multiply by $\HF(c)$, seen in Figure~\ref{fig:hyper-f}(ii); note that we align the top edge of the new triangle with the top
    edge of the triangle associated to $\HF(b+c)$.  Last, we divide by $\HF(a+c)$ creating a parallelogram, seen in Figure~\ref{fig:hyper-f}(iii).
    \begin{figure}[!ht]
        \begin{minipage}[b]{0.32\linewidth}
            \centering
            \includegraphics[scale=0.5]{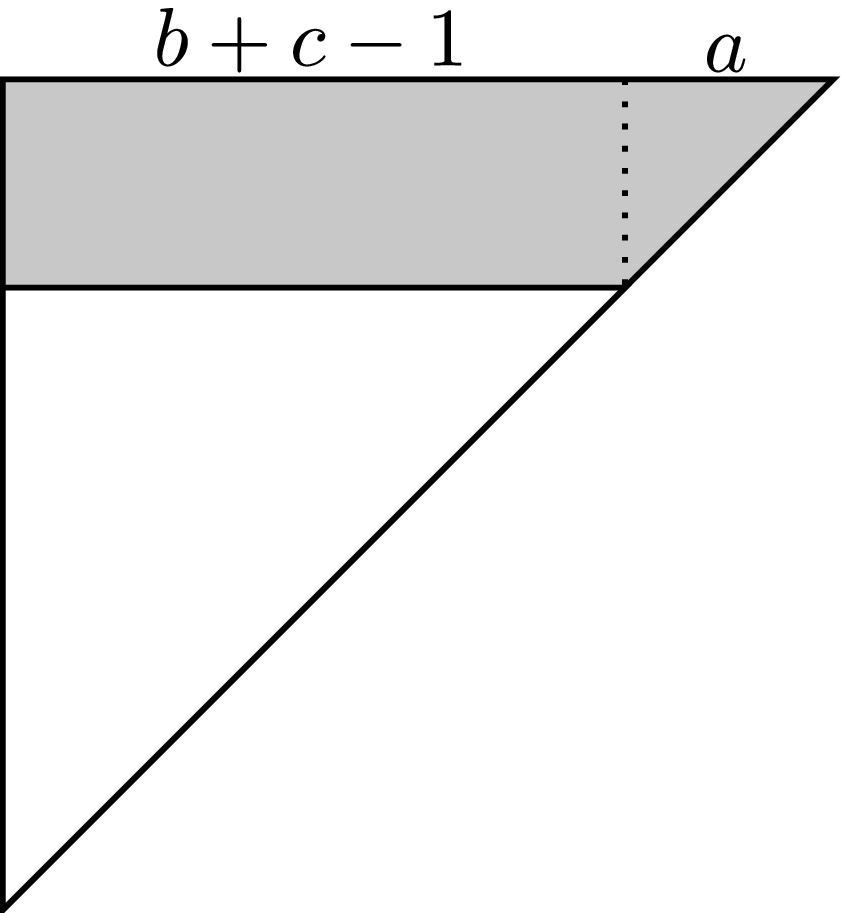}\\
            {\em (i) $\HF(a+b+c) / \HF(b+c)$}
        \end{minipage}
        \begin{minipage}[b]{0.32\linewidth}
            \centering
            \includegraphics[scale=0.5]{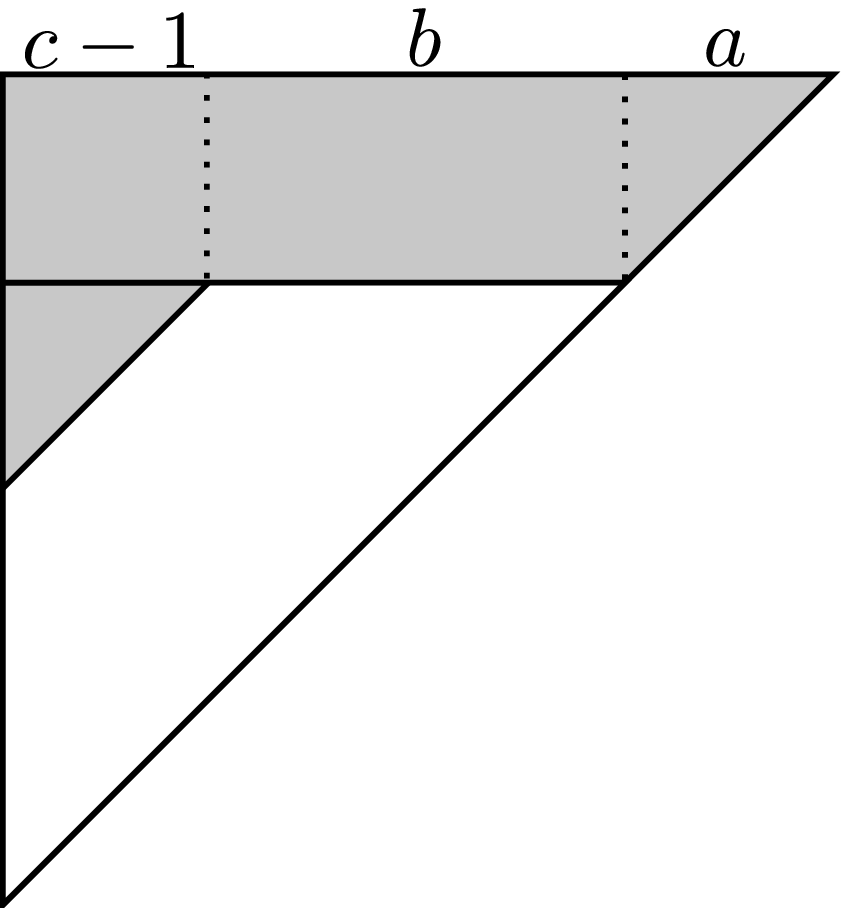}\\
            {\em (ii) Multiply by $\HF(c)$}
        \end{minipage}
        \begin{minipage}[b]{0.32\linewidth}
            \centering
            \includegraphics[scale=0.5]{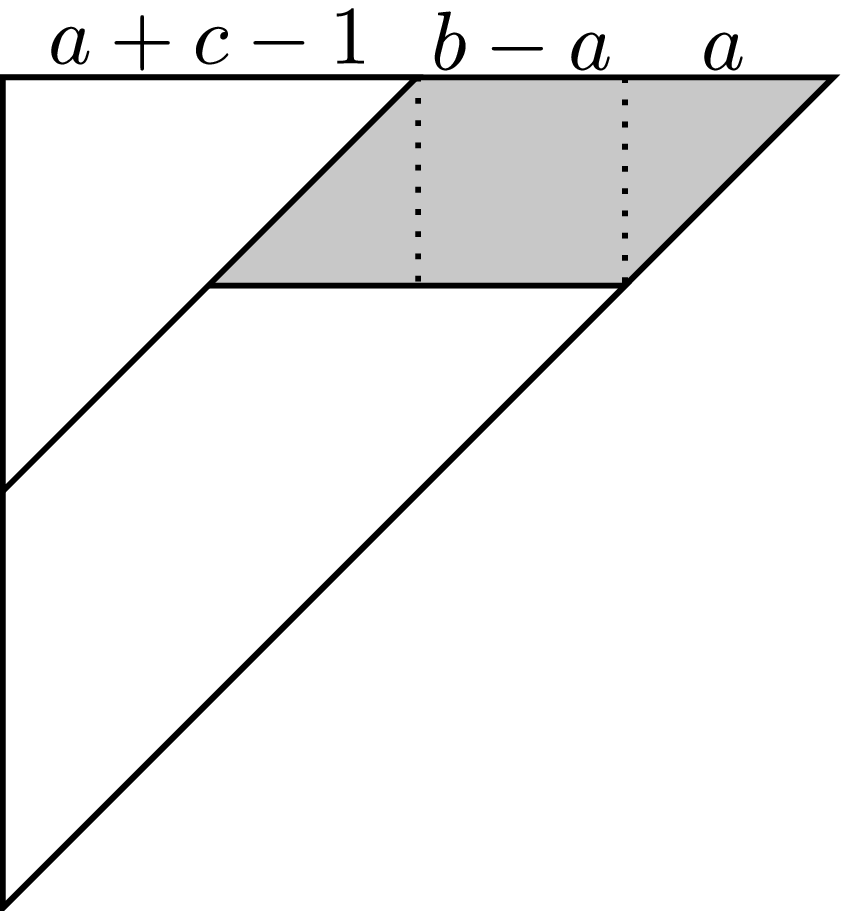}\\
            {\em (iii) Divide by $\HF(a+c)$}
        \end{minipage}
        \caption{A picture-calculus proof that $f_{a,b}(c) = \frac{\HF(a+b+c) \HF(c)}{\HF(a+c) \HF(b+c)}$}
        \label{fig:hyper-f}
    \end{figure}
    
    Notice that the parallelogram is $a$ units tall, $b$ units long, and is shifted to be $c$ units from the left edge.  Further,
    \begin{enumerate}
        \item The left grey triangular region corresponds to $\prod_{i=1}^{a}(c+i)^i$;
        \item The central grey rectangular region corresponds to $\prod_{i=1}^{b-a}(c+a+i)^a$; and
        \item The right grey triangular region corresponds to $\prod_{i=1}^{a}(c+b+i)^{a-i}$.
    \end{enumerate}
    Thus, this region is exactly the polynomial $f_{a,b}(c)$ evaluated at the integer $c$.
\end{proof}

\begin{example} \label{exa:hyper-f}
    For example, notice that
    \[
        f_{3,3}(c) = (c+1)(c+2)^2(c+3)^3(c+4)^2(c+5)
    \]
    and
    \[
        f_{3,5}(c) = (c+1)(c+2)^2(c+3)^3(c+4)^3(c+5)^3(c+6)^2(c+7).
    \]
\end{example}

A key example of using Proposition~\ref{pro:hyper-f} is with MacMahon's formula (see, e.g., \cite[Equation~(1.1)]{CEKZ}).  MacMahon's
formula for the number of lozenge tilings of a hexagon with side-lengths $(a,b,c,a,b,c)$, where $a, b,$ and $c$ are positive integers,
is given by
\[
    \Mac(a,b,c) = \frac{\HF(a) \HF(b) \HF(c) \HF(a+b+c)}{\HF(a+b) \HF(a+c) \HF(b+c)},
\]
Thus, for fixed $a$ and $b$, MacMahon's formula is a polynomial in $c$.
\begin{corollary} \label{cor:mac-poly}
    Let $a$ and $b$ be non-negative integers with $a \leq b$.  Then $\Mac(a,b,c)$ is equal to a polynomial in $c$, when evaluated at positive 
    integers; in particular,
    \[
        \Mac(a,b,c) = \frac{\HF(a) \HF(b)}{\HF(a+b)} f_{a,b}(c)
    \]
    for positive integers $c$.
\end{corollary}
\begin{proof}
    This follows immediately from Proposition~\ref{pro:hyper-f} after noticing that $\frac{\HF(a) \HF(b)}{\HF(a+b)}$ is independent of $c$.
\end{proof}

%Another simple, though useful, identity is a particular product of the polynomials in Proposition~\ref{pro:hyper-f}.
%\begin{corollary} \label{cor:hyper-g}
%    Let $a$ and $b$ be non-negative integers with $a \leq b$.  Set $\underline{a} = \left\lfloor \frac{1}{2}a \right\rfloor$ and $\overline{a} = \left\lceil \frac{1}{2}a \right\rceil$.
%    Then the polynomial ZZ[c]$, given by
%    \[
%         g_{a,b}(c) := \prod_{i=1}^{\underline{a}}(c+i)^{2i} \prod_{i=1}^{b-\underline{a}}(c+\underline{a}+i)^a \prod_{i=1}^{a}(c+b+i)^{a-i},
%    \]
%    is equal to
%    \[
%        f_{a,b}(c) f_{\underline{a}, \overline{a}}(c)
%    \]
%    and hence, for positive integers $c$, to 
%    \[
%        \frac{\HF(a+b+c) \HF^2(c)}{\HF(\underline{a}+c) \HF(\overline{a}+c) \HF(b+c)}.
%    \]
%\end{corollary}
%\begin{proof}
%    In Figure~\ref{fig:hyper-g}, the light grey region corresponds to $f_{a,b}(c)$ and the dark grey region corresponds to $f_{\underline{a}, \overline{a}}(c)$.
%    \begin{figure}[!ht]
%        \includegraphics[scale=0.8]{hyper-g}
%        \caption{A picture-calculus proof that $g_{a,b}(c) = f_{a,b}(c) f_{\underline{a}, \overline{a}}(c)$}
%        \label{fig:hyper-g}
%    \end{figure}
%\end{proof}
%
%\begin{example} \label{exa:hyper-g}
%    For example, notice that
%    \[
%        g_{4,5}(c) = (c+1)^2(c+2)^4(c+3)^4(c+4)^4(c+5)^4(c+6)^3(c+7)^2(c+8).
%    \]
%\end{example}

When considering polynomials such as $f_{a,b}(c)$, we may want only the terms where the factors are all of the form $(c+i)$, where $i$ has
a fixed parity.  To do this, we define the {\em even part of the hyperfactorial} of $n$, a positive integer, to be the even terms in the product $\HF(n)$, that is
\[
    \HF_e(n) := \prod_{i=0}^{n-1} \prod_{j=1}^{\left \lfloor \frac{1}{2}i \right \rfloor} 2j,
\]
and we define the {\em odd part of the hyperfactorial} of $n$ to be 
\[
    \HF_o(n) := \frac{\HF(n)}{\HF_e(n)}.
\]

We notice though, that $\HF_e(n)$ can be written in terms of hyperfactorials, after an appropriate scaling.
\begin{lemma} \label{lem:hyper-e}
    For positive integers $n$, the even part of the hyperfactorial of $n$, $\HF_e(n)$, is
    \[ 
        2^{\binom{\left \lfloor \frac{1}{2}n \right \rfloor}{2} + \binom{\left \lceil \frac{1}{2}n \right \rceil}{2}} 
              \HF\left(\left \lfloor \frac{1}{2}n \right \rfloor\right)  \HF\left(\left \lceil \frac{1}{2}n \right \rceil \right).
    \]
\end{lemma}
\begin{proof}
    By definition, $\HF_e(n)$ is the product of the even columns of the pictorial representation of $\HF(n)$.  In Figure~\ref{fig:hyper-e}(i), we see the product
    of $\HF\left(\left \lfloor \frac{1}{2}n \right \rfloor\right)$ and $\HF\left(\left \lceil \frac{1}{2}n \right \rceil \right)$ and in part (ii) we see this 
    same product after simplification, that is, condensing the columns to be contiguous.  In Figure~\ref{fig:hyper-e}(iii), we multiply each element of the triangular
    representation by $2$; since there are $\binom{\left \lfloor \frac{1}{2}n \right \rfloor}{2} + \binom{\left \lceil \frac{1}{2}n \right \rceil}{2}$ terms, then we
    are simply scaling $\HF\left(\left \lfloor \frac{1}{2}n \right \rfloor\right)  \HF\left(\left \lceil \frac{1}{2}n \right \rceil \right)$
    by $2^{\binom{\left \lfloor \frac{1}{2}n \right \rfloor}{2} + \binom{\left \lceil \frac{1}{2}n \right \rceil}{2}}$.  This is exactly the even 
    columns of $\HF(n)$.
    \begin{figure}[!ht]
        \begin{minipage}[b]{0.32\linewidth}
            \centering
            \includegraphics[scale=0.5]{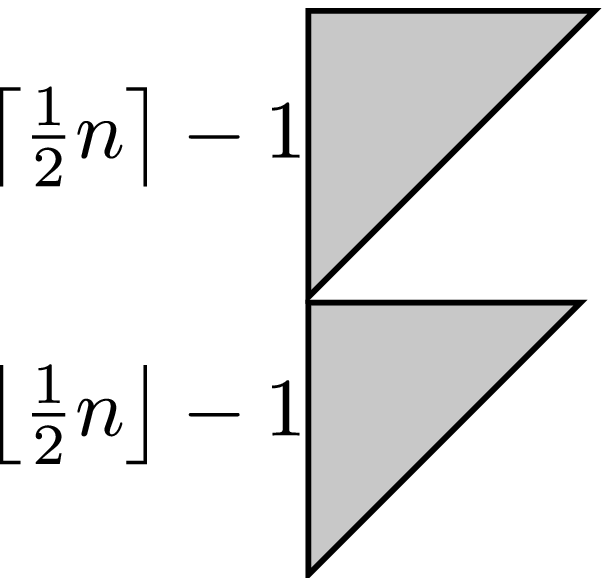}\\
            {\em (i) $\HF\left(\left \lfloor \frac{1}{2}n \right \rfloor\right)  \HF\left(\left \lceil \frac{1}{2}n \right \rceil \right)$}
        \end{minipage}
        \begin{minipage}[b]{0.32\linewidth}
            \centering
            \includegraphics[scale=0.5]{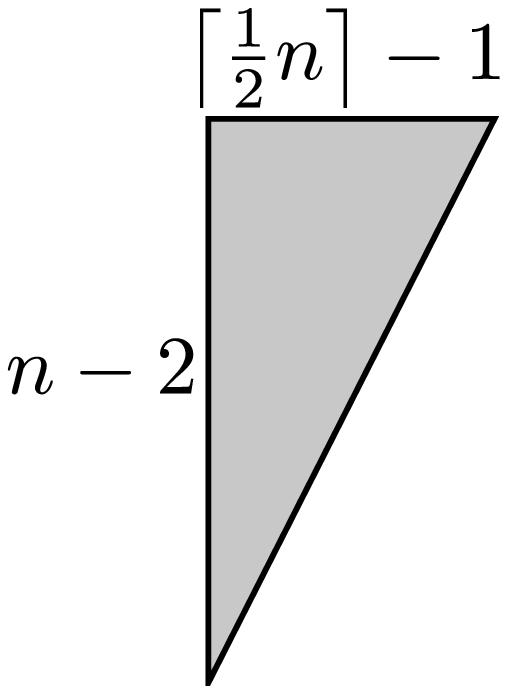}\\
            {\em (ii) After simplification}
        \end{minipage}
        \begin{minipage}[b]{0.32\linewidth}
            \centering
            \includegraphics[scale=0.5]{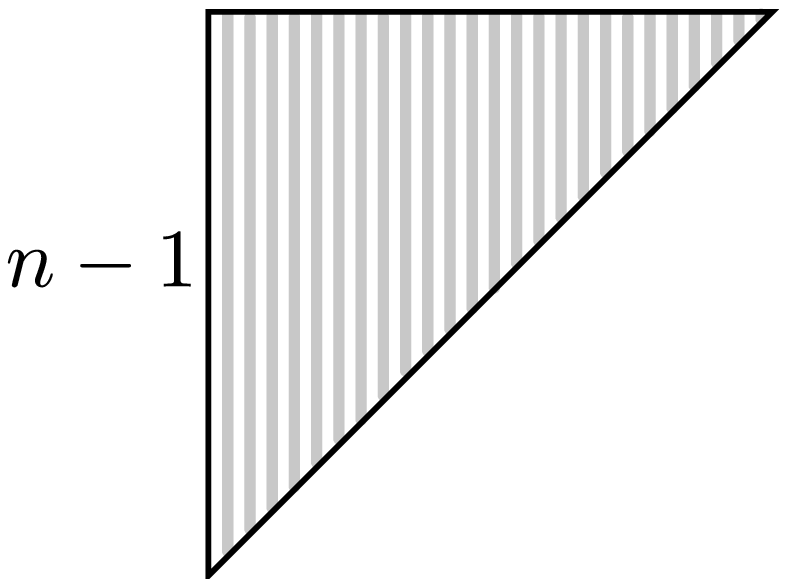}\\
            {\em (iii) After scaling}
        \end{minipage}
        \caption{A picture-calculus proof of an identity of $\HF_e(n)$}
        \label{fig:hyper-e}
    \end{figure}
\end{proof}

Now we can find the polynomials which represent $f_{a,b}(c)$ with only factors of the form $(c+i)$, where $i$ has a fixed parity, present.
\begin{corollary} \label{cor:hyper-eo-f}
    Let $a$ and $b$ be non-negative integers with $a \leq b$.  Set $\underline{a} = \left\lfloor \frac{1}{2}a \right\rfloor$ and $\underline{b} = \left\lfloor \frac{1}{2}b \right\rfloor$.
    Then the polynomial $f^e_{a,b}(c) \in \ZZ[c]$, given by
    \[
         f^e_{a,b}(c) := \prod_{i=1}^{\underline{a}}(c+2i)^{2i}
                         \prod_{i=1}^{\underline{b}-\underline{a}} (c+2\underline{a}+2i)^a 
                         \prod_{i=1}^{\underline{a}}(c+2\underline{b}+2i)^{b-2\underline{b}+a-2i},
    \]
    is equal to 
    \[
        \left\{ \begin{array}{ll}
            \displaystyle \frac{\HF_e(a+b+c) \HF_e(c)}{\HF_e(a+c) \HF_e(b+c)} & \mbox{ if $c$ is even and} \\[0.8em]
            \displaystyle \frac{\HF_o(a+b+c) \HF_o(c)}{\HF_o(a+c) \HF_o(b+c)} & \mbox{ if $c$ is odd},
        \end{array} \right.
    \]
    for positive integers $c$.

    Further, the polynomial $f^o_{a,b}(c) \in \ZZ[c]$, given by 
    \[
        f^o_{a,b}(c) := \frac{f_{a,b}(c)}{f^e_{a,b}(c)}
    \]
    is equal to 
    \[
        \left\{ \begin{array}{ll}
            \displaystyle \frac{\HF_o(a+b+c) \HF_o(c)}{\HF_o(a+c) \HF_o(b+c)} & \mbox{ if $c$ is even and} \\[0.8em]
            \displaystyle \frac{\HF_e(a+b+c) \HF_e(c)}{\HF_e(a+c) \HF_e(b+c)} & \mbox{ if $c$ is odd},
        \end{array} \right.
    \]
    for positive integers $c$.
\end{corollary}
\begin{proof}
    Notice first that $f^e_{a,b}(c)$ is defined to be the factors of $f_{a,b}(c)$ of the form $(c+i)$, where $i$ is even.  The hyperfactorial representation then 
    follows when considering $(c+i)$ would then be even exactly when $c$ is even.  

    The second claim follows similarly as the first.
\end{proof}

\begin{example} \label{exa:hyper-feo}
    For example, notice that
    \[
        f^e_{3,5}(c) = (c+2)^2(c+4)^3(c+6)^2
    \]
    and
    \[
        f^o_{3,5}(c) = (c+1)(c+3)^3(c+5)^3(c+7).
    \]
\end{example}

% -- Acknowledgement
\begin{acknowledgement}
    The authors would like to acknowledge the invaluable nature of the computer algebra system Macaulay2~\cite{M2}.
    Macaulay2 was used extensively throughout both the original research and the writing of this manuscript.
\end{acknowledgement}

%--------------------
% -- The Bibliography

%---------------------------
% -- Happy, Happy, Joy, Joy!
\end{document}